\documentclass[11pt]{article}
\usepackage[english]{babel}

\usepackage[letterpaper,top=2cm,bottom=2cm,left=3cm,right=3cm,marginparwidth=1.75cm]{geometry}

\usepackage[titletoc,title]{appendix}
\usepackage{amsmath,amssymb,mathrsfs,amsthm}
\usepackage{enumerate}
\usepackage{cases}
\usepackage{graphicx}
\usepackage[colorlinks=true, allcolors=blue]{hyperref}
\usepackage{color}
\usepackage{titlesec}
\newcommand{\vep}{\varepsilon}
\newcommand{\R}{\mathbb R}
\newcommand{\CC}{\mathbb C}
\definecolor{HW}{rgb}{0,0,0}
\definecolor{HW1}{rgb}{0,0,0}

\numberwithin{equation}{section}
\numberwithin{figure}{section}
\numberwithin{table}{section}

\newtheorem{theorem}{Theorem}[section]
\newtheorem{lemma}{Lemma}[section]
\newtheorem{pro}{Proposition}[section]
\newtheorem{remark}{Remark}[section]

\newtheorem{definition}{Definition}[section]

\title{High-Contrast Transmission Resonances\\ for the Lam\'e System }
\author{Long Li \thanks{{RICAM, Austrian Academy of Sciences, A-4040, Linz, Austria (long.li@ricam.oeaw.ac.at)}} \; and Mourad Sini \thanks{{RICAM, Austrian Academy of Sciences, A-4040, Linz, Austria (mourad.sini@oeaw.ac.at)}}}

\begin{document}

\date{} 
\maketitle

\begin{abstract}
{\color{HW1}
We consider the Lam\'e transmission problem in $\mathbb{R}^3$ with a bounded isotropic elastic inclusion in a high-contrast setting, where the interior-to-exterior Lam\'e moduli and densities scale like $1/\tau$ as $\tau\to0$. We study the scattering resonances of the associated self-adjoint Hamiltonian, defined as the poles of the meromorphic continuation of its resolvent.

We obtain a sharp asymptotic description of resonances near the real axis as $\tau\to0$.
Near each nonzero Neumann eigenvalue of the interior Lam\'e operator there is a cluster of
resonances lying just below it in the complex plane; in this wavelength-scale regime the
imaginary parts are of order $\tau$ with non-vanishing leading coefficients. In addition, near zero (a subwavelength regime), we identify resonances with real parts of order $\sqrt{\tau}$ and prove a lifetime dichotomy: their imaginary parts are of
order $\tau$ generically, but of order $\tau^2$ for an explicit admissible set $\mathcal E$.
This yields a classification of long-lived elastic resonances in the high-contrast limit.

We also establish resolvent asymptotics for both fixed-size resonators and microresonators.
We derive explicit expansions with a finite-rank leading term and quantitative remainder bounds,
valid near both wavelength-scale and subwavelength resonances. For microresonators, at the
wavelength scale the dominant contribution is an anisotropic elastic point scatterer. Near the
zero eigenvalue, the leading-order behaviour is of monopole or dipole type, and we give a
rigorous criterion distinguishing the two cases.


}

\vspace{.2in}
{\bf Keywords}: Lam\'e system; high-contrast transmission problems; scattering resonances; wavelength-scale and subwavelength resonances; resolvent asymptotics.

\end{abstract}

\tableofcontents
\section{Introduction}

\subsection{Background and setting}
The propagation of elastic waves in high-contrast media lies at the heart of several problems in mechanics and materials science, notably in the design of elastic metamaterials and locally resonant structures. A widely used strategy in this field is to embed small inclusions or ``micro\-resonators'' into a host medium and tune their local resonances, often deeply subwavelength, so as to control effective dynamic properties such as band gaps, negative effective mass, or cloaking and other unusual dispersion effects, see for instance, \cite{LWPZ-11,LCS-02,LS-02}.

From a mathematical point of view, local resonances are naturally described in terms of the spectral and scattering theory of the underlying wave operator. In the scalar acoustic setting, the analysis of high-contrast inclusions and their Minnaert-type resonances has by now reached a mature stage: resonances can be located and expanded as functions of the contrast, effective point-scatterer models can be derived, and links to effective medium theory are well understood, see \cite{ADHE, AFGLZ-17, AZ18, AZ-17,FH, FH-04, MPS}. 
In contrast, the corresponding theory for the vectorial Lamé system in elasticity is less developed. The presence of both longitudinal and transverse waves, the richer structure of the Neumann spectrum, and the interaction of the vectorial traction with transmission conditions introduce substantial additional difficulties.

The purpose of this work is to develop a rigorous framework for \emph{high-contrast transmission resonances} in linear elasticity, and to apply it to the analysis of elastic microresonators.  We consider time-harmonic elastic waves in $\R^3$ in the presence of a bounded isotropic inclusion
$\Omega\subset\R^3$ with $C^{1,1}-$ smooth boundary $\Gamma:=\partial\Omega$ and outward unit normal $\nu$. The background medium
$\R^3\setminus\overline{\Omega}$ is characterized by positive constant Lam\'e parameters and density
$(\lambda_0,\mu_0,\rho_0)$, while the inclusion has high-contrast parameters $(\lambda_1/\tau,\mu_1/\tau,\rho_1/\tau)$ with a small parameter $\tau > 0$.
{We consider a scattering resonance $z\in\mathbb{C}\backslash \{0\}$ for which there exists a nontrivial \emph{outgoing} solution $v$ of}
\[
L_{\lambda_0,\mu_0}v + \rho_0 z^2 v = 0 \quad\text{in }\R^3\setminus\overline{\Omega},\qquad
L_{\lambda_1,\mu_1}v + \rho_1 z^2 v = 0 \quad\text{in }\Omega,
\]
coupled with the high-contrasted transmission condition
\begin{align*}
v^+ = v^- \quad \partial_{\nu,0} v^+ = \frac{1}\tau \partial_{\nu,1} v^-
\quad\text{on }\Gamma,
\end{align*}
and the appropriate outgoing radiation condition at infinity. Here,
\begin{align*}
L_{\lambda_j,\mu_j} v := \mu_j \Delta v+ (\lambda_j+\mu_j)\nabla(\nabla\!\cdot v)\;\;\mathrm{and}\;\; \partial_{\nu,j} v := \lambda_j \mathrm{div}v \; \nu + \mu_j (\nabla v + \nabla^T v)\nu, \quad j\in\{0,1\}
\end{align*}
denote the Lam\'e operators and the associated associated traction operators, respectively. 
{Equivalently, scattering resonances are the poles of the meromorphic continuation of the resolvent of the associated self-adjoint Hamiltonian $H^{e}(\tau;\Omega)$, which we refer to as the high-contrast elastic Hamiltonian, see \textnormal{(N\ref{No:5})} for its definition.}

Physically, the parameter $\tau$ measures the stiffness and inertial contrast between the inclusion and the surrounding background medium, as well as the effective strength of their coupling across the interface. The regime $\tau\to 0$ corresponds to an extreme-contrast limit, which can model, for example, a very soft thin interfacial layer, a large density mismatch, or an extremely compliant attachment between the inclusion and the background. 
In this regime, the composite can exhibit pronounced amplification of its elastic response at certain frequencies, driven by the inclusions and thereby influencing the effective macroscopic behaviour. This amplification manifests as scattering resonances of the high-contrast Hamiltonian $H^e(\tau;\Omega)$, lying close to the real axis, corresponding to long-lived responses.

\subsection{Main results}
Our first group of results concerns a sharp asymptotic description of scattering resonances of $H^e(\tau;\Omega)$ near the real axis in the high-contrast limit $\tau \to 0$. 
Let $\Lambda$ denote the Neumann spectrum of the interior Lamé operator in $\Omega$,
\begin{align}
\Lambda := \bigl\{\,z\in\mathbb{R} : \exists v\not\equiv0 \text{ with } L_{\lambda_1,\mu_1} v + z^2 \rho_1 v = 0 \text{ in }\Omega,\;
\partial_{\nu,1} v = 0 \text{ on }\Gamma\,\bigr\},\notag
\end{align}
and let $\Lambda_N(z_0)$ be the corresponding eigenspace for $z_0 \in \Lambda$.
Our first main result concerns the scattering resonances of $H^e(\tau;\Omega)$ near the Neumann spectrum $\Lambda$. For each eigenvalue $z_0 \in \Lambda$, we show that there exists a cluster of scattering resonances lying just below $z_0$ (see Theorem \ref{th:0})
and derive local $\tau$-asymptotic expansions whose coefficients are determined by the explicit finite-dimensional effective matrices (see Theorems \ref{th:1} and \ref{th:2}).

More precisely, for each $z_0\in\Lambda\setminus\{0\}$ we construct an effective matrix $M^{(1)}(z_0)$, whose entries involve the elastic Dirichlet-to-Neumann map on $\Gamma$ evaluated at $z_0$ and the corresponding interior Neumann eigenmodes. We prove that, as $\tau\to0$, there are at most $\dim\Lambda_N(z_0)$ scattering resonances $z^{(l)}(\tau)$ near $z_0$ , and that they admit the expansions
\begin{align*}
z^{(l)}(\tau) = z_0 - \frac{\tau}{2\rho_1 z_0}\,\kappa^{(l)}(z_0) + o(\tau),
\end{align*}
where $\kappa^{(l)}(z_0)$ are the eigenvalues of $M^{(1)}(z_0)$. In particular, the imaginary parts of $\kappa^{(l)}(z_0)$ are strictly positive, which yields quantitative information on the lifetime of wavelength-scale resonances as $\tau\to0$.

The behaviour near zero frequency is more delicate. We prove that the scattering resonances near $0$ are of subwavelength type: their real parts are of exactly order $\sqrt \tau$ and are determined by the eigenvalues of the leading-order effective matrix $M^{(1)}(0)$. Moreover, the leading imaginary parts satisfy a dichotomy: either they are of exactly order $\tau$ or of exactly order $\tau^2$. In the first case, the leading term is determined by an explicit algebraic relation $M^{(1)}(0)$ and $M^{(2)}(0)$; in the second regime, it depends on higher-order relations among $M^{(1)}(0), M^{(2)}(0), \ldots, M^{(5)}(0)$. These effective matrices are constructed from the elastic Dirichlet-to-Neumann map and its derivatives at $0$, together with Neumann correctors. To characterize {which of these two regimes} occur, we introduce an admissible set $\mathcal E$ (see \eqref{eq:3} for the definition and {\color{HW1}Remark \ref{rem:interpretation-E} for its interpretation}) and give an explicit criterion in terms of $\mathcal E$. In particular, if $\mathcal E \ne \emptyset$, there exists at least one subwavelength resonance with imaginary part of exact order $\tau^2$, and therefore the system admits very long-lived subwavelength resonances.

Our second group of results is to translate this spectral information into the resolvent, both for fixed-size resonators and for {microresonators} {( i.e. of small diameter)}. For a fixed-size resonator, we derive detailed expansions of the resolvent of ${H^e(\tau;\Omega)}$, with quantitative bounds, near wavelength-scale and subwavelength resonances (see Theorems \ref{th:3} and \ref{th:4}). In both cases, the leading term is a finite Neumann-eigenfunction expansion inside $\Omega$ with resonantly amplified coefficients, coupled to an outgoing field in the exterior determined by its boundary trace. 

We then consider scaled inclusions of the form
\[
\Omega_\varepsilon(y_0) := \{\,x : x = y_0 + \varepsilon(y-y_0),\; y\in\Omega\,\},
\qquad 0<\varepsilon\ll1,
\]
with center $y_0 \in \R^3$,
and obtain resolvent expansions for the corresponding Hamiltonian $H_e(\tau;\Omega_\varepsilon(y_0))$ with quantitative bounds in joint limits $\tau \rightarrow 0$ and $\varepsilon \to 0$. In the wavelength regime, the leading term is the sum of a $p-$ type  {(i.e. Pressure-type)} highly oscillatory point-scatterer and a $s-$ type {(i.e. Shear-type)} highly oscillatory point-scatterer supported at $y_0$, propagating with the $p-$ and $s-$ wave speeds, respectively, and modulated $p-/s-$ far-field patterns of the fixed exterior elastic scattering problem (see Theorem \ref{th:5}). This shows that the microresonator response near wavelength-scale resonances can display pronounced anisotropy. In the subwavelength regime, the leading term exhibits either monopole- or dipole-type behaviour. This is determined by the admissible set $\mathcal E$: if the frequency is close to an element of $\mathcal E$, the leading term is dipole-type; otherwise, the monopole-type contribution dominates (see Theorem \ref{th:6}). We also prove that $\mathcal E \ne \emptyset $ for a large class of inclusions symmetric with respect to its centre (see Remark \ref{re:1}), which guarantees the occurrence of dipole-dominated subwavelength behaviour within this family. Moreover, thanks to our sharp characterisation of the imaginary parts, our resolvent expansion undergoes a transition from a trivial limiting behaviour to a nontrivial one as the frequency approaches the subwavelength resonances.

\noindent {\emph{Finite-dimensional reduction and scaling regimes.} 
The analysis rests on two complementary resonance mechanisms and on an explicit
finite-dimensional reduction that controls both the \emph{location} and the \emph{width} of
scattering poles in the extreme-contrast limit $\tau\to 0$.
Away from $0$, resonances are driven by the interior Neumann spectrum $\Lambda$:
near each $z_0\in\Lambda\setminus\{0\}$ we obtain a resonance cluster lying just below $z_0$,
whose leading shifts and radiative damping are encoded by the effective matrix
$M^{(1)}(z_0)$ built from the interior Neumann eigenmodes and the exterior
Dirichlet-to-Neumann map, yielding sharp pole expansions and resolvent asymptotics with
quantitative remainders.
Near $0$, a subwavelength elastic regime emerges: the relevant reduction occurs on
the six-dimensional space of rigid body motions, the resonance scale is $\mathrm{Re}(z)\sim \sqrt{\tau}$,
and the leading leakage exhibits a lifetime dichotomy—typically $\mathrm{Im}(z) \sim \tau $, while
for the admissible set $\mathcal E$ the leading radiative term cancels and one obtains long-lived modes
with $\mathrm{Im}(z) \sim \tau^2$.
Finally, an $\varepsilon$-microresonator limit connects these expansions to effective
point-scatterer descriptions, producing anisotropic $p/s$ point interactions at the wavelength
scale and a monopole-versus-dipole classification in the subwavelength regime governed by $\mathcal E$.}

\noindent Methodologically, our analysis proceeds via a reduction to a boundary value problem. Its solvability reduces the spectral analysis to a perturbation problem for Neumann eigenvalues. The presence of eigenvalue multiplicities makes quantitative perturbation delicate. {In particular, Kato-like perturbation arguments do not necessarily apply.} We address this difficulty by exploiting the structure of the effective matrices and carrying out a detailed finite-dimensional matrix analysis of their interactions. General  perturbation theory provides a useful qualitative framework, but it does not by itself determine the first non-vanishing contribution to the imaginary parts in the presence of eigenvalue multiplicities. Our effective-matrix analysis provides this quantitative information. It should be noted that, in the subwavelength regime, whether the admissible set is empty is governed by a nontrivial interplay between $M^{(1)}(0)$, $M^{(2)}(0)$ and the associated enhancement spaces, a mechanism {that} doesn't appear to have a direct analogue in the existing resonance literature. This interplay yields a precise and computable classification of long-lived elastic resonances and of monopole versus dipole elastic point scatterers, which to the best of our knowledge is new.


\subsection{Related works}
The analysis of high-contrast inclusions and subwavelength resonances has a long history in the scalar setting. In acoustics, Minnaert-type resonances for bubbly media and related high-contrast configurations, together with effective point-scatterer descriptions and connections to homogenization, have been developed in a series of works; see, for instance, \cite{ACCS-1, ACCS-2, AZ18, AZ-17,FH, FH-04, LS-04} and the references therein. Related high-contrast and small-inclusion asymptotics for Helmholtz and Schr\"odinger operators, have also been investigated in \cite{AL, DKLZ, MP,MPS, MMS}. These works typically treat scalar operators, and the contrast is encoded either in the bulk coefficients or through boundary conditions depending on the spectral parameter. By contrast, the present paper concerns the full three-dimensional vectorial Lam\'e system with simultaneous contrast in inertial and traction transmission across $\Gamma$. The resulting hierarchy of enhancement spaces is genuinely vectorial: it couples longitudinal and transverse modes and is strongly influenced by the six-dimensional space of rigid body motions. In particular, even near a fixed interior Neumann eigenfrequency, the structure of the resonance cluster and the associated effective matrices are markedly richer than in the scalar Minnaert-type picture.

Subwavelength resonances for a related high-contrast Lam\'e system were studied in \cite{LZ-25}. There, resonances were characterized as the zeros of an associated system of boundary integral equations, and asymptotics were derived under the condition that the pressure and shear wave speeds have the same contrast across the interface. In contrast, we derive sharp subwavelength {(and also wavelength-scaled)} asymptotics, including both the real and imaginary parts, and quantify their impact on the corresponding
resolvent expansions. We also provide a computable criterion distinguishing monopole- and dipole-type behaviour in the microresonator and subwavelength regimes. For further results on high-contrast and subwavelength effects in elasticity from different perspectives and applications, we refer the reader, for example, to \cite{ CDS-1, CKVZ, DRL, LX}. 

{Besides subwavelength (Minnaert-type) resonances, high-contrast transmission configurations may also support \emph{Fabry--P\'erot-type} resonance mechanisms, driven by multiple internal reflections. In our recent work~\cite{LS-05}, we analyzed such high-contrast transmission and Fabry--P\'erot-type resonances in a scalar acoustic transmission setting and obtained asymptotic characterizations of the associated resonant frequencies in the extreme-contrast limit.} The present paper complements and extends this line of research to the full three-dimensional Lam\'e system, where the coexistence of compressional and shear modes and the traction transmission coupling generate resonance clusters near interior Neumann eigenfrequencies and require a genuinely vectorial effective-matrix analysis to capture the leading imaginary parts and the resulting resolvent behaviour. 

A different scattering resonance mechanism arises for the elastic wave equation outside a bounded obstacle with \emph{Neumann} (traction-free) boundary condition, which has been analyzed in a series of works by
Stefanov--Vodev; see, e.g., \cite{StefanovVodev1994,VP-95, VP-96} and the references therein.
In that setting, resonances are primarily boundary-driven (notably by Rayleigh-type surface waves),
and the techniques are predominantly microlocal: parametrix constructions,
meromorphic continuation of the exterior resolvent, the existence of resonances approaching the real axis
under geometric assumptions such as convexity, and quantitative distribution/counting results for the resulting scattering poles.
The present manuscript instead addresses a qualitatively different mechanism, namely a \emph{transmission}
problem with a high-contrast inclusion, in which scattering resonances bifurcate from interior traction-free eigenfrequencies of the inclusion and are shifted into the lower half-plane by radiation through the interface. As a consequence, our main results are of a different nature from their distribution theory: we obtain \emph{local} and \emph{sharp} asymptotic
expansions of resonance \emph{clusters} near each interior eigenfrequency, including explicit first-order
formulas for both the real shift and the imaginary part (resonance width) in terms of computable effective
matrices. In addition, we identify and analyze a genuinely high-contrast subwavelength regime near zero,
where the low-frequency resonances are governed by a finite-dimensional reduction on rigid-motion modes;
this produces a lifetime dichotomy (orders \(\tau\) versus \(\tau^2\)) encoded by the admissible set
\(\mathcal E\), a phenomenon absent in the classical Neumann obstacle case. Since, in our model, resonances bifurcate from interior traction-free eigenfrequencies, the associated resonant states are perturbations of the corresponding interior Neumann eigenfunctions; they are typically localized in the inclusion for large frequencies. Surface-wave (Rayleigh-type) localization may occur only insofar as the underlying interior eigenmode itself concentrates near $\partial \Omega$, and should be distinguished from the exterior Neumann obstacle resonances driven by Rayleigh-wave propagation.

{\color{HW1} The analysis developed in this work is not tied to the specific high-contrast transmission law
\(\partial_{\nu,0} v^+ = \tau^{-1}\partial_{\nu,1} v^-\) (or, equivalently, to the particular simultaneous scaling of
Lam\'e parameters and density inside the inclusion), but rather to the fact that the transmission
problem can be reduced to an interior boundary value problem with \(\tau\)-dependent boundary system. This reduction can be adapted to other contrast mechanisms and interface laws; a more detailed discussion is given in Remark \ref{re:ex}. Within this setting, the similar Fredholm--analytic framework and the finite-dimensional reduction near characteristic frequencies continue to apply.}

The paper is organized as follows. In Section \ref{sec:2} we introduce the functional setting, the high-contrast elastic Hamiltonian, the Neumann projectors, and the effective matrices, and we fix the notation used in the sequel. Section \ref{sec:3} collects our main results on the localization and asymptotics of scattering resonances and on the resolvent behaviour for both fixed-size resonators and microresonators. In Section \ref{sec:4}, we establish the meromorphic continuation of the resolvent and provide several alternative characterizations of scattering resonances. Sections~\ref{sec:6} contain the detailed proofs of the resonance asymptotics and resolvent expansions in the wavelength-scale and subwavelength regimes. Appendix~\ref{sec:A} gathers the proofs of Propositions \ref{pro:1}--\ref{pro:4}. Appendix~\ref{sec:B} collects the proofs of additional technical lemmas used in the proofs of the main results.

\section{Notations and conventions}\label{sec:2}

\setcounter{secnumdepth}{4}
\makeatletter
\renewcommand\theparagraph{\arabic{paragraph}}
\renewcommand\theHparagraph{\theHsection.\arabic{paragraph}}
\makeatother
\titleformat{\paragraph}[runin]{\normalfont\bfseries}{(N\theparagraph)}{0.6em}{}

\paragraph{Function Spaces.}
For any open set $D \subset \R^3$ and $s \ge 0 $, define $\mathbf H^s(D):= [H^s(D)]^3$, $\mathbf H^{s}_{\mathrm{loc}}(D):=\{\mathbf u:\ \chi\,\mathbf u\in \mathbf H^{s}(D),\; \forall\chi\in C_c^\infty(D)\}$, and $\mathbf H^{s}_{\mathrm{comp}}(D):=\Big\{\mathbf u\in \mathbf H^{s}(D):\ \operatorname{supp}\mathbf u\Subset D\Big\}$. We identify $\mathbf H^0(D)$, $\mathbf H^{0}_{\mathrm{loc}}(D)$ and $\mathbf H^{0}_{\mathrm{comp}}(D)$ with $\mathbf L^2(D)$, $\mathbf L^{2}_{\mathrm{loc}}(D)$ and $\mathbf L^{2}_{\mathrm{comp}}(D)$, respectively. Let $\mathbf L_\beta^2(\R^3) := \{u\in \mathbf L^2_{\textrm{loc}}\left(\R^3\right): (1+|x|^2)^{\beta/2} u(x) \in \mathbf L^2(\R^3)\}$ for $\beta \in \R$.
Set $\mathbb L^2(\mathbb S^2):= (L^2(\mathbb S^2))^3$ and $\mathbf H^{\pm 1/2}(\Gamma):= [H^{\pm 1/2}(\Gamma)]^3$.
Let $\gamma: \mathbf H^1(\Omega) \to \mathbf H^{1/2}(\Gamma)$ denote the (componentwise) trace operator. $\langle \cdot,\cdot\rangle_{\Gamma}
$ denotes the duality pairing $\mathbf H^{-1/2}(\Gamma) \times \mathbf H^{1/2}(\Gamma) \rightarrow \mathbb C$; $(\cdot,\cdot)_{\Omega}$ denotes a $\mathbf L^2(\Omega)$-inner product on $\Omega$.

\paragraph{High-contrast Hamiltonian.} \label{No:5}
For $\tau >0$, define 
\begin{align*}
H^{e}(\tau;\Omega) v := \frac{1}{\rho_\tau} [\nabla \lambda_\tau (\nabla \cdot v) + \nabla \cdot \mu_\tau (\nabla v + \nabla^T v)]
\end{align*}
with the domain 
\begin{align*}
\mathbf D(H^{e}(\tau;\Omega)):=\left\{v\in\mathbf H^1(\R^3): \frac{1}{\rho_\tau} [\nabla \lambda_\tau (\nabla \cdot v) + \nabla \cdot \mu_\tau (\nabla v+ \nabla^T v) ] \in \mathbf L^2(\R^3)\right\}.
\end{align*}
For $z \in \overline{\CC_+}$ and $x\in \R^3$, let 
\begin{align*}
R_{H^{e}(\tau;\Omega)}(z): \mathbf L^2(\R^3) \to \mathbf D(H^{e}(\tau;\Omega)),\quad R_{H^{e}(\tau;\Omega)}(z) f(x) := \left(-H^{e}(\tau;\Omega) - z ^2 \right)^{-1}f(x).
\end{align*}

\paragraph{Neumann specturm and projectors.} \label{No:1}
We denote by 
\begin{align*}
\Lambda:=\big\{\omega\in \mathbb R: \exists v \ne 0 \; \mathrm{with}\; L_{\lambda_{1}, \mu_1} v + \omega^2\rho_1 v = 0 \; \mathrm{in}\; \Omega, \; \partial_{\nu,1} v = 0 \; \mathrm{on}\; \Gamma\big\}
\end{align*}
the Neumann spectrum of the Lam\'e operator. For $z_0 \in \Lambda$, define 
\begin{align} \label{eq:222}
n(z_0):= \textrm{dim}(\Lambda_N(z_0))
\end{align}
with 
\begin{align*}
\Lambda_N(z_0):=\{v \in \mathbf H^1(\Omega): L_{\lambda_{1}, \mu_1} v + z_0^2\rho_1 v = 0 \; \mathrm{in}\; \Omega, \; \partial_{\nu,1} v = 0 \; \mathrm{on}\; \Gamma\big\}.
\end{align*}
Let $\{\mathbf e^{(j)}(z_0)\}^{n(z_0)}_{j=1}$ be an $\mathbf L^2(\Omega)$-orthonormal basis in $\Lambda_{N}(z_0)$, and let  
\begin{align*}
&\Pi(z_0): \mathbf L^2(\Omega) \rightarrow \mathbb C^{n(z_0)}, \;\; [\Pi(z_0)\psi]_j:=((\psi, e^{(j)}(z_0))_{\Omega})^{n(z_0)}_{j=1},\\
&P(z_0): \mathbf L^2(\Omega)\rightarrow \Lambda_N(z_0), \;\; P(z_0)\psi:= \sum^{n(z_0)}_{j=1} [\Pi(z_0)\psi]_j e^{(j)}(z_0).
\end{align*}
It is known that 
\begin{align*}
n(0) = 6\;\; \mathrm{and}\;\;\Lambda_N(0) = \{v: v(x) = q_1 + q_2 \times x,\; q_1,q_2\in \R^3\}.
\end{align*}
Without loss of generality, we choose $e^{(1)}(0),e^{(2)}(0),e^{(3)}(0)\in\Lambda_N(0)$ as the normalised translational rigid motions:
\begin{align*}
e^{(1)}(0) =\frac{1}{\sqrt{|\Omega|}}(1,0,0)^T,\quad
e^{(2)}(0) =\frac{1}{\sqrt{|\Omega|}}(0,1,0)^T,\quad
e^{(3)}(0) =\frac{1}{\sqrt{|\Omega|}}(0,0,1)^T.
\end{align*}
The remaining basis functions are chosen to complete $\{e^{(j)}(0)\}_{j=1}^6$ to an $L^2(\Omega)$-orthonormal basis of $\Lambda_N(0)$; their explicit form will not be needed.

\paragraph{Kupradze fundamental matrix, layer potentials and free resolvents.} 
For $z\in \mathbb C_+$ and $x,y \in \R^3$ with $x\ne y$, $G^{(j)}(x,y;z)$ denotes the Green matrix solving 
\begin{align*}
L_{\lambda_j,\mu_j} G^{(j)}(x,y;z) + z^2 \rho_j G^{(j)}(x,y;z) = -\delta(x-y) I
\end{align*}
with the Kupradze radiation condition ($j \in \{0,1\}$). Define
\begin{align*}
&SL_{j} (z): \mathbf H^{-1/2}(\Gamma) \rightarrow \mathbf H_{\textrm{loc}}^{1}(\R^3 \backslash \Gamma),\;\;\;
\big(SL_{j} (z) \phi\big)(x):= \int_{\Gamma} G^{(j)}(x,y;z)\phi(y)dS(y),\\
&DL_{j} (z): \mathbf H^{-1/2}(\Gamma) \rightarrow \mathbf H_{\textrm{loc}}^{1}(\R^3 \backslash \Gamma),\;\;\;\big(DL_{j}(z)\phi\big)(x):= \int_{\Gamma}\partial_{{\nu,j}(y)}G^{(j)}(x,y;z)\phi(y)dS(y),
\end{align*}
and 
\begin{align*}
&S_{j}(z): \mathbf H^{-1/2}(\Gamma)\rightarrow \mathbf H^{1/2}(\Gamma),\;\;\;\left(S_{j}(z)\phi\right)(x) 
:= \int_{\Gamma} G^{(j)}(x,y;z) \phi(y)d S(y), \\
&K_{j}(z): \mathbf H^{-1/2}(\Gamma)\rightarrow \mathbf H^{1/2}(\Gamma),\;\;\;\left(K_{j}({z})\phi \right)(x),
:=\int_{\Gamma}  \partial_{{\nu,j}(y)}G^{(j)}(x,y;z) \phi(y)dS(y)\\
&K^*_{j}(z): \mathbf H^{-1/2}(\Gamma)\rightarrow \mathbf H^{1/2}(\Gamma),\;\;\;\left(K^*_{j}({z})\phi \right)(x)
:=\int_{\Gamma}  \partial_{{\nu,j}(x)}G^{(j)}(x,y;z) \phi(y)dS(y).
\end{align*}
{For $z\in \overline{\mathbb C_+}$ and $x\in \R^3$, define
\begin{align}
R_{0}(z): \mathbf L_{\beta}^2(\R^3) \to \mathbf L_{-\beta}^2(\R^3),\;\;\; \left(R_{0}(z)f\right)(x): = \rho_0 \int_{\R^3} G^{(0)}(x,y;z)f(y)dy \label{eq:245}.
\end{align}
Here, $\beta > 1/2$.} Obviously, 
\begin{align*}
\left(-\frac{1}{\rho_0}L_{\lambda_0,\mu_0} - z^2\right) R_{0}(z)(f) = f.
\end{align*}

\paragraph{Exterior operators.} \label{No:4}
For $z\in \overline{\mathbb C_+}$, let 
\begin{align*}
    &\mathcal F^+(z):\mathbf H^{1/2}(\Gamma)\to \mathbf H^1_{\mathrm{loc}}(\R^3 \backslash \Omega),\quad \mathcal F^+(z)g:=v_z,\\
    &\mathcal N(z): \mathbf H^{1/2}(\Gamma) \to \mathbf H^{-1/2}(\Gamma), \quad \mathcal N(z)g:= \partial_{\nu,1}v_z,
\end{align*}
where $v_z$ is the unique radiating solution to 
\begin{align*}
(L_{\lambda_0,\mu_0} + \rho_0 z^2) v_z = 0\;\;\; \mathrm{in}\;\; \R^3\backslash \overline \Omega\;\; \mathrm{and}\;\; \gamma v_z = g\;  \mathrm{on}\; \Gamma.
\end{align*}
For $\sigma \in \{p,s\}$ and $\omega \in \R \backslash \{0\}$, define 
\begin{align}
&\mathcal F_{\sigma}^\infty(\omega):\mathbf H^{1/2}(\Gamma) \rightarrow \mathbf L^2(\mathbb S^2), \notag \\
&\big(\mathcal F_\sigma^\infty(\omega)g\big)(\hat x):=\int_{\Gamma} \partial_{\nu,1(y)}(G^{(0)})_\sigma^{\infty}(\hat x, y;\omega)g(y) - (G^{(0)})_\sigma^{\infty}(\hat x, y;\omega)(\mathcal N(\omega)g)(y), \;\;\; \hat x \in \mathbb S^2. \label{eq:206}
\end{align}
where
\begin{align}
&(G^{(0)})_p^{\infty}(\hat x, y;\omega) := \frac{\hat x \hat x^T }{4\pi(\lambda_0 + 2\mu_0)}e^{-i\omega\frac{ \hat x \cdot y}{c_{p,0}}}\;\;\textrm{and}\;\;(G^{(0)})_s^{\infty}(\hat x, y;\omega) :=  \frac{I - \hat x \hat x^T}{4\pi \mu_0} e^{-iz\frac{\hat x \cdot y}{c_{s,0}}} \label{eq:188}
\end{align}
denote the compressional and shear far-field patterns of the Kupradze fundamental matrix $G^{(0)}$, respectively.
Here,  
\begin{align}
c_{s,0}:= \sqrt{\frac{\mu_0}{\rho_0}} \;\; \mathrm{and} \; \;c_{p,0}:= \sqrt{\frac{\lambda_0 + 2\mu_0}{\rho_0}}. \notag
\end{align}
denote the background $S$-wave and $P$-wave speeds, respectively.
Furthermore, for each $z\in \overline{\CC_+}$, define 
\begin{align}
R^{\mathrm{ex}}(z): \mathbf H^1(\Omega) \to \mathbf L^2_{\mathrm{loc}}(\R^3),\;\;\;  R^{\mathrm{ex}}(z)\psi := \begin{cases}
\mathcal F^+(z) \gamma \psi    \quad & \mathrm{in}\; \R^3 \backslash \Omega,\\
\psi  \quad & \mathrm{in}\; \Omega. 
\end{cases}\notag
\end{align}

\paragraph{Effective matrices.}\label{No:2}
Define 
\begin{align*}
&M^e(\tau, z;z_0):= 2(z - z_0) z_0 \rho_1 I + \tau M^{(1)}(z_0)\quad \textrm{for}\; z_0 \in \Lambda \backslash \{0\}\; \mathrm{and}\; (\tau,z) \in \R_+\times \CC
\end{align*}
and
\begin{align*}
&M^e(\tau,z;0):= z^2 \rho_1 I + \tau M^{(1)}(0) - i \tau z M^{(2)}(0) - \tau^2 M^{(3)}(0) - \tau z^2 \left(M^{(4)}(0) + 2\rho_1M^{(1)}(0)\right)\\
&\qquad\qquad\;\;\;- z^4 \rho^2_1 I - i \tau z^3 M^{(5)}(0) - \tau^2 z M^{(6)}(0)   \quad \mathrm{for}\; (\tau,z) \in \R_+\times \CC.
\end{align*}
Here, $M^{(1)}(z_0)$ and $M^{(j)}(0)$ ($j = 0, 1,\ldots 6$) are called effective matrices. Their entries are defined as follows:
\begin{align} 
&M^{(1)}_{lj}(z_0) := \left\langle  \mathcal N(z_0)e^{(l)}(z_0), e^{(j)}(z_0) \right\rangle_{\Gamma}\quad l,j \in \left\{1,\ldots,n(z_0)\right\}\; \textrm{for}\; z_0 \in \Lambda, \label{eq:30}
\end{align}
and
\begin{align}
& M^{(2)}_{lj}(0):= i \left\langle \partial_z \mathcal N(0)e^{(l)}(0), e^{(j)}(0)\right\rangle_{\Gamma}, \notag \\
& M^{(3)}_{lj}(0) := - \left\langle  e^{(l)}_{\mathcal N}(0), \partial_{\nu,1} e_{\mathcal N}^{(j)}(0) \right\rangle_{\Gamma}, \label{eq:56}\\ 
& M^{(4)}_{lj}(0) := - \frac{1}2\left\langle \partial^2_z \mathcal N(0) e^{(l)}(0),  e^{(j)}(0)\right\rangle_{\Gamma}, \notag \\
& M^{(5)}_{lj}(0) := \frac{i}6 \left\langle  \partial^3_z \mathcal N(0) e^{(l)}(0), e^{(j)}(0)\right\rangle_{\Gamma} - 2 \rho_1 M_{jl}^{(2)}(0), \label{eq:52} \\
& M^{(6)}_{lj}(0):= - \left\langle \partial_{\nu,1} e^{(l)}_{\mathcal N'}(0), e_{\mathcal N}^{(j)}(0)\right\rangle_{\Gamma} - \left\langle  e^{(l)}_{\mathcal N}(0), \partial_{\nu,1} e_{\mathcal N'}^{(j)}(0)\right\rangle_{\Gamma}, \quad l,j \in \left\{1,\ldots,6\right\},\label{eq:24}
\end{align} 
where $e^{(l)}_{\mathcal N}(0)$ and $e^{(l)}_{\mathcal N'}(0)$ (for $l \in \{1,\ldots,6\}$) solve
\begin{equation} \label{eq:93}
\begin{aligned}
\left\{
\begin{aligned}
\bigl[L_{\lambda_1,\mu_1} + P(0)\bigr] e^{(l)}_{\mathcal N}(0) = 0  \;\; & \text{in } \Omega,\\
\partial_{\nu,1}e^{(l)}_{\mathcal N}(0) = \mathcal N(0)\, e^{(l)}(0)\;\; & \text{on } \Gamma
\end{aligned}
\right.
\;\; \mathrm{and}\;\;
\left\{
\begin{aligned}
\bigl[L_{\lambda_1,\mu_1} + P(0)\bigr] e^{(l)}_{\mathcal N'}(0)  = 0 \;\;  &\text{in}\; \Omega,\\
\partial_{\nu,1}e^{(l)}_{\mathcal N'}(0) = \partial_z \mathcal N(0)\, e^{(j)}(0) \;\; &\text{on } \Gamma
\end{aligned}
\right.
\end{aligned}
\end{equation}
respectively.

\paragraph{Principal enhancement spaces and their associated eigenvalues.} \label{No:3}
Let 
\begin{align} \label{eq:1}
\mathcal L^{(0)}:= \left\{\kappa^{(1)}(0),\ldots,\kappa^{(m)}(0)\right\} 
\end{align}
denote the set of the distinct eigenvalues of $M^{(1)}(0)$. It should be remarked that all elements of $\mathcal L^{(0)}$ are strictly negative (see statement \eqref{b2} of Lemma \ref{le:6}). For each $\kappa \in \mathcal L^{(0)}$, let $Q(\kappa)$ denote the matrix whose columns form an orthonormal basis of the associated eigenspace, and let 
\begin{align} \label{eq:2}
\mathcal L^{(1)}(\kappa):=  \left\{\kappa' \in \R: \mathrm{det}\left(\kappa' I - Q^T({\kappa}) M^{(2)}(0) Q(\kappa)\right) = 0 \right\}
\end{align}
denote the distinct eigenvalues of the reduced matrix $M^{(2)}(0)$ on the range of $Q(\kappa)$. Let
\begin{align}
\mathcal E:= \left\{\kappa \in \R: \mathrm{det} \left(\kappa I - M^{(1)}(0)\right) = 0, \; \mathrm{Ker}\left(\kappa I - M^{(1)}(0)\right) \cap \mathrm{Ker}\left(M^{(2)}(0)\right) \neq \emptyset \right\} \label{eq:3}
\end{align}
denote the admissible set.

{\color{HW1} \begin{remark}[Nonemptiness of the admissible set $\mathcal E$]\label{re:1}
 Assume that $\Omega$ is origin-symmetric, i.e. $\Omega = - \Omega$. Define 
\begin{align*}
\left( \mathcal R f \right)(x) = f(-x), \quad x\in \Gamma. 
\end{align*}
It is well-known that is $\mathcal R$ is a linear bounded mapping from $\mathbf H^{\pm1/2}(\Gamma)$ to  $\mathbf H^{\pm1/2}(\Gamma)$. It is easy to verify that 
\begin{align} \label{eq:136}
S_0 \mathcal R f = \mathcal R S_0 f, \quad f\in\mathbf H^{-1/2}(\Gamma).
\end{align}
We note that $S_0$ is invertible as a mapping from $\mathbf H^{-1/2}(\Gamma)$ to $\mathbf H^{1/2}(\Gamma)$. Then, \eqref{eq:136} is equivalent to 
\begin{align} \label{eq:137}
  \mathcal R\left(S_0(0)\right)^{-1} g = \left(S_0(0)\right)^{-1} \mathcal Rg, \quad g \in\mathbf H^{1/2}(\Gamma).
\end{align}
Since $\mathcal R e^{(k)}(0) = e^{(k)}(0)$ for $k\in \{1,2,3\}$ and $\mathcal R e^{(j)}(0) = - e^{(j)}(0) $ for $j \in \{4,5,6\}$, it follows from \eqref{eq:137} that  
\begin{align*}
&\int_{\Gamma} e^{(k)}(0)(x) \left(\left(S_0(0)\right)^{-1}e^{(j)}(0)\right)(x) dS(x) \\
&= \int_{\Gamma} e^{(k)}(0)(-x) \left(\left(S_0(0)\right)^{-1}e^{(j)}(0)\right)(-x) dS(x)\\
& = - \int_{\Gamma} e^{(k)}(0)(x) \left(\left(S_0(0)\right)^{-1}e^{(j)}(0)\right)(x) dS(x), \quad k\in \{1,2,3\},\; j\in \{4,5,6\}.
\end{align*}
With the aid of \eqref{eq:30} and the fact that $\mathcal N(0) e(0) = -\left(S_0(0)\right)^{-1} e(0)$ for each $e(0) \in {\Lambda_N(0)}$, we have 
\begin{align*}
M^{(1)}_{kj}(0) = 0\quad \mathrm{for}\; k\in \{1,2,3\}\; \mathrm{and}\; j\in \{4,5,6\}.
\end{align*}
This means that for origin-symmetric domain, there exists eigenvalues of $M^{(1)}(0)$ such that \eqref{eq:3} holds, that is $\mathcal E \ne \emptyset$. 
\end{remark}
}
\begin{remark}[Interpretation of the admissible set $\mathcal E$]
\label{rem:interpretation-E}
The admissible set $\mathcal E$ may be understood as a \emph{mode-selection} condition at the
quasi-static level. Recall that $\kappa\in\mathcal E$ means that $\kappa$ is an eigenvalue of
$M^{(1)}(0)$ and, moreover, that $M^{(1)}(0)$ admits an eigenvector $a\neq 0$ for which the
next matrix $M^{(2)}(0)$ vanishes, namely
$a\in\mathrm{Ker}(\kappa I-M^{(1)}(0))\cap \mathrm{Ker}(M^{(2)}(0))$; see \eqref{eq:3}.
In particular, $\mathcal E$ singles out those quasi-static rigid-motion modes for which the leading
``radiative correction/damping'' encoded by $M^{(2)}(0)$ is absent, so that the first non-trivial
imaginary part (and hence the resonance width) appears only at higher order.

In symmetric geometries (e.g.\ origin-symmetric domains), $M^{(1)}(0)$ becomes block diagonal with
respect to the natural splitting of rigid motions into translations and rotations; see Remark~\ref{re:1}.
This clarifies why eigenmodes can be chosen purely translational or purely rotational. Nevertheless,
membership in $\mathcal E$ is stronger than this translation--rotation decoupling, since it
additionally requires compatibility with $\mathrm{Ker}(M^{(2)}(0))$.

Finally, $\mathcal E$ is precisely the threshold that distinguishes the leading multipole content of
the resonant perturbation. Away from $\mathcal E$ the dominant contribution is \emph{monopole-type}
(the leading term is proportional to the Green tensor $G_0(\cdot,y_0;\omega)$; cf.\ \eqref{eq:146}),
whereas for frequencies tuned to $\mathcal E$-modes (under the refined scaling leading to
\eqref{eq:3.14}) the monopole contribution is suppressed at leading order and the dominant
term is \emph{dipole-type}, involving $\nabla_y G_0(\cdot,y_0;\omega)$ and the first moment of the
boundary density. In this sense, tuning to $\mathcal E$ shifts the leading resonant response from a
monopolar to a dipolar regime.
\end{remark}


\textbf{Case 1}: $\mathbf{\kappa \in \mathcal L^{(0)} \backslash \mathcal E}$.
In this case, we define the associated principal enhancement space by 
\begin{align}
 E_{\mathrm p}(\kappa):= \bigoplus_{k'\in \mathcal L^{(1)}(\kappa)}E_{\mathrm p}(\kappa; \kappa'), \notag
\end{align}
where each space $E_{\mathrm p}(\kappa;\kappa')$ is defined by
\begin{align*}
E_{\mathrm p}(\kappa;\kappa'):= \left\{ Q(\kappa)a : a\in \mathrm{Ker}\left(\kappa' I - \left(Q(\kappa)\right)^T M^{(2)}(0)Q(\kappa)\right)\right\}.           
\end{align*}
The projection onto $E_{\mathrm p}(\kappa)$ is defined by 
\begin{align*}
P_{E_{\mathrm p}(\kappa)}:=\sum_{\kappa'\in \mathcal L^{(1)}(\kappa)} P_{E_{\mathrm p(\kappa;\kappa')}},
\end{align*}
where $P_{E_{\mathrm p(\kappa;\kappa')}}$ denotes the orthogonal projection onto the principal enhancement subspace ${E_{\mathrm p(\kappa;\kappa')}}$, that is
\begin{align} \label{eq:89}
P_{E_{\mathrm p(\kappa;\kappa')}}: \CC^6 \to E_{\mathrm p(\kappa;\kappa')}, \quad \kappa \in \mathcal L^{(0)} \backslash \mathcal E\; \mathrm{and}\; \kappa' \in \mathcal L^{(1)}(\kappa).
\end{align}

\textbf{Case 2}: $\kappa \in \mathcal L^{(0)} \cap \mathcal E$. 
In this case, it can be seen that $0 \in \mathcal L^{(1)}(\kappa)$. We first prepare some new notations. Let $Q_{0}(\kappa)$ and $Q_{\perp}(\kappa)$ denote the matrices whose columns form an orthonormal basis of 
\begin{align*}
    \mathrm{Ker}\left(\kappa I - M^{(1)}(0)\right) \cap \mathrm{Ker}\left(M^{(2)}(0)\right) \;\; \mathrm{and}\;\; \mathrm{Ker}\left(\kappa I - M^{(1)}(0)\right) \backslash \mathrm{Ker}\left(M^{(2)}(0)\right),
\end{align*}
respectively, and let 
\begin{align} \label{eq:4}
  \mathcal L^{(2)}{(\kappa;0)}:= \left\{\kappa'' \in \R: \mathrm{det}\left(\kappa'' I - M^{(1)}_{\mathrm p}(\kappa)\right) = 0 \right\},
\end{align}
where $M^{(1)}_{\mathrm p} (\kappa)$ is defined by
\begin{align}&M^{(1)}_{\mathrm p}(\kappa):=  Q^T_0(\kappa) M^{(0)}_{\mathrm p}(\kappa) Q_0(\kappa) \label{eq:203}
\end{align}
with
\begin{align}
&M^{(0)}_{\mathrm p}(\kappa):=M^{(3)}(0) - {\frac{\kappa} {\rho_1}} M^{(4)}(0) - 2\kappa M^{(1)}(0) + \kappa^2I.\label{eq:213}
\end{align}
Then, for $\kappa \in \mathcal L^{(0)} \cap \mathcal E$ and $\kappa'' \in \mathcal L^{(2)}(\kappa,0)$, let $Q_0(\kappa;0,k'')$ denote the matrix whose columns form an orthonormal basis of the eigenspace of $M_p^{(1)}(\kappa)$ associated with $\kappa''$, 
and let 
\begin{align}
&\mathcal L^{(3)}{(\kappa;0,k'')}:= \notag\\ 
&\left\{\kappa''' \in \R: \mathrm{det}\left(\kappa''' I - Q^T_{0}(\kappa;0,k'') \left(\sum^3_{l=2}M^{(l)}_{\mathrm p}(\kappa)\right)Q_{0}(\kappa;0,k'')\right)  = 0 \right\}, \label{eq:5}
\end{align}
where 
\begin{align}
&M^{(2)}_{\mathrm p}(\kappa):= -\frac{\kappa}{\rho_1}Q^T_0(\kappa) M^{(5)}(0) Q_0(\kappa), \label{eq:204}\\
&M_{\mathrm p}^{(3)}(\kappa):= -\frac{\rho_1}{\kappa}\left(Q^T_{0}(\kappa) M_{\mathrm p}^{(0)}(\kappa)Q_{\perp}(\kappa)\right) \left(Q^T_{\perp}(\kappa) M^{(2)}(0) Q_{\perp}(\kappa)\right)^{-1}\left(Q^T_{\perp}(\kappa) M_{\mathrm p}^{(0)}(\kappa) Q_{0}(\kappa)\right). \label{eq:205}
\end{align} 

  Now we define the associated principal enhancement space at $\kappa$ by
\begin{align*}
E_{\mathrm p}(\kappa):= \bigoplus_{\kappa''\in \mathcal L^{(2)}(\kappa;0)}E_{\mathrm p}(\kappa; 0, \kappa''), 
\end{align*}
where each space $E_{\mathrm p}(\kappa; 0, \kappa'')$ is defined by
\begin{align*}
E_{\mathrm p}(\kappa;0,\kappa''):= \bigoplus_{\kappa'''\in \mathcal L^{(3)}(\kappa;0;k'')}E_{\mathrm p}(\kappa; 0,\kappa'',k''')  
\end{align*}
with
\begin{align*}
&E_{\mathrm p}(\kappa; 0,\kappa'',\kappa'''):= \\
&\left\{Q_0(\kappa)Q_0(\kappa;0,k'')a : a\in \mathrm{Ker}\left(\kappa''' I - 
Q^T_{0}(\kappa;0,k'')\left(\sum^3_{l=2}M^{(l)}_{\mathrm p}(\kappa)\right) Q_0(\kappa;0,k'')\right)\right\}. 
\end{align*}
The projection onto the principal enhancement subspace $E_{\mathrm p}(\kappa; 0, \kappa'')$ is defined by 
\begin{align}\notag
    P_{E_{\mathrm p}(\kappa;0,\kappa'')}:= \sum_{\kappa'''\in \mathcal L^{(3)}(\kappa;0,\kappa'')}P_{E_{\mathrm p(\kappa;0,\kappa'',\kappa''')}},
\end{align}
where $P_{E_{\mathrm p(\kappa;0,\kappa'',\kappa''')}}$ denotes the orthogonal projection onto the principal enhancement subspace ${E_{\mathrm p(\kappa;0,\kappa'',\kappa''')}}$, that is
\begin{align} \label{eq:90}
P_{E_{\mathrm p(\kappa;0,\kappa'',\kappa''')}}: \CC^6 \to E_{\mathrm p(\kappa;0,\kappa'',\kappa''')}, \quad \kappa \in \mathcal L^{(0)} \cap \mathcal E,\; \kappa'' \in \mathcal L^{(2)}(\kappa;0)\; \mathrm{and}\; \kappa''' \in \mathcal L^{(3)}(\kappa;0,\kappa'').
\end{align}

\paragraph{Sesquilinear forms.} \label{No:6}
Given $\tau > 0$ and $z\in \CC$, for any $\psi_1, \psi_2 \in\mathbf H^1(\Omega)$, we define
\begin{align}
&J_{\tau}(\psi_1, \psi_2,z) =  J(\psi_1,\psi_2,z) +  \tau \left\langle \mathcal N(z)\gamma\psi_1, \psi_2 \right\rangle_{\Gamma}, \label{eq:34}\\
\mathrm{and}\;&J^{\rm{dom}}_{z_0,\tau}(\psi_1,\psi_2,z) = J(\psi_1, \psi_2,z_0) + \left(P(z_0)\psi_1,\psi_2\right)_\Omega + \left(z^2-z_0^2\right)\rho_1 \left(\psi_1, \psi_2\right)_{\Omega} \notag \\
&\qquad\qquad\qquad\quad+ \tau \left\langle \mathcal N(z)\gamma\psi_1, \psi_2 \right\rangle_{\Gamma}. \label{eq:33}
\end{align}
Here, 
\begin{align}
J(\psi_1, \psi_2,z) & = -\lambda_1 \int_{\Omega} \textrm{div} \psi_1(x) \; \textrm{div} \psi_2(x) dx  - 2\mu_1 \int_{\Omega} \mathbb D\psi_1(x) : \mathbb D \psi_2(x)  dx \notag \\
&\quad+ z^2\rho_1 \int_{\Omega} \psi_1(x) \cdot \psi_2(x) dx, \label{eq:31}
\end{align}
where the operator $\mathbb D$ is defined as follows
$\mathbb D v = (\nabla v + \nabla^T v)/2$ for $v\in \mathbf H^1(\Omega)$,
and $A:B$ denotes the Frobenius product of two $3\times 3$ matrices.

\paragraph{Asymptotic notation.}  
We use the standard symbols $O(\cdot)$, $o(\cdot)$, and $\Theta(\cdot)$ for scalar quantities: given $\varepsilon_1,\varepsilon_2>0$ and $\varepsilon_2\to 0$, 
we write
\begin{align*}
&\vep_1 = O(\vep_2),\quad  \mathrm{if}\; \exists\; C \in \R_+\; \mathrm{with} \; \vep_1 <C\vep_2,\\ 
&\vep_1 = o(\vep_2), \quad \;\mathrm{if}\; \vep_1/\vep_2 \to 0,\\
&\vep_1 = \Theta(\vep_2),\quad  \mathrm{if}\; \exists C_1,C_2\in \R_+\; \mathrm{with} \; C_1\vep_2 < \vep_1 <C_2\vep_2.
\end{align*}

If $X$ and $Y$ are two Banach spaces and $g(\vep_2) \in X$ and $h \in Y$, we write $g(\vep_2) = O_X\bigl(\vep_2 \|h\|_Y\bigr)$ to mean that there exists a positive
constant $C$ independent of $\vep_2$ and $f$ such that
\begin{align*}
\|g(\vep_2)\|_X \le C\,\vep_2 \|h\|_Y.
\end{align*}
In particular, $g(\vep_2) = O_X(\vep_2)$ means that $\|g(\vep_2)\|_X = O(\varepsilon_2)$ in the scalar sense above. 
When the underlying space is clear (e.g. for vectors or matrices with a fixed norm), we simply write $O(\cdot)$ instead of $O_X(\cdot)$.

\paragraph{Miscellaneous notation}\
$\mathbb C_{\pm} := \{z\in \mathbb C: \pm\textrm{Im}(z) \ge 0\}$, $\mathbb R_+:= \{s\in \mathbb R: s > 0\}$. $B_r(y) \subset \R^3$ denotes an open ball of radius $r$ centered at the point $y\in \R^3$. Given $x\in \R^3\setminus\{y_0\}$, the direction from $x$ to $y_0$ is denoted by $\hat x_{y_0}:=
({x-y_0})/{|x-y_0|}.$
Throughout the paper, $I$ denotes an identity operator in various spaces, and the constants may be different at different places.

\section{Main theorems} \label{sec:3}

\subsection{Existence and asymptotics of scattering resonances near the real axis}\label{sec:3.1}

In this subsection, the existence of scattering resonances of $H^e(\tau;\Omega)$ near the real axis is stated in Theorem \ref{th:0}, and their asymptotic expansions at the wavelength and subwavelength scales are provided in Theorems \ref{th:1} and \ref{th:2}.

\begin{theorem} \label{th:0}
Let $\Lambda$ be as specified in \textnormal{(N\ref{No:1})}. Assume that $\tau > 0$ and $ \mathcal I \subset \R$ is a bounded interval. Then, there exists a constant $\delta_{\mathcal I}> 0$ independent of $\tau$, and a constant $\tau_{\delta_{\mathcal I}} > 0$ dependent on $\delta_{\mathcal I}$, such that when $\tau \in (0, \tau_{\delta_{\mathcal I}})$, for each $z_0 \in \Lambda \cap \mathcal I$, the open disk $B_{\CC, \delta_{\mathcal I}}(z_0) \subset \mathbb C$ contains scattering resonances of the Hamiltonian $H^e(\tau;\Omega)$, and their number satisfies the following estimates
\begin{align}
 &\sum |\{z : z \; \textrm{is a scattering resonance}\; \textrm{in}\; B_{\CC, \delta_{\mathcal I}}(0)\}| \le 12, \label{eq:109} \\ 
&\mathrm{and}\;\sum |\{z: z\; \textrm{is a scattering resonance}\; \textrm{in}\; B_{\CC, \delta_{\mathcal I}}(z_0)\}| \le n(z_0), \quad \mathrm{if}\; z_0 \in  (\Lambda \cap \mathcal I) \backslash \{0\}. \label{eq:110}
\end{align}

Furthermore, all scattering resonances in $B_{\CC,\delta_{\mathcal I}}(z_0)$  converge to $z_0$ as $\tau\rightarrow 0$.

Moreover, no scattering resonances are located inside 
\begin{align}
\left\{z\in \mathbb C: {\rm{Re}}(z) \in \mathcal I,\; |{\rm{Im}}(z)| \le\delta_{\mathcal I}\right\} \backslash \bigcup_{z_0\in \mathcal I \cap \Lambda_N} B_{\mathbb C, \delta_{\mathcal I}}(z_0). \notag
\end{align}
\end{theorem}

\begin{theorem} \label{th:1}
Let $\tau>0$. 
Assume that $z_0 \in \Lambda \backslash \{0\}$ with $\Lambda$ be as in \textnormal{(N\ref{No:1})}. Let $M^{(1)}(z_0)$ be as in \eqref{eq:30}. The scattering resonances of the Hamiltonian $H^e(\tau;\Omega)$ in the neighborhood of $z_0$ satisfy one of the following asymptotic expansions 
\begin{align} \label{eq:58}
z^{(l)}(\tau) = 
z_0 - \tau \frac{1}{2\rho_1z_0}{\kappa^{(l)}}(z_0)
+ o(\tau), \quad l\in\{1,\ldots,n(z_0)\}, \quad \mathrm{as}\; \tau \rightarrow 0,
\end{align}
where $\kappa^{(l)}(z_0)$ $(l \in\{1,\ldots,n(z_0)\})$ are the eigenvalues of the matrix $M^{(1)}{(z_0)}$, counted with algebraic multiplicity. In particular,
the imaginary parts of these eigenvalues are strictly positive, that is,
\begin{align} \label{eq:194}
\mathrm{Im}\big({\kappa^{(l)}}(z_0)\big) > 0, \quad l\in\{1,\ldots,n(z_0)\}.
\end{align}
\end{theorem}

\begin{theorem} \label{th:2}
Let $\tau >0$, and let $\mathcal L^{(0)}$ be as in \eqref{eq:1}. The scattering resonances of the Hamiltonian $H^e(\tau;\Omega)$ in the neighborhood of $0$ satisfy one of the following asymptotic expansions
\begin{align}
z(\tau;\kappa,\kappa') = 
\pm\sqrt{\tau} \sqrt{-\frac{\kappa}{\rho_1}} + i \frac{\tau}{2\rho_1} \kappa'
+ o(\tau), \quad \kappa \in \mathcal L^{(0)}\; \mathrm{and}\; \kappa' \in \mathcal L^{(1)}(\kappa) \quad \mathrm{as}\; \tau \rightarrow 0, \label{eq:59}
\end{align}
where for $\kappa \in \mathcal L^{(0)}$, the set $\mathcal L^{(1)}(\kappa)$ is defined by \eqref{eq:2}. In particular, we have 
\begin{align}
    \kappa' \le 0, \quad \kappa' \in \mathcal L^{(1)}(\kappa). \label{eq:147}
\end{align}

Furthermore, if the admissible set $\mathcal E$ defined in \eqref{eq:3} is not empty, then for $\kappa \in \mathcal L^{(0)} \cap \mathcal E$, we have 
\begin{align} \label{eq:112} 
z(\tau;\kappa,0) &= \pm\sqrt{\tau} \sqrt{-\frac{\kappa}{\rho_1}} \pm \frac{\kappa''}{2\sqrt{-{\rho_1}\kappa}}\tau^{\frac 32} + i\frac{\kappa'''}{2\rho_1}\tau^2 + o(\tau^2),\notag \\
&\qquad \qquad\qquad \quad \qquad \qquad \;\;\; \kappa'' \in \mathcal L^{(2)}(\kappa;0),\; k'''\in \mathcal L^{(3)}{(\kappa;0,k'')}  \quad \mathrm{as}\; \tau \rightarrow 0, \quad 
\end{align} 
where the set $\mathcal L^{(2)}(\kappa;0)$ is defined by \eqref{eq:4}, and for $\kappa'' \in \mathcal L^{(2)}(\kappa;0)$, the set $\mathcal L^{(3)}(\kappa;0;\kappa'')$ is defined by \eqref{eq:5}. In particular, we have 
\begin{align}
\kappa''' < 0, \quad k'''\in \mathcal L^{(3)}{(\kappa;0,k'')} \;  \mathrm{with}\;  \kappa \in \mathcal L^{(0)} \cap \mathcal E \;\mathrm{and}\; \kappa'' \in \mathcal L^{(2)}(\kappa;0). \label{eq:236}
\end{align}
\end{theorem}


\begin{figure}[t]
  \centering
  \includegraphics[width=0.82\textwidth]{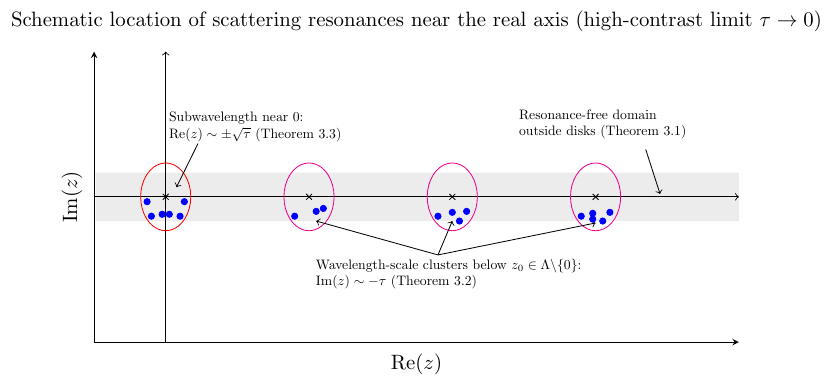}
  \caption{Schematic location of scattering resonances in the high-contrast limit $\tau\to0$.
  Theorem~\ref{th:0} yields that resonances are confined to small disks around the interior Neumann spectrum and that a resonance-free domain persists away from these disks.
  For each nonzero Neumann eigenvalue $z_0\in\Lambda\setminus\{0\}$, Theorem~\ref{th:1} shows that a finite cluster of resonances lies just below $z_0$, with $\mathrm{Im}(z^{(l)}(\tau))\sim-\tau$.
  Near $0$, Theorem~\ref{th:2} yields subwavelength resonances with $|\mathrm{Re}(z(\tau))|\sim \sqrt{\tau}$ and $\mathrm{Im}(z(\tau))\sim -\tau$ generically, while the exceptional $\tau^2$-regime may also occur.}
  \label{fig:resonances-overview}
\end{figure}

\begin{figure}[t]
  \centering
  \includegraphics[width=0.95\textwidth]{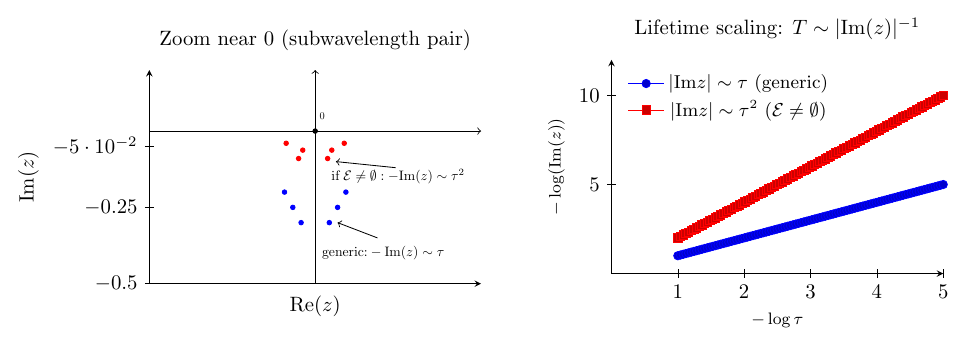}
  \caption{Role of the admissible set $\mathcal{E}$ in the subwavelength regime (Theorem~\ref{th:2}).
  Left: for fixed small $\tau$, both regimes have $\mathrm{Re}(z(\tau)) \sim \pm \sqrt{\tau}$, but the depth below the real axis differs.
  Right: log scaling of $|\mathrm{Im}(z(\tau))|$: generically $|\mathrm{Im}{(z(\tau))}|\sim \tau$ (leading to lifetime $T\sim \tau^{-1}$), while if $\mathcal{E}\neq\emptyset$ there exist resonances with $|\mathrm{Im}( z(\tau))|\sim \tau^2$ (hence $T\sim \tau^{-2}$).}
  \label{fig:E-dichotomy}
\end{figure}

\noindent \textbf{Interpretation of Theorems~3.1--3.3 (geometry of resonance clusters and lifetimes).}
Theorem~\ref{th:0} provides a global confinement picture: for $\tau$ sufficiently small, scattering resonances can only occur in small neighborhoods of the interior Neumann spectrum $\Lambda$, and there is a resonance-free domain around the real axis away from these neighborhoods; see Figure~\ref{fig:resonances-overview}. In particular, each $z_0\in \Lambda$ generates a {cluster} of resonances converging to $z_0$ as $\tau\to0$.
\begin{enumerate}
\item For {wavelength-scale} resonances near a fixed $z_0\in\Lambda\setminus\{0\}$, Theorem~\ref{th:1} yields the first-order expansion
\[
z^{(l)}(\tau)=z_0-\frac{\tau}{2\rho_1 z_0}\,\kappa^{(l)}(z_0)+o(\tau),
\]
where $\mathrm{Im}(\kappa^{(l)}(z_0))>0$. Consequently, these resonances lie strictly in the lower half-plane and satisfy
$|\mathrm{Im}(z^{(l)}(\tau))|\sim \tau$, so the associated decay time (lifetime) scales like
$T\sim \tau^{-1}$; see Figure~\ref{fig:resonances-overview}. With the time dependence $e^{-izt}$, scattering resonances (poles of the meromorphically continued resolvent in the lower half-plane) correspond to damped oscillatory modes: $\mathrm{Re}(z)$ gives the oscillation frequency, while $-\mathrm{Im}(z)$ gives the decay rate. In particular, the associated lifetime is proportional to $1/\mathrm{Im}(z)$.

\item The {subwavelength} regime near $0$ is qualitatively different (Theorem~\ref{th:2}): resonances satisfy
$\mathrm{Re}(z(\tau))\sim \pm\sqrt{\tau}$ and come in $\pm$ pairs, reflecting the $z^2$-dependence of the transmission problem. 
Moreover, the imaginary part exhibits a sharp dichotomy.
Generically one has $|\mathrm{Im} (z(\tau)|\sim \tau$, but if the admissible set $\mathcal{E}$ is nonempty, then there exist subwavelength resonances with
$|\mathrm{Im}(z(\tau))| \sim \tau^2$ (hence lifetimes $T \sim \tau^{-2}$), i.e. {much longer-lived} subwavelength responses; see Figure~\ref{fig:E-dichotomy}.
In particular, Remark \ref{re:1} shows that $\mathcal{E}\neq\emptyset$ can occur for a large class of origin-symmetric resonators, guaranteeing the presence of these very long-lived subwavelength resonances.
\end{enumerate}


\subsection{Asymptotic behaviour of resolvents for resonators and microresonators}

\subsubsection{Resolvent asymptotics near wavelength-scale resonances}

In this subsection, the resolvent expansions for the resonator and the microresonator will appear in Theorems \ref{th:3} and \ref{th:5}. 

\begin{theorem} \label{th:3}
Let $z_0 \in \Lambda \backslash \{0\}$ with $\Lambda$ be as in \textnormal{(N\ref{No:1})}, and let $ 0<\tau\ll 1$ and $\omega \in \R$. Let $R^{\mathrm{ex}}(z_0)$ and $M^e(\tau,\omega;z_0)$ be as in \textnormal{(N\ref{No:4})} and \textnormal{(N\ref{No:2})}, respectively.
Then, given any $f \in \mathbf L^2_{{\rm{comp}}}(\R^3)$, we have 
\begin{align}
R_{H^e(\tau;\Omega)}(\omega) f = R^{\mathrm{ex}}(z_0)b_{\tau}^{(1)}(\omega,f;z_0) + R^{(\mathrm{Rem})}_{H^e(\tau;\Omega)}(\omega;z_0)f, \notag
\end{align}
where $b_\tau^{(1)}(\omega,f;z_0)$ is specified by 
\begin{align}
&b^{(1)}_\tau(\omega,f;z_0):= \sum^{n(z_0)}_{j=1}\sum^{n(z_0)}_{l=1}e^{(l)}(z_0)\left[\left(M^e(\tau, \omega;z_0)\right)^{-1}\right]_{lj}\rho_1 [\Pi(z_0) f]_j, \notag
\end{align}
and $R^{(\mathrm{Rem})}_{{H^e(\tau;\Omega)}}(\omega;z_0)$ satisfies
\begin{align*}
\left\|R^{(\mathrm{Rem})}_{H^e(\tau;\Omega)}(\omega;z_0) \right\|_{\mathbf L^2(\R^3)\rightarrow\mathbf L^2_{-\beta}(\R^3)} & \le C\frac{\tau + |\omega - z_0| }{\|M^e(\tau, \omega;z_0)\|}, \quad \mathrm{as}\; \omega \rightarrow z_0 \;\mathrm{and}\; \tau \rightarrow 0.
\end{align*} 
\end{theorem}

\begin{theorem} \label{th:5}
Let $z_0 \in \Lambda \backslash \{0\}$ with $\Lambda$ be as in \textnormal{(N\ref{No:1})}. Assume that $0<\tau\ll 1$, $ 0<\vep\ll 1$ and $\omega = \widetilde \omega/\vep$ with $\widetilde \omega \in \R$. 
Given a fixed point $y_0 \in \R^3$, for any $f \in \mathbf L^2_{{\rm{comp}}}(\R^3) \cap\mathbf H^2(B_1(y_0))$, we have the asymptotic expansion
\begin{align}
&\left(R_{H^e(\tau;\Omega_\vep(y_0))}(\omega) f\right)(x) - \left(R_{0}(\omega)f\right)(x) \notag\\
& = \vep \sum_{\sigma \in \{p,s\}}\frac{e^{\frac{i\omega}{c_{\sigma,0}}|x-y_0|}}{|x-y_0|} e^{\frac{i z_0 \hat x_{y_0} \cdot y_0}{c_{\sigma,0}}}\left(\mathcal F_{\sigma}^{\infty}(z_0) \gamma\left[-(I - P(z_0)) R^\infty(z_0) f + b_{\tau,\vep}(\widetilde\omega,f;z_0)\right]\right)(\hat x_{y_0}) \notag \\
&\qquad+ \left(R_{{H^e(\tau;\Omega_\vep(y_0))}}^{\mathrm{Res}}(\omega) f\right)(x) \notag
\end{align}
with $R_{{H^e(\tau;\Omega_\vep(y_0))}}^{\mathrm{Res}}(\omega)$ satisfying
\begin{align*}
\left\|R_{{H^e(\tau;\Omega_\vep(y_0))}}^{\mathrm{Res}}(\omega) \right\|_{\mathbf L^2_{-\beta}(\R^3)} &\le 
C \vep \bigg[{\max\left(\tau,\vep^{1/2},|\widetilde \omega-z_0|\right)} + \frac{\vep^{5/2}}{\|M^e(\tau, \widetilde \omega;z_0)\|}\bigg]\notag\\
&\left[\|f\|_{\mathbf L^2(\R^3)} + \|f\|_{\mathbf H^2(B_1(y_0))}\right] , \quad \mathrm{as}\; \widetilde w \rightarrow z_0,\; \tau\rightarrow 0\; \mathrm{and}\; \vep\rightarrow 0, 
\end{align*} 
where $R^{\infty}(z_0) f$ and $b_{\tau,\vep}(\widetilde \omega,f;z_0)$ are specified by 
\begin{align*}
&\left(R^\infty(z_0)f\right)(x) := \sum_{\sigma \in \{p,s\}}\int_{\R^3}\frac{e^{\frac{iz_0}{\vep c_{\sigma,0}}|y-y_0|}}{|y-y_0|} (G^{(0)})_\sigma^{\infty}(\hat y_{y_0}, x - y_0 ; z_0) f(y)dy
\end{align*}
and
\begin{align*}
& b_{\tau,\vep}(\widetilde \omega,f;z_0):= \sum^{n(z_0)}_{l=1}\sum^{n(z_0)}_{j=1} e^{(l)}(z_0)\left[\left(M^e(\tau,\widetilde w;z_0)\right)^{-1}\right]_{lj}\\
&\bigg(-\rho_1\vep^2[\Pi(z_0) f(y_0)]_j - 2z_0(\widetilde w - z_0)[\Pi(z_0)R^{\infty}(z_0)f]_j\\
&- \tau \left\langle \left(\mathcal N(z_0)\gamma (I - P(z_0))R^{\infty}(z_0) f + \partial_{\nu,0} R^{\infty}(z_0) f\right), e^{(j)}(z_0)\right\rangle_{\Gamma}\bigg),
\end{align*}
respectively, and $C$ is a positive constant independent of $f$, $\vep$ and $\tau$.
\end{theorem}

\noindent {\textbf{Consequences of Theorems~\ref{th:3}--\ref{th:5}.}} Theorems~\ref{th:3}--\ref{th:5} convert the wavelength-scale spectral information from Section~\ref{sec:3.1}
into explicit {resolvent asymptotics} in both resonator and microresonator regime. 
\begin{enumerate}
\item For a fixed size resonator $\Omega$, Theorem \ref{th:3} shows that in this regime
the resolvent of ${H^e(\tau;\Omega)}$ admits a leading order finite-dimensional effective description: inside $\Omega$ the field is expanded in Nemuann eigenmodes with coeffcients  governed by $(M^e)^{-1}$, while outside it is coupled to an outgoing exterior field determined by its boundary trace. In particular, the structure of $M^e$ implies that the resolvent is amplified when $\omega$ approaches wavelength-scale resonance. 

\item For a microresonator $\Omega_\vep(y_0)$, the leading term has a structure of an elastic point scatter at $y_0$: it is the superposition of highly oscillatory $p-$ wave and $s-$ wave, each modulated by corresponding ($p-$/$s-$) far-field pattern of the exterior elastic scattering problem on the reference domain $\Omega$. In particular, the effective response is generally anisotropic, in the sense that it depends on the observation directions. 
This explicit connection to $p-$ and $s-$ far-field patterns suggests potential applications to inverse problems, at least in principle, to extract geometric information from single-mode ($p-$ only or $s-$ only) data.
\end{enumerate}

\subsubsection{Resolvent asymptotics near subwavelength resonances}

In this subsection, the contribution of subwavelength resonances to the resolvent is isolated via projections onto the associated principal enhancement spaces. The definitions of the projections (see \eqref{eq:89} and \eqref{eq:90}) and related notations are as in \textnormal{(N\ref{No:3})}. The resolvent asymptotics near the subwavelength resonances for the resonator are stated in the next theorem, while those for the microresonator are given in Theorem \ref{th:6}.

\begin{theorem} \label{th:4}
Let $0<\tau\ll 1$ and $0<\omega \ll 1$, and let $ R^{\mathrm{ex}}(\omega)$ be as in \textnormal{(N\ref{No:4})}. Given any $f \in \mathbf L^2_{{\rm{comp}}}(\R^3)$, we have the following arguments regarding the asymptotic behaviours of $R_{H^e(\tau;\Omega)}(\omega) f$.
\begin{enumerate}[(a)]
\item \label{h1}
Let $\kappa \in \mathcal L^{(0)} \backslash \mathcal E$. Assume that $\omega = \sqrt \tau \widetilde w$. We have 
\begin{align}
R_{H^e(\tau;\Omega)}(\omega) f  = R^{\mathrm{ex}}(\omega) b^{(1)}_\tau (\omega,f;0) + R^{(\mathrm{Rem})}_{H^e(\tau;\Omega)}(\omega)f, \notag
\end{align}
where $b^{(1)}_\tau(\omega,f;0)$ is defined by 
\begin{align}
b^{(1)}_\tau(\omega,f;0):=\begin{bmatrix}
e^{(1)}(0),\ldots,e^{(6)}(0)
\end{bmatrix}
\sum_{\kappa' \in \mathcal L^{(1)}(\kappa)} \frac{-1}{\tau}\frac{P_{E_{\mathrm p}(\kappa;\kappa')}\rho_1 \Pi(0)f}{\rho_1\widetilde w^2 +\kappa - i \widetilde w \kappa' \sqrt \tau} \notag
\end{align}
and $R^{(\mathrm{Rem})}_{H^e(\tau;\Omega)}(\omega)$ satisfies
\begin{align}
&\left\|R^{(\mathrm{Rem})}_{H^e(\tau;\Omega)}(\omega) \right\|_{\mathbf L^2(\R^3)\rightarrow\mathbf L^2_{-\beta}(\R^3)} \notag \\
&\le C \frac{|\rho_1\widetilde w^2 + k| + \tau^{1/2}}{\tau\sqrt{|\rho_1\widetilde w^2 + \kappa|^2 - i \sqrt \tau \widetilde w \min_{\kappa' \in \mathcal L(\kappa) }|\kappa'|}} \quad \mathrm{as}\; \widetilde w \rightarrow \pm \sqrt{-\frac{\kappa}{\rho_1}}\;\mathrm{and}\; \tau \rightarrow 0. \notag
\end{align}

\item \label{h2}
Let $\kappa \in \mathcal L^{(0)} \cap \mathcal E$ and $\kappa'' \in \mathcal L^{(2)}(\kappa;0)$. Assume that $\omega = \pm\sqrt\tau\sqrt\frac{-\kappa}{\rho_1} + \tau^{3/2} \widetilde w$. We have 
\begin{align}
R_{H^e(\tau;\Omega)}(\omega) f  = R^{\mathrm{ex}}(\omega) b^{(2)}_\tau (\omega,f;0) + R^{(\mathrm{Rem})}_{H^e(\tau;\Omega)}(\omega)f, \notag 
\end{align}
where 
$b^{(2)}_\tau(\omega,f;0)$ is defined by 
\begin{align*}
&b^{(2)}_\tau(\widetilde \omega,f;0):= \pm \begin{bmatrix}
e^{(1)}(0),\ldots,e^{(6)}(0)
\end{bmatrix} \frac{-1}{2\sqrt{-{\rho_1}\kappa}} \sum_{\kappa''' \in \mathcal L^{(3)}(\kappa;0,\kappa'')} \frac{1}{\tau^2}\frac{P_{E_{\mathrm p}(\kappa;0, \kappa'',\kappa''')}\rho_1 \Pi(0)f }{ \widetilde w \mp \frac{\kappa''}{2\sqrt{-{\rho_1}\kappa}} - i  \frac{1}{2\rho_1} \kappa'''\sqrt{\tau}}
\end{align*}
and $R^{(\mathrm{Rem})}_{H^e(\tau;\Omega)}(\omega)$ satisfies
\begin{align}
\left\|R^{(\mathrm{Rem})}_{H^e(\tau;\Omega)}(\omega) \right\|_{\mathbf L^2(\R^3)\rightarrow\mathbf L^2_{-\beta}(\R^3)} & \le C \frac{|\widetilde w \mp \frac{\kappa''}{2\sqrt{-{\rho_1}\kappa}}| + \tau^{1/2}}{\tau^2 \sqrt{|\widetilde w \mp \frac{\kappa''}{2\sqrt{-{\rho_1}\kappa}}| - i \frac{1}{2\rho_1}\sqrt{\tau} \min_{\kappa'''\in \mathcal L^{(3)}(\kappa;0,\kappa'')}|\kappa'''|} }\notag\\
&\qquad\qquad\qquad\qquad
\qquad\mathrm{as}\; \widetilde w \to \pm \frac{\kappa''}{2\sqrt{-{\rho_1}\kappa}} \;\mathrm{and}\; \tau \rightarrow 0. \notag
\end{align}
\end{enumerate}
Here, $C$ is a positive constant independent of $f$, $\omega$ and $\tau$.
\end{theorem}

\begin{theorem} \label{th:6}
Assume that $0 < |\omega| \ll 1$, $0<\vep< 1$ and $0<\tau\ll 1$. Given a fixed point $y_0 \in \R^3$, we have the following arguments.

\begin{enumerate}[(a)]
\item \label{m1}
Let $\kappa \in \mathcal L^{(0)} \backslash \mathcal E$. Assume that $\omega = \sqrt \tau \widetilde w/\vep$ with $\omega \in \R$. For any $f \in \mathbf L^2_{{\rm{comp}}}(\R^3) \cap\mathbf H^2(B_1(y_0))$, we have
\begin{align}
&\left(R_{H^e(\tau;\Omega_\vep(y_0))}(\omega) f\right)(x) - \left(R_{0}(\omega)f\right)(x) \notag \\
& = \vep G_0(x,y_0;\omega) \left\langle \gamma b^{(1)}_{\tau,\vep}(w, f) , \left(S_0(0)\right)^{-1} 1\right\rangle_{\Gamma}  + \left(R^{\mathrm{Res}}_{H^e(\tau;\Omega_\vep(y_0))}(\omega)f\right)(x), \label{eq:146}
\end{align}
where $b^{(1)}_{\tau,\vep}(w, f)$ is specified by 
\begin{align}
b^{(1)}_{\tau,\vep}(\omega,f):= \begin{bmatrix}
e^{(1)}(0),\ldots,e^{(6)}(0)
\end{bmatrix}\left[\sum_{\kappa' \in \mathcal L^{(1)}(\kappa)} \frac{1}{\tau}\frac{P_{E_{\mathrm p}(\kappa;\kappa')} a^{(1)}_{\tau,\vep}(\omega,f)}{\rho_1\widetilde w^2 +\kappa - i \widetilde w \kappa' \sqrt \tau}\right] \notag
\end{align}
with  $a^{(1)}_{\tau,\vep}(w, f) \in \mathbb C^6$ given by
\begin{align} 
\left[a^{(1)}_{\tau,\vep}(\omega,f)\right]_j = -\rho_1 [\Pi(0)(\vep^2 f(y_0) + \vep^2\omega^2 (R_0(\omega)f)(y_0))]_j,\quad j\in \{1,\ldots,6\}, \label{eq:238}
\end{align}
and the remainder term $R^{\mathrm{Res}}_{H^e(\tau;\Omega_\vep(y_0))}(\omega)$ satisfies
\begin{align}
&\left\|R^{\mathrm{Res}}_{H^e(\tau;\Omega_\vep(y_0))}(\omega)f\right\|_{\mathbf L^2_{-\beta}(\R^3)} \le C\vep\frac{\max(\vep^2,\tau)(|\rho_1\widetilde w^2 + k| + \tau^{1/2}) + \max\left(\tau \vep^{1/2},\tau^{3/2},\vep^{5/2}\right)}{\tau\sqrt{|\rho_1\widetilde w^2 + \kappa|^2 - i \sqrt \tau \widetilde w \min_{\kappa' \in \mathcal L(\kappa) }|\kappa'|}} \notag\\
&\quad\qquad\qquad\left[\|f\|_{\mathbf L^2(\R^3)} + \|f\|_{\mathbf H^2(B_1(y_0))}\right]\quad \mathrm{as}\;  \widetilde w \rightarrow \pm \sqrt{-\frac{\kappa}{\rho_1}}, \tau \rightarrow 0 \;\mathrm{and}\; \vep \rightarrow 0.\notag
\end{align}
Here, $C$ is a positive constant independent of $f$, $\vep$ and $\tau$.

\item \label{m2} Let $\kappa \in \mathcal L^{(0)} \cap \mathcal E$ and $\kappa'' \in \mathcal L^{(2)}(\kappa;0)$. Let $r>0$ be arbitrary. 
Assume that $\omega = \left(\pm\sqrt\tau\sqrt{-\kappa}/{\rho_1} + \tau^{3/2} \widetilde w\right)/\vep$ and that $\tau = O(\vep^2)$. Suppose that 
\begin{align} \label{eq:233}
\mathrm{Ker}\left(\kappa I - M^{(1)}(0)\right) \subset \mathrm{Ker}\left(M^{(2)}(0)\right).
\end{align}
For any $f \in \mathbf L^2_{{\rm{comp}}}(\R^3) \cap\mathbf H^3(B_1(y_0))$, we have 
\begin{align}
&\left(R_{H^e(\tau;\Omega_\vep(y_0))}(\omega) f\right)(x) - \left(R_{0}(\omega)f\right)(x) \notag \\
& =  \vep \int_{\Gamma}(y-y_0)\nabla_{y}G(x,y_{0};\omega)\left[\left(S_0(0)\right)^{-1}(\gamma b^{(2)}_{\tau,\vep}(w, f))\right](y) dS(y) \notag\\ 
&+\left(R^{\mathrm{Res}}_{H^e(\tau;\Omega_\vep(y_0))}(\omega)f\right)(x), \label{eq:3.14}
\end{align}
where $b^{(2)}_{\tau,\vep}(w, f)$ is specified by 
\begin{align*}
b^{(2)}_{\tau,\vep}(w, f)&:=  \pm\begin{bmatrix}
e^{(1)}(0),\ldots,e^{(6)}(0)
\end{bmatrix} \frac{1}{2\sqrt{-{\rho_1}\kappa}} \sum_{\kappa''' \in \mathcal L^{(3)}(\kappa;0,\kappa'')} \frac{1}{\tau^2}\frac{P_{E_{\mathrm p}(\kappa;0,\kappa'',\kappa''')} a^{(2)}_{\tau,\vep}(w, f)}{\widetilde w \mp \frac{\kappa''}{2\sqrt{-{\rho_1}\kappa}} - i  \frac{1}{2\rho_1} \kappa'''\sqrt{\tau}}
\end{align*}
with $a^{(2)}_{\tau,\vep}(w, f) \in \mathbb C^6$ given by
\begin{align} \label{eq:239}
&\left[a^{(2)}_{\tau,\vep}(w, f)\right]_j:= -\rho_1\left[\Pi(0)(\vep^3\nabla f(y_0) + \vep^3\omega^2 \nabla (R_0(\omega)f)(y_0))\cdot (x-y_0))\right]_j\notag \\
& + \int_{\Gamma} \vep^3(I-P(0))\left(\left[\nabla R_0\left(\omega\right)\right](y_0) \cdot (x-y_0)\right) \left[\left(S_0(0)\right)^{-1}e^{(j)}(x)\right] dS(x),\quad j\in\{1,\ldots,6\},
\end{align}
and the remainder term $R^{\mathrm{Res}}_{H^e(\tau;\Omega_\vep(y_0))}(\omega)$ satisfies
\begin{align}
&\left\|R^{\mathrm{Res}}_{H^e(\tau;\Omega_\vep(y_0))}(\omega)\right\|_{\mathbf L^2_{-\beta}(\R^3 \backslash B_r(y_0))} \le C_r \vep \frac{\vep^3(|\widetilde w \mp \frac{\kappa''}{2\sqrt{-{\rho_1}\kappa}}| + \tau^{1/2}) + \vep^{7/2}}{\tau^2 \sqrt{|\widetilde w \mp \frac{\kappa''}{2\sqrt{-{\rho_1}\kappa}}| - i \frac{1}{2\rho_1}\sqrt{\tau} \min_{\kappa'''\in \mathcal L^{(3)}(\kappa;0,\kappa'')}|\kappa'''|} } \notag \\ 
&\qquad\qquad\quad\left[\|f\|_{\mathbf H^1(\R^3)} + \|f\|_{\mathbf H^3(B_1(y_0))}\right]\quad\mathrm{as}\; \widetilde w \to \pm \frac{\kappa''}{2\sqrt{-{\rho_1}\kappa}},\;\tau \rightarrow 0 \;\mathrm{and}\; \vep \rightarrow 0. \notag
\end{align}
Here, $C_r$ is a positive constant independent of $f$, $\vep$ and $\tau$.
\end{enumerate}
\end{theorem}

\noindent \textbf{Consequences of Theorems~\ref{th:4}--\ref{th:6}} Theorems~\ref{th:4}--\ref{th:6} translate the subwavelength spectral information from Section~\ref{sec:3.1}
into explicit \emph{resolvent asymptotics} in the regime $\omega=\Theta(\sqrt{\tau})$
(fixed resonator) and $\omega=\Theta(\sqrt{\tau}/\varepsilon)$ (microresonator).
In particular, the contribution of subwavelength resonances to the resolvent is isolated
via projections onto the associated principal enhancement spaces; the remainder terms
are controlled quantitatively and remain uniformly small away from the resonance set.
This subsection records the main consequences of these results.

\medskip
\noindent\textbf{1) Fixed resonator: finite-dimensional reduction and explicit resonant denominators.}
For the fixed-size resonator $\Omega$, Theorem~\ref{th:4} shows that in the subwavelength regime
the resolvent $R_{H^e(\tau;\Omega)}(\omega)$ is governed, to leading order, by a
\emph{finite-dimensional effective system} on the enhancement space at $0$.
More precisely, the leading part is $R^{\mathrm{ex}}(0)b^{(1)}_\tau(\omega,f;0)$ in the generic case
$\kappa\in \mathcal L^{(0)}\setminus \mathcal{E}$, and $R^{\mathrm{ex}}(\omega)b^{(2)}_\tau(\omega,f;0)$ in the
exceptional case $\kappa\in \mathcal L^{(0)}\cap \mathcal{E}$ (with $\kappa''\in \mathcal L^{(2)}(\kappa;0)$),
while the remainder $R^{(\mathrm{Rem})}_{H^e(\tau;\Omega)}(\omega)$ satisfies explicit bounds.
In particular, the scalar factors appearing in $b^{(1)}_\tau$ and $b^{(2)}_\tau$ contain
\emph{resonant denominators} whose smallness is equivalent to $\omega$ being close to the
subwavelength scattering resonances described in Theorem~\ref{th:2}. Hence Theorem~\ref{th:4} provides a
direct mechanism for resolvent amplification (and thus for large fields) near those resonances.

\medskip
\noindent\textbf{2) Microresonator: an effective point-scatterer at $y_0$ (monopole vs.\ dipole).}
For the scaled inclusion $\Omega_\varepsilon(y_0)$, Theorem~\ref{th:6} yields a much more explicit
physical picture: in the exterior region the difference
\[
R_{H^e(\tau;\Omega_\varepsilon(y_0))}(\omega)f - R_0(\omega)f
\]
is asymptotically equivalent to a \emph{point interaction supported at $y_0$}, with the type
of point-scatterer determined by the admissible set $\mathcal{E}$:
\begin{itemize}
\item If $\kappa\in \mathcal L^{(0)}\setminus \mathcal{E}$ and $\omega=\sqrt{\tau}\,w_e/\varepsilon$, the
leading term is proportional to $G_0(x,y_0;\omega)$ (formula \eqref{eq:146}), which is
\emph{monopole-type} and therefore produces a dominant field that is isotropic. The corresponding amplitude is frequency dependent and becomes large
when the resonant denominators in $b^{(1)}_{\tau,\varepsilon}$ approach zero.
\item If $\kappa\in \mathcal L^{(0)}\cap \mathcal{E}$ (with $\kappa''\in L^{(2)}(\kappa;0)$) and the additional
scaling/structural assumptions of Theorem~\ref{th:6}(b) hold, then the leading term involves
$\nabla_y G_0(x,y_0;\omega)$ (formula \eqref{eq:3.14}), hence is \emph{dipole-type} and generically
anisotropic. In this regime the excitation depends on first-order Taylor data at $y_0$
(through $a^{(2)}_{\tau,\varepsilon}$), and the scattered field may exhibit a markedly
direction-dependent dominant part.
\end{itemize}

\medskip
\noindent\textbf{3) Lifetimes and the role of $\mathcal{E}$: $\Theta(\tau)$ vs.\ $\Theta(\tau^2)$ widths.}
A key consequence, already visible at the spectral level (Theorem~\ref{th:2}), becomes transparent
in the resolvent expansions: the width of subwavelength resonances (distance to the real axis)
controls the decay time of the corresponding resonant responses.
Generically one has $|\mathrm{Im}(z(\tau))|=\Theta(\tau)$ near $0$, while if $\mathcal{E}\neq\emptyset$
there exist subwavelength resonances with $|\mathrm{Im}(z(\tau))|=\Theta(\tau^2)$ (and thus much longer
lifetimes). In particular, for origin-symmetric obstacles such $\Theta(\tau^2)$ resonances do
occur (Remark~\ref{re:1}), so Theorem~\ref{th:6}\eqref{m2} predicts long-lived, dipole-dominated
subwavelength microresonator responses within this symmetry class.

\medskip
\noindent\textbf{4) Hybrid monopole--dipole dimers (design principle).}
Theorem~\ref{th:6} suggests a genuinely \emph{heterogeneous} dimer mechanism: consider two close microresonators
$D_1^\varepsilon(y_1)$ and $D_2^\varepsilon(y_2)$ such that $D_1^\varepsilon$ is tuned to a subwavelength branch with
$\kappa\in \mathcal L^{(0)}\setminus\mathcal E$ (so the leading exterior response is \emph{monopole-type}, cf.\ \eqref{eq:146}),
whereas $D_2^\varepsilon$ is designed so that $\kappa\in\mathcal E$ and \eqref{eq:233} holds (so the leading response is
\emph{dipole-type}, cf.\ \eqref{eq:3.14}). If, in addition, the two resonators are adjusted so that their leading-order
subwavelength resonances coincide, then multiple scattering between $D_1^\varepsilon$ and
$D_2^\varepsilon$ yields a coupled effective system whose solution produces a field that is, already at leading
order, a \emph{linear combination} of monopole and dipole contributions with resonantly enhanced coefficients.
This differs from homogeneous dimer scenarios where a dipolar response is \emph{induced} by near-field hybridization
of two otherwise monopolar resonators (cf.\ the acoustic double-negative paradigm
\cite{AmmariFitzpatrickLeeYuZhang2019DoubleNegative}; rather, it parallels the hybrid dielectric--plasmonic dimer mechanism in
electromagnetism, where two distinct particles (electric-driven vs.\ magnetic-driven) are tuned to share a common
resonant frequency, enabling simultaneous polarization of both channels \cite{CaoGhandricheSini2025HybridDimer}.
In elasticity, such a controlled superposition of monopolar and dipolar subwavelength channels is a natural way
toward simultaneously negative effective parameters near resonance (cf.
\cite{WuLaiZhang2011} for engineered elastic negative-parameters).



{\color{HW1} \section{Scattering resonances} \label{sec:4}

In this section, we first establish the meromorphic continuation of the resolvent of the high-contrast Hamiltonian $H^e(\tau;\Omega)$, as defined in \textnormal{(N\ref{No:5})} . We then define the resonances as a pole of this continuation and provide several characterizations of scattering resonances for the Hamiltonian $H^e(\tau;\Omega)$. We postpone the proofs of all propositions in this section to Appendix \ref{sec:A}.

\subsection{\texorpdfstring{Meromorphic properties of $H^e({\tau;\Omega})$}{}} \label{sec:4.1}
We begin with the following fact: the $(k,l)$ entry of $G^{(0)}(x,y\; z)$ is given by 
\begin{align*}
\left(G^{(0)}\right)_{kl}(x,y;z) = \frac{\delta_{k,l}}{4\pi \mu_0} \frac{e^{i\frac{z}{c_{s,0}}|x-y|}}{|x-y|} + \frac{1}{4\pi \rho_0 z^2} \partial_k \partial_l\left(\frac{e^{i\frac{z}{c_{s,0}}|x-y|}}{|x-y|} - \frac{e^{i\frac{z}{c_{p,0}}|x-y|}}{|x-y|}\right).
\end{align*}
Building upon the expression of the free resolvent $R_0$ (see \eqref{eq:245}), it can be seen that $R_0(z)$ admits an analytic continuation from $\mathbb C_+$ to $\mathbb C$ as a bounded operator mapping from $\mathbf L_{\mathrm{comp}}^2(\R^3)$ to $\mathbf L_{\mathrm{comp}}^2(\R^3)$.
It should be noted that the $(k,l)$ entry of $G^0(x,y;0)$ is given by 
\begin{align}
\left(G^{(0)}\right)_{kl}(x,y;0) &= \frac{1}{8\pi \mu_0}\left[\frac {\delta_{k,l}} {|x-y|} + \frac {(x_k-y_k)(x_l-y_l)}{|x-y|^3}\right] \notag \\
&+\frac{1}{8\pi (\lambda_0 + 2 \mu_0)} \left[\frac {\delta_{k,l}} {|x-y|} - \frac {(x_k-y_k)(x_l-y_l)}{|x-y|^3}\right]. \notag 
\end{align}

With the aid of the meromorphicity of the free resolvent and the coincidence of the Hamiltonian $H^e(\tau;\Omega)$ and the background Lam\'e operator outside the inclusion (in the spirit of the black-box framework for the Laplacian; see \cite{DM}), we arrive at the next result.

\begin{pro}\label{pro:1}
For each $\tau>0$, the operator $R_{H^e(\tau;\Omega)}(z)$ extends to a meromorphic family of operators for $z \in \CC$: $R_{H^e(\tau;\Omega)}(z):\mathbf L_{\mathrm{comp}}^2(\R^3)\to \mathbf D_{\mathrm{loc}}(H^e(\tau;\Omega))$. Here, the space $\mathbf D_{\mathrm{loc}}(H^e(\tau;\Omega))$ is defined as follows: 
\begin{align*}
\mathbf D_{\mathrm{loc}}(H^e(\tau;\Omega))&:= \bigg\{v\in \mathbf L(\R^3): \mathrm{for}\; r>0\; \mathrm{such}\;\mathrm{that}\; \overline \Omega \subset B_r(0),\\
&v|_{\R^3\backslash B_{r}(0)} \in \mathbf L^2_{\mathrm{loc}}(\R^3)\;\mathrm{and}\; \chi v \in \mathbf D(H^e(\tau;\Omega)\;\mathrm{if}\; \chi \in C^{\infty}_c(\R^3)\;\mathrm{with}\; \chi|_{B_{r}(0)} = 1\bigg\}.
\end{align*}
\end{pro} 

Proposition \ref{pro:1} leads to the following definition of scattering resonances.

\begin{definition}
Let $\tau>0$. We define the scattering resonance of the Hamiltonian $H^e(\tau;\Omega)$ as a pole of the resolvent $R_{H^e(\tau;\Omega)}$.
\end{definition}

\begin{remark}
For the exterior Neumann obstacle problem in linear elasticity, the meromorphic continuation of the resolvent (and the associated resonance framework) has been established in \cite{VP-95,VP-96}. For the high-contrast
transmission setting considered here, we could not locate an explicit treatment in the literature. For completeness, we therefore include a proof of the corresponding meromorphic properties; see Section \ref{sec:A1}.    
\end{remark}

\subsection{Alternative characterizations of scattering resonances} \label{sec:4.2}

Based on the meromorphic properties of the resolvent $R_{H^e(\tau;\Omega)}(z)$ established in Proposition \ref{pro:1}, and using that $H^e(\tau;\Omega)$ agrees with $-\rho_0^{-1}L_{\lambda_0,\mu_0}$, we obtain a first characterization of scattering resonances.

\begin{pro}\label{le:12}
Let $\tau > 0$. $z_\tau \in \mathbb C \backslash \{0\}$ is the scattering resonance of the Hamiltonian $H^e(\tau;\Omega)$ if and only if there exists nontrivial $v_{z_\tau} \in D_{\mathrm{loc}}(H^e(\tau;\Omega))$ satisfying
\begin{align}
& L_{\lambda_0,\mu_0} v_{z_\tau} + z_\tau^2 \rho_0  v_{z_\tau}= 0 \qquad
\quad\quad\quad\quad\quad\quad\quad\;\; \mathrm{in}\; \R^3 \backslash \overline \Omega, \label{eq:228}\\
& L_{\lambda_1,\mu_1} v_{z_\tau} + z_\tau^2 {\rho_1} v_{z_\tau}= 0 
\quad\qquad\quad\quad\quad\quad\;\;\;\;\;\;\;\;\;\;\mathrm{in}\; \Omega,\\
&\gamma_+ v_{z_\tau} = \gamma_- v_{z_\tau}, \quad \partial^+_{\nu,0} v_{z_\tau} = \frac{1}\tau\partial^-_{\nu,1} v_{z_\tau} \quad\quad\quad\quad\;\mathrm{on}\; \Gamma,\\
& v_{z_\tau}\; \mathrm{is}\; z_\tau -\;\mathrm{outgoing}. \label{eq:229}
\end{align}
Here, for $z \in \mathbb C \backslash \{0\}$, we say that $v$ is $z-$ outgoing if there exists $g\in \mathbf L_{\mathrm{comp}}^2(\R^3)$ and $r>0$ such that $u = R_0(z) g$ outside $B_r(0)$.
\end{pro}


Proposition \ref{le:12} yields a second characterization, reducing the resonance condition to the existence of nontrivial solution to a suitable boundary value problem.

\begin{pro} \label{pro:2}
Let $z_\tau \in \mathbb C \backslash \{0\}$ is the scattering resonance of the Hamiltonian $H^e(\tau;\Omega)$. Assume that $z_\tau \in \mathbb C \backslash \{0\}$ is not a zero of $S_{0}(z)$. Then $z_\tau$ is the scattering resonance of the Hamiltonian $H^e(\tau;\Omega)$ if and only if there exists nontrivial $v_{z_\tau} \in \mathbf H^1(\Omega)$ such that 
\begin{align}
&L_{\lambda_{1}, \mu_1} v_{z_\tau} + z_\tau^2\rho_1 v_{z_\tau} = 0 \qquad \mathrm{in}\; \Omega, \label{eq:96}\\
& \partial_{\nu,1} v_{z_\tau} = \tau\mathcal N(z_\tau) \gamma v_{z_\tau} \qquad\quad\mathrm{on}\; \Gamma. \label{eq:97}
\end{align}
\end{pro}

The above boundary-value formulation is equivalent to a characterization in terms of the zeros of an associated boundary integral equation, which is well suited to a Fredholm analytic argument.

\begin{pro}\label{pro:3}
Under the same assumption of Proposition \ref{pro:2}, $z_\tau$ is the scattering resonance of the Hamiltonian $H^e(\tau;\Omega)$ if and only if $z_\tau$ is a non-injective point of the boundary integral operator $1/2 I + K_{1}(z)- \tau S_{1}(z) \mathcal N(z)$, mapping from $\mathbf H^{1/2}(\Gamma)$ to $\mathbf H^{1/2}(\Gamma)$.
\end{pro}

Proposition \ref{pro:2} indicates that, in the high-contrast limit, scattering resonances are expected to occur near the elements in $\Lambda$. Accordingly, we introduce a localized characterization (statement \eqref{a2} of Proposition \ref{pro:4}): in a neighborhood of each $z_0 \in \Lambda$, scattering resonances are described by the zeros of an explicit finite-dimensional matrix.

\begin{pro}\label{pro:4}
Let $z_0 \in \Lambda$ with $\Lambda$ be as in \textnormal{(N\ref{No:1})} and let two sesquilinear forms $J^{\rm{dom}}_{z_0,\tau}$ and $J_{\tau}$ be as in \textnormal{(N\ref{No:6})}. 

\begin{enumerate}[(a)]
\item \label{a1} We can find $\delta^{(1)}_{z_0}, \delta_{z_0}^{(2)} \in \R_+$ such that when $\tau \in (0,\delta^{(1)}_{z_0})$ and $z \in B_{\CC, \delta^{(2)}_{z_0}}(z_0)$, for each $h\in\mathbf H^1(\Omega)$, there exists a unique $g^h_{{\rm{dom}}}(\tau,z;z_0) \in\mathbf H^1(\Omega)$ satisfying
\begin{align} 
J^{\rm{dom}}_{z_0,\tau}(g_{{\rm{dom}}}^h(\tau,z;z_0), \psi, z) = (h,\psi)_{\mathbf H^1(\Omega)}, \quad \forall \psi \in\mathbf H^1(\Omega). \label{eq:37}
\end{align}
Here, $(\cdot)_{\mathbf H^1(\Omega)}$ denotes the inner product on $\mathbf H^1(\Omega)$.

Furthermore, $g^h_{{\rm{dom}}}(\tau,z;z_0)$ is analytic with respect to $\tau$ and $z-z_0$, i.e. 
\begin{align*}
g^h_{{\rm{dom}}}(\tau,z;z_0) = \sum^{\infty}_{q_2 = 0}\sum^{\infty}_{q_1=0} \tau^{q_1}(z-z_0)^{q_2} g^h_{q_1,q_2}(z_0) \quad {\rm{for}}\; \tau \in (0,\delta^{(1)}_{z_0}) \; {\rm{and}}\; z \in B_{\CC,\delta^{(2)}_{z_0}}(z_0),
\end{align*}
where 
\begin{align*}
&\|g^h_{q_1,q_2}(z_0)\|_{\mathbf H^1(\Omega)} \le C_{q_1,q_2} \|h\|_{\mathbf H^1(\Omega)},\\
&\quad {\rm{and}} \quad \sum^{\infty}_{q_2= 0}\sum^{\infty}_{q_1=0}\tau^{q_1} |z-z_0|^{q_2} C_{q_1,q_2} < \infty, \quad {\rm{for}}\; \tau \in (0,\delta^{(1)}_{z_0}) \; {\rm{and}}\; z \in B_{\CC,\delta^{(2)}_{z_0}}(z_0).
\end{align*}
Here, $C_{q_1,q_2}$ with $q_1,q_2\in \mathbb N$ is a positive constant independent of $h$.

\item \label{a2}  
Let $\delta^{(1)}_{z_0}, \delta_{z_0}^{(2)} \in \R_+$ be as in the previous statement \eqref{a1}. Assume that $\tau \in (0,\delta^{(1)}_{z_0})$ and $z \in B_{\CC, \delta^{(2)}_{z_0}}(z_0)$.
For any $h\in\mathbf H^1(\Omega)$, there exists a unique $g^h(\tau,z)\in\mathbf H^1(\Omega)$ satisfying 
\begin{align}
J_{\tau}(g^h(\tau,z), \psi,z) = (h,\psi)_{\mathbf H^1(\Omega)}, \quad \forall \psi \in\mathbf H^1(\Omega), \label{eq:38}
\end{align}
if and only if the matrix $I - M(\tau,z;z_0)$ is invertible, where $M(\tau,z;z_0)$ is a $n(z_0) \times n(z_0)$ matrix whose $(j,l)$ entry is given by
\begin{align} \label{eq:35}
M_{j,l}(\tau,z;z_0) := [\Pi(z_0)\eta_j(\tau,z;z_0)]_l, \quad j,l\in \left\{1,\ldots,n(z_0) \right\}.
\end{align}
Here, $\eta_l(\tau, z; z_0)$ is the unique solution of 
\begin{align} \label{eq:36}
J^{{\rm{dom}}}_{z_0,\tau}(\eta_{l}(\tau, z;z_0),\psi,z) = [\Pi(z_0)\psi]_l,
\quad \forall \psi \in\mathbf H^1(\Omega).
\end{align}

Furthermore, when $I- M(\tau,z;z_0)$ is invertible, the solution has the representation of 
\begin{align}
&g^h({\tau,z})  = g_{{\rm{dom}}}^h(\tau,z;z_0) \notag\\
&+ \left(\eta_1(\tau,z;z_0), \ldots, \eta_{n(z_0)}(\tau,z;z_0)\right)\left(I- M (\tau,z;z_0)\right)^{-1} (\Pi(z_0)g_{{\rm{dom}}}^h(\tau,z;z_0)). \label{eq:39}
\end{align}
Here, $g_{\mathrm{dom}}^h(\tau,z;z_0)$ is the solution of equation \eqref{eq:37}. 

Moreover, if $I- M(\tau,z;z_0)$ is not invertible, for any $(b_1,\ldots b_{n(z_0)}) \in \mathrm{Ker}(I - M (\tau,z;z_0))$, we have 
\begin{align}
J_{\tau}\left(\sum^{n(z_0)}_{l=1} b_l \eta_l(\tau, z; z_0), \psi,z\right) = 0, \quad \forall \psi \in\mathbf H^1(\Omega). \label{eq:41}
\end{align}
\end{enumerate}
\end{pro}

\begin{remark}[Extensions]\label{re:ex}
It should be noted that the above reduction to the boundary value system for characterizing scattering resonances is flexible and suggests extensions to other contrast scalings and interface laws. Two natural examples are as follows:
\begin{enumerate}[(1)]
\item High density contrast with moderate stiffness contrast.
Assume that the Lam\'e parameters remain of order one (or only moderately perturbed) while the
mass density is highly contrasting. For instance, suppose that the density has an exterior-to-interior ratio of order $\tau$ as $\tau \rightarrow 0$.
In this regime, one expects a {sequence} of subwavelength
resonances near \(0\), obtained as small-frequency perturbations (``shifts'')
of the spectral data of the \emph{static} (Newtonian) elasticity potentials, i.e.\ integral operators whose kernels are given by the Kelvin tensor (see \cite{CDS-1}). From the viewpoint of our method, the resonance condition again reduces to the existence of a nontrivial solution of a boundary value system: the Lam\'e equation holds in the inclusion, coupled with a $\tau$-dependent Neumann-type boundary condition on $\Gamma$, expressed via the exterior Dirichlet-to-Neumann map evaluated at a frequency parameter that depends linearly on $\sqrt \tau$. The leading-order resonance behaviour is then obtained by performing a low-frequency (static) expansion of the corresponding interface maps and tracking their coupled scaling with the contrast parameter.

\item High stiffness (Lam\'e) contrast with moderate density contrast.
Assume next that the density remains of order one while the Lam\'e parameters are highly contrasting. For instance, suppose that the Lam\'e parameters have an interior-to-exterior ratio of order $\tau$ as $\tau \rightarrow 0$.
In this case, the inclusion behaves as an increasingly rigid body and the limiting transmission mechanism is closer to an \emph{effective Dirichlet constraint} on the displacement inside the inclusion. Accordingly, the resonances are expected to be governed by
small perturbations of the \emph{Dirichlet} spectrum of the interior Lam\'e operator on \(\Omega\) (in contrast with the Neumann-type interior spectrum that drives the resonance clustering in the
present manuscript). In the boundary-value formulation, this simply amounts to replacing the $\tau$-dependent interface law written in Neumann form via the exterior Dirichlet-to-Neumann map with the corresponding Dirichlet form written via the exterior Neumann-to-Dirichlet map on the boundary $\Gamma$.

\end{enumerate}

These examples illustrate that this type of reduction provides a unified route, within our framework, to (i) resonance localisation, (ii) asymptotic expansions for resonance shifts and widths, and (iii) resolvent estimates and effective point-scatterer descriptions in the microresonator limit, across a range of contrast mechanisms.
\end{remark}
}
 
\section{Proofs of the Main Theorems} \label{sec:6}

\subsection{Proofs of Theorems \ref{th:0}--\ref{th:2}} \label{sec:5.1}

{\color{HW1}Before proving Theorem \ref{th:0}, we first establish a lemma on the zeros of the leading boundary operator $1/2 I + K_1$, arising from the boundary-integral characterization of scattering resonances in Proposition \ref{pro:3}. The proof of this lemma is given in Appendix \ref{sec:B1}.}

\begin{lemma}\label{le:2}
Assume that $z_0 \in \Lambda$. Then, the operator $1/2 I +  K_{1}(z)$ has the following expansion near $z_0$
\begin{align}\label{eq:13}
&\frac 1 2 I + K_{1}(z) = \begin{cases}
E_{z_0}(z) \left(\mathcal P_0(z_0) + \sum_{l=1}^{n(z_0)} (z - z_0) \mathcal P_l(z_0)\right) F_{z_0}(z), \quad &\textrm{if}\; z_0 \ne 0,\\
E_{0}(z) \left(\mathcal P_0(0) + \sum_{l=1}^{6} z^2 \mathcal P_l(0)\right) F_{0}(z), \quad &\textrm{if}\; z_0 = 0,
\end{cases}
\end{align}
where $n(z_0)$ is defined in \eqref{eq:222}, and 
$\mathcal P_1(z_0),\ldots, \mathcal P_{n({z_0})}(z_0)$ are mutually disjoint one-dimensional projections, satisfying 
\begin{align*}
\sum_{l=0}^{n(z_0)} \mathcal P_l(z_0) = I,
\end{align*}
and operators $E_{z_0}(z)$ and $F_{z_0}(z)$ are both holomorphic and invertible near $z_0$. 
\end{lemma}

\begin{proof}[Proof of Theorem \ref{th:0}]
With the aid of Proposition \ref{pro:3}, it suffices to investigate the injectivity of the boundary integral operator 
\begin{align*}
\mathcal A(z,\tau):=I /2 + K_1(z) + \tau S_1(z)\mathcal N(z).
\end{align*}
It is readily verified that there exists $\delta_{\mathcal I} > 0$ such that the following three arguments hold true:
\begin{enumerate}[(1)]
\item $I/2 + K_1(z)$ and $S_1(z)$ and $ \mathcal N(z)$ are all analytic with respect to $z$ in the strip 
\begin{align*}
\mathcal T_{\mathcal I}(\delta_\mathcal I):=\left\{z\in \mathbb C: {\rm{Re}}(z) \in \mathcal I,\; |{\rm{Im}}(z)| \le\delta_{\mathcal I}\right\}.
\end{align*}
\item $I/2 + K_1(z)$ is invertible for $z \in \overline{B_{\CC, \delta_{\mathcal I}}(z_0)} \backslash \{z_0\}$.
\item There exists a constant $\tau_{\delta_{\mathcal I}} > 0$ dependent on $\delta_I$, such that when $\tau \in (0, \tau_{\delta_{\mathcal I}})$,
\begin{align} \label{eq:111}
\tau\left\|\left(I /2 + K_1(z)\right)^{-1}S_1(z)\mathcal N(z)\right\|_{\mathbf H^{1/2}(\Gamma) \rightarrow\mathbf H^{1/2}(\Gamma)} < 1\; \mathrm{on}\; \partial B_{\CC, \delta_{\mathcal I}}(z_0).
\end{align}
\end{enumerate}
Utilizing generalized Rouch\'e's theorem (see \cite[Theorem 1.15]{AK-09}) and Lemma \ref{le:2}, we obtain that when $\tau \in (0, \tau_{\delta_I})$, $B_{\CC,\delta_{\mathcal I}}(z_0)$ contains the non-injective points of $\mathcal A(z,\tau)$, that is, the scattering resonances of the Hamiltonian $H^e(\tau;\Omega)$; moreover, their number satisfies the estimates \eqref{eq:109} and \eqref{eq:110}. 

Since \eqref{eq:111} can be valid on $\partial B_{\CC, \delta}(z_0)$ for sufficiently small $\delta$ and $\tau \le \tau_{*}(\delta)$, the same reasoning as above implies the existence of a scattering resonance in $B_{\CC, \delta}(z_0)$. From this, we readily obtain that all previously found scattering resonances converge to $z_0$ as $\tau\rightarrow 0$.

Finally, it can be seen that $I/2 + K_1(z)$ is invertible for  
$z\in \mathcal T^c_{\mathcal I}(\delta_\mathcal I)$, where 
\begin{align*}
 \mathcal T^c_{\mathcal I}(\delta_\mathcal I):=\mathcal T_{\mathcal I}(\delta_\mathcal I) \backslash \bigcup_{z_0\in \mathcal I \cap \Lambda_N} B_{\CC,\delta_{\mathcal I}}(z_0).
\end{align*}
Therefore, using arguments analogous to those used in the derivations of \eqref{eq:109} and \eqref{eq:110}, we readily obtain the last assertion of this statement. 
\end{proof}

{\color{HW1} Theorem \ref{th:0} shows that, near each point of the Neumann spectrum $\Lambda$, the scattering resonances form clusters. To prove Theorems \ref{th:1} and \ref{th:2}, we use the local characterization of resonances near $\Lambda$ given in Proposition \ref{pro:4}, which reduces the problem to locating the zeros of the matrix $I- M(\tau,z;z_0)$.  We therefore first derive asymptotic expansions for $M(\tau,z;z_0)$ in the following lemma; its leading term is $I- M^e(\tau,z;z_0)$. We then establish the key properties of the effective matrices in Lemmas \ref{le:7} and \ref{le:6}. We postpone the proof of Lemmas \ref{le:5}--\ref{le:6} to Appendix \ref{sec:B1}.

\begin{lemma} \label{le:5} 
Assume that $ 0<\tau\ll 1$ and $z \in \mathbb C$.
Let $z_0 \in \Lambda $ with $\Lambda$ be as in \textnormal{(N\ref{No:1})}, and let $M(\tau,z;z_0)$ be as in \textnormal{(N\ref{No:4})}. The following arguments hold true.
\begin{enumerate}[(a)]
    \item \label{d1}
Assume that $z_0 \ne 0$. The matrix $M(\tau,z;z_0)$ has the asymptotic expansion
\begin{align}
& M (\tau,z;z_0) = 
I - M^e(\tau,z;z_0) + O(|z-z_0|^2 + \tau^2), \; \mathrm{as}\; \tau \rightarrow 0 \;\mathrm{and}\; z\rightarrow z_0. \label{eq:27}
\end{align}

\item \label{d2} The matrix $M(\tau,z;0)$ has the asymptotic expansion
\begin{align}
& M(\tau,z;0) =  I - M^e(\tau,z;0) 
 + O(|z|^2 +\tau)^3,\; \mathrm{as}\; \tau \rightarrow 0 \;\mathrm{and}\; z\rightarrow 0. \label{eq:28}
\end{align} 
\end{enumerate}
\end{lemma}

\begin{lemma} \label{le:7} 
Let $z_0 \in \Lambda \backslash \{0\}$ with $\Lambda$ be as in \textnormal{(N\ref{No:1})}. Assume that $\kappa^{(l)}(z_0) \in \mathbb C$ ($l \in \{1,\ldots, n(z_0)\}$) is any eigenvalue of the matrix $M^{(1)}(z_0)$. Here, $M^{(1)}(z_0)$ is specified by \eqref{eq:30}. 
Then, we have 
\begin{align} \label{eq:54}
{\rm{Im}}(\kappa^{(l)}(z_0)) > 0.
\end{align}
\end{lemma}

\begin{lemma} \label{le:6}
Let effective matrices $M^{(j)}(0), \; j\in \{1,2,3,4,5,6\}$ be defined by \eqref{eq:30}--\eqref{eq:24}. We have the following arguments regarding the properties of theses matrices.
\begin{enumerate}[(a)]
\item \label{b2} 
$M^{(1)}(0)$ is a negative matrix and $M^{(2)}(0)$ is a non-positive matrix. Furthermore, suppose that $a \in  \mathrm{Ker}\left(M^{(2)}(0)\right) \backslash \{0\}$ is an eigenvector of $M^{(1)}(0)$, then the first three entries of $a$ vanish.

\item \label{b3} $M^{(3)}(0)$, $M^{(4)}(0)$ and $M^{(5)}(0)$ are symmetric matrices.

\item \label{b4}
Assume that $\mathrm{Ker} \left(M^{(2)}(0)\right) \ne \emptyset$. 
Then, for each column vector $a \in \mathrm{Ker} \left(M^{(2)}(0)\right) \backslash \{0\}$, we have 
\begin{align}
a^T M^{(5)}(0)a < 0. \label{eq:98}
\end{align}
Moreover, for every two column vectors $a^{(1)}, a^{(2)} \in \mathrm{Ker} \left(M^{(2)}(0)\right)\backslash \{0\}$, we have 
\begin{align}
\left(a^{(1)}\right)^T M^{(6)}(0) a^{(2)} = 0. \label{eq:99}
\end{align}
\end{enumerate}
\end{lemma}}

Now we are ready to prove Theorems \ref{th:1} and \ref{th:2}.

\begin{proof}[Proof of Theorems \ref{th:1} and \ref{th:2}]
Let $\tau > 0$ be sufficiently small throughout the proof. 
Define
\begin{align} \label{eq:57}
\eta_{\tau,z_0}(\kappa):= 
\begin{cases}
z_0 + \tau \kappa &  \mathrm{if}\; z_0 \ne 0,\\
\sqrt {\tau} \kappa &  \mathrm{if}\; z_0 = 0,
\end{cases}\quad \kappa \in \mathbb  C.
\end{align}
It follows from Theorem \ref{th:0} that all splitting scattering resonances in some neighborhood of $z_0$ converge to $z_0$ as $\tau$ tends to $0$. Moreover, with the aid of Proposition \ref{pro:4}, we obtain that $z$ is a scattering resonance in the neighborhood of $z_0$ if and only if $I- M(\tau,z;z_0)$ is invertible. Therefore, it is sufficient to investigate the injectivity of $I-M(\tau,\eta_{\tau,z_0}(\kappa);z_0)$ when ${\tau \kappa = o(1)}$ as $\tau \rightarrow 0$ in the case $z_0\ne 0$, and when ${\sqrt \tau \kappa = o(1)}$ as $\tau \rightarrow 0$ in the case $z_0 =0$. 

The rest of the proof consists of three parts: the first part prove Theorem \ref{th:1}, and the second and third parts prove asymptotic expansions \eqref{eq:59} and \eqref{eq:112} in Theorem \ref{th:2}, respectively.

\textbf{Part I:}
Since $z_0 \ne 0$, using the asymptotic expansion \eqref{eq:27} and the transformation \eqref{eq:57}, we have 
\begin{align} \label{eq:125}
I-M(\tau,\eta_{\tau,z_0}(\kappa);z_0) = 2z_0\rho_1 \tau \kappa I + \tau M^{(1)}(z_0) + O(\tau^2\kappa^2 + \tau^2), \quad \mathrm{as}\; \tau\rightarrow 0.
\end{align}
Furthermore, Lemma \ref{le:7} directly yields the invertibility of $M^{(1)}(z_0)$. From this, \eqref{eq:125}, and the condition $\tau \kappa = o(1)$ as $\tau \rightarrow 0$, we obtain that non-injective points of $I-M(\tau,\eta_{\tau,z_0}(\kappa);z_0)$ have one of the asymptotic expansion
\begin{align*}
\kappa(\tau;j) = - \frac{1}{2\rho_1z_0}{\kappa^{(j)}}(z_0)
+ o(1), \quad j\in\{1,\ldots,n(z_0)\}, \quad \mathrm{as}\; \tau \rightarrow 0,
\end{align*}
whence the asymptotic expansion \eqref{eq:58} follows from the transformation \eqref{eq:57}. Moreover, it immediately follows from Lemma \ref{le:7} that \eqref{eq:194} holds.

\textbf{Part II:}
When $z_0 = 0$, with the aid of the asymptotic expansion \eqref{eq:28} and the transformation \eqref{eq:57}, we have 
\begin{align}
I-M(\tau,\eta_{\tau,0}&(\kappa);0) = \tau \kappa^2 \rho_1 I + \tau M^{(1)}(0) - i \tau^{3/2} \kappa M^{(2)}(0) - \tau^2 M^{(3)}(0) \notag \\
&- \tau^2 \kappa^2 \left(M^{(4)}(0) + 2\rho_1M^{(1)}(0)\right) - \tau^2 \kappa^4 \rho_1^2 I -i\tau^{5/2}\kappa^3 M^{(5)}(0) - \tau^{5/2}\kappa M^{(6)}(0) \notag \\
& + O((\tau \kappa^2 +\tau)^3), \quad \mathrm{as}\;\tau \rightarrow 0. \label{eq:120}
\end{align}
Since $M^{(1)}(0)$ is a negative matrix (statement \eqref{b2} of Lemma \ref{le:6}), when $\sqrt \tau \kappa = o(1)$ as $\tau \rightarrow 0$, it can be seen that the non-injective points of $I-M(\tau,\eta_{\tau,0}(\kappa))$ should satisfy
\begin{align}
\kappa^{(l)}(\tau) &= \pm \sqrt{-\frac{\kappa^{(l)}}{\rho_1}}  + \kappa^{(l)}_{\mathrm{res}}(\tau), \quad \mathrm{for}\; l \in \{1,\ldots,m\} \notag\\
&\qquad\qquad\qquad\qquad\qquad\;\mathrm{with}\; \kappa^{(l)}_{\mathrm{res}}(\tau) = o(1)\; {as}\; \tau \rightarrow 0. \label{eq:60}
\end{align}
Here, $\kappa^{(l)}$ is one of the elements of $\mathcal L^{(0)}$.
To precisely characterize these non-injective points, in particular for the second term in the above expansion, we introduce
\begin{align} \label{eq:119}
\widetilde M_{\tau}(\kappa):= \frac{I-M(\tau,\eta_{\tau,0}(\kappa);0)}{\tau}.
\end{align}
Clearly, $\widetilde M_{\tau}(\kappa)$ and $I-M(\tau,\eta_{\tau,0}(\kappa))$ share same non-injective points. For simplicity, we only focus on the case $\kappa^{(1)}$; the other cases can be proved in a same manner.

Let
\begin{align*}
\widetilde Q(\kappa^{(1)}):=\begin{bmatrix}
Q(\kappa^{(1)})  \, \big |\,\cdots  \, \big |\,Q(\kappa^{(m)})
\end{bmatrix}.
\end{align*}
Obviously,
\begin{align*}
\widetilde Q(\kappa^{(1)})^T M^{(1)}(0) \widetilde Q(\kappa^{(1)}) = \operatorname{diag}\left(
{\kappa^{(1)}, \ldots, \kappa^{(1)}},
\ldots,\ 
{\kappa^{(m)}, \ldots, \kappa^{(m)}}
\right).
\end{align*}
We note that 
\begin{align} 
\mathrm{det}(\widetilde M_\tau(\kappa)) = \mathrm{det}\left(\widetilde M_{\widetilde Q}(\tau;\kappa^{(1)})\right), \notag
\end{align}
where 
\begin{align*}
\widetilde M_{\widetilde Q}(\tau;\kappa^{(1)}) :=\widetilde Q(\kappa^{(1)})^T \widetilde M_\tau(\kappa^{(1)}(\tau))\widetilde Q(\kappa^{(1)}).
\end{align*}
From \eqref{eq:120}, \eqref{eq:60} and \eqref{eq:119}, we have
\begin{align*}
& \widetilde M_{\widetilde Q}(\tau;\kappa^{(1)}) = \begin{bmatrix}
    \widetilde M^{(1,1)}_{\widetilde Q}(\tau;\kappa^{(1)}) & \widetilde M^{(1,2)}_{\widetilde Q}(\tau;\kappa^{(1)}) \\
    \widetilde M^{(2,1)}_{\widetilde Q}(\tau;\kappa^{(1)}) &\widetilde M^{(2,2)}_{\widetilde Q}(\tau;\kappa^{(1)})
\end{bmatrix},
\end{align*}
where $\widetilde M^{(l,j)}_{\widetilde Q}(\tau;\kappa)$ (for $l,j \in \{1,2\}$) satisfies
\begin{align*}
& \widetilde M^{(1,1)}_{\widetilde Q}(\tau;\kappa^{(1)}) = \pm \sqrt{-\frac{\kappa^{(1)}}{\rho_1}}\left(2\rho_1 \kappa_{\mathrm{res}}^{(1)}(\tau) - i \sqrt \tau M^{(2,1,1)}(\kappa^{(1)})\right) + o(\sqrt{\tau}),\\
& \widetilde M^{(1,2)}_{\widetilde Q}(\tau;\kappa^{(1)}) = \mp i\sqrt{-\frac{\kappa^{(1)}}{\rho_1}}\sqrt \tau M^{(2,1,2)}(\kappa^{(1)}) + o(\sqrt{\tau}),\\
&\widetilde M^{(2,1)}_{\widetilde Q}(\tau;\kappa^{(1)}) = \mp i\sqrt{-\frac{\kappa^{(1)}}{\rho_1}} \sqrt \tau M^{(2,2,1)}(\kappa^{(1)}) + o(\sqrt{\tau}),\\
& \widetilde M^{(2,2)}_{\widetilde Q}(\tau;\kappa^{(1)}) = \kappa^{(1)} I + D(\kappa^{(1)}) + o(\sqrt{\tau}), \quad \mathrm{as}\; \tau \rightarrow 0.
\end{align*}
Here, the diagonal matrix $D(\kappa^{(1)})$ is specified by
\begin{align}
D(\kappa^{(1)}):=\operatorname{diag}({\kappa^{(2)}, \ldots, \kappa^{(2)}},
\ldots,\ 
{\kappa^{(s)}, \ldots, \kappa^{(s)}}),\notag
\end{align}
and the matrices $M^{(2,1,1)}(\kappa^{(1)})$, $M^{(2,1,2)}(\kappa^{(1)})$ and $M^{(2,2,1)}(\kappa^{(1)})$ are given by
\begin{align*}
&M^{(2,1,1)}(\kappa^{(1)}) := 
Q^T(\kappa^{(1)}) M^{(2)}(0)
Q(\kappa^{(1)}),\\
&M^{(2,1,2)}(\kappa^{(1)}) := 
Q^T(\kappa^{(1)}) M^{(2)}(0)
\begin{bmatrix}
Q(\kappa^{(2)}) \, \big |\, \cdots  \, \big |\, Q(\kappa^{(m)})  
\end{bmatrix}, \\
&M^{(2,2,1)}(\kappa^{(1)}) :=  \begin{bmatrix}
Q(\kappa^{(2)}) \, \big |\, \cdots  \, \big |\, Q(\kappa^{(m)})  
\end{bmatrix}^T M^{(2)}(0)
Q(\kappa^{(1)}),
\end{align*}
respectively.
Since $\kappa^{(1)} I - D(\kappa^{(1)})$ is invertible, we obtain that the matrix $\widetilde M^{(2,2)}_{\widetilde Q}(\tau;\kappa^{(1)})$ is invertible. Therefore, $\widetilde M_{\widetilde Q}(\tau;\kappa^{(1)})$ is invertible if and only if the Schur complement of the lower-right block (the $(2,2)$-block) of $\widetilde M_{\widetilde Q}(\tau;\kappa^{(1)})$, denoted by $\widetilde M_{\widetilde Q,S}(\tau;\kappa^{(1)})$, is invertible. In fact, if $\widetilde M_{\widetilde Q,S}(\tau;\kappa^{(1)})$ is invertible, we have 
\begin{align*}
\left(\widetilde M_{\widetilde Q}(\tau;\kappa^{(1)})\right)^{-1} = \begin{bmatrix}
    \left(\widetilde M_{\widetilde Q}(\tau;\kappa^{(1)})\right)_{11}^{-1} &  \left(\widetilde M_{\widetilde Q}(\tau;\kappa^{(1)})\right)_{12}^{-1} \\
    \left(\widetilde M_{\widetilde Q}(\tau;\kappa^{(1)})\right)_{21}^{-1} & \left(\widetilde M_{\widetilde Q}(\tau;\kappa^{(1)})\right)_{22}^{-1}
\end{bmatrix},
\end{align*}
where
\begin{align*}
&\left(\widetilde M_{\widetilde Q}(\tau;\kappa^{(1)})\right)_{11}^{-1}:=\left(\widetilde M_{\widetilde Q,S}(\tau;\kappa^{(1)})\right)^{-1},\\
&\left(\widetilde M_{\widetilde Q}(\tau;\kappa^{(1)})\right)_{12}^{-1}:= -\left(\widetilde M_{\widetilde Q,S}(\tau;\kappa^{(1)})\right)^{-1} \widetilde M^{(1,2)}_{\widetilde Q}(\tau;\kappa^{(1)}) \left(\widetilde M^{(2,2)}_{\widetilde Q}(\tau;\kappa^{(1)}) \right)^{-1}, \\
&\left(\widetilde M_{\widetilde Q}(\tau;\kappa^{(1)})\right)_{21}^{-1}:= -\left(\widetilde M^{(2,2)}_{\widetilde Q}(\tau;\kappa^{(1)}) \right)^{-1}\widetilde M^{(2,1)}_{\widetilde Q}(\tau;\kappa^{(1)})\left(\widetilde M_{\widetilde Q,S}(\tau;\kappa^{(1)})\right)^{-1},\\
&\left(\widetilde M_{\widetilde Q}(\tau;\kappa^{(1)})\right)_{22}^{-1}:= \left(\widetilde M^{(2,2)}_{\widetilde Q}(\tau;\kappa^{(1)}) \right)^{-1}\\
&+ \left(\widetilde M^{(2,2)}_{\widetilde Q}(\tau;\kappa^{(1)}) \right)^{-1}\widetilde M^{(2,1)}_{\widetilde Q}(\tau;\kappa^{(1)})\left(\widetilde M_{\widetilde Q,S}(\tau;\kappa^{(1)})\right)^{-1} \widetilde M^{(1,2)}_{\widetilde Q}(\tau;\kappa^{(1)})\left(\widetilde M^{(2,2)}_{\widetilde Q}(\tau;\kappa^{(1)}) \right)^{-1}.
\end{align*}
Furthermore, it can be seen that
\begin{align*}
 \widetilde M_{\widetilde Q,S}(\tau; \kappa^{(1)})  = \pm \sqrt{-\frac{\kappa^{(1)}}{\rho_1}}\left(2\rho_1 \kappa^{(1)}_{\mathrm{res}}(\tau) I - i \sqrt{\tau} M^{(2,1,1)}(\kappa^{(1)})\right) + o(\sqrt\tau), \quad \mathrm{as}\; \tau \rightarrow 0.
\end{align*}
Therefore, $\kappa^{(1)}_{\mathrm{res}}(\tau)$ should satisfy one of the following expansions
\begin{align} \label{eq:63}
\kappa^{(1)}_{\mathrm{res}}(\tau) = \sqrt{\tau}\left(\frac{i}{2\rho_1} \kappa^{(1,j)}
+ o(1)\right)\quad \mathrm{for}\; j \in \{1,\ldots,d_1\}, \quad \mathrm{as}\; \tau \rightarrow 0.
\end{align}
Here, $\kappa^{(1,1)},\ldots, \kappa^{(1,d_1)}$ are eigenvalues of $M^{(2,1,1)}(\kappa^{(1)})$, counted with multiplicity.

Combining with \eqref{eq:57}, \eqref{eq:60} and \eqref{eq:63} gives \eqref{eq:59} for the case when $\kappa = \kappa^{(1)}$. Moreover, using statement \eqref{b2} of Lemma \ref{le:6}, we obtain \eqref{eq:147}.

\textbf{Part III}: 
Building upon the proof of the second part, there exist scattering resonances $z(\tau;\kappa)$ dependent on $\kappa$, and given by 
\begin{align}\label{eq:131}
z(\tau;\kappa) = \sqrt \tau \kappa(\tau). 
\end{align}
Moreover, $\kappa(\tau)$ should have the following asymptotic expansion 
\begin{align} \label{eq:135}
\kappa(\tau) = \pm \sqrt{-\frac{\kappa} {\rho_1}} + \kappa_{\mathrm{res}_1}(\tau)  \quad \mathrm{with}\; \kappa_{\mathrm{res}_1}(\tau) = o\left(\sqrt{\tau}\right), \quad \mathrm{as}\; \tau \rightarrow 0. 
\end{align}

In the sequel, we investigate the asymptotic behaviour of $\kappa_{\mathrm{res}_1}(\tau)$. Since $\kappa \in \mathcal L^{(0)}\cap \mathcal E$, with the aid of the fact that $M^{(2)}(0)$ is a symmetric matrix, we have
\begin{align} \label{eq:124}
M^{(2)}(0) Q_0(\kappa)= Q^T_0(\kappa) M^{(2)}(0) = 0.
\end{align}
Furthermore, it follows from statement \eqref{b3} of Lemma \ref{le:6} that 
\begin{align}\label{eq:122}
Q^T_0(\kappa)
M^{(6)}(0)
Q_0(\kappa)= 0.
\end{align}
Define 
\begin{align*}
\widetilde M_{\widetilde Q}(\tau;\kappa) :=\widetilde Q(\kappa)^T \widetilde M_\tau(\kappa(\tau))\widetilde Q(\kappa).
\end{align*}
Here, the matrix $\widetilde M_\tau$ is defined by \eqref{eq:119}, and the matrix $\widetilde Q(\kappa)$ is defined by 
\begin{align*}
\widetilde Q(\kappa):= 
\begin{bmatrix}
Q_0(\kappa)  \, \big |\, Q_\perp(\kappa)  \, \big |\, Q_{\mathrm{st}}(\kappa)
\end{bmatrix},
\end{align*}
where $Q_{\mathrm{st}}(\kappa)$ is defined by
\begin{align} \label{eq:211}
Q_{\mathrm{st}}(\kappa):= \begin{bmatrix}
    q^{^{\mathrm{dim}(\mathrm{Ker}(\kappa I-M^{(1)}(0)))+1}} \cdots q^{(6)}
\end{bmatrix}
\end{align}
whose columns form the orthonormal eigenvectors of $M^{(1)}(0)$ different from $\kappa$.
Using \eqref{eq:120}, \eqref{eq:124} and \eqref{eq:122}, we find
\begin{align}
\widetilde M_{\widetilde Q}(\tau;\kappa) = 
\begin{bmatrix}
\widetilde M^{(1,1)}_{\widetilde Q}(\tau;\kappa) &  \widetilde M^{(1,2)}_{\widetilde Q}(\tau;\kappa)\\
\widetilde M^{(2,1)}_{\widetilde Q}(\tau;\kappa) &  \widetilde M^{(2,2)}_{\widetilde Q}(\tau;\kappa)      
\end{bmatrix}, \notag
\end{align}
where $\widetilde M^{(l,j)}_{\widetilde Q}(\tau;\kappa)$ (for $l,j \in \{1,2\}$) satisfies
\begin{align}
& \widetilde M^{(1,1)}_{\widetilde Q}(\tau;\kappa) = \pm \sqrt{-\frac{\kappa} {\rho_1}} \left(2\rho_1 \kappa_{\mathrm{res}_1}(\tau)\right) - \tau M^{(1)}_{\mathrm p}(\kappa) \mp \sqrt{-\frac{\kappa} {\rho_1}} i \tau^{3/2} M_{\mathrm{p}}^{(2)}(\kappa) + o(\tau^{3/2}), \label{eq:126}\\
  &\widetilde M^{(1,2)}_{\widetilde Q}(\tau;\kappa) = -\tau Q^T_0(\kappa) M_{\mathrm p}^{(0)}(\kappa)
\begin{bmatrix}
Q_\perp(\kappa) \, \big |\, Q_{\mathrm{st}}(\kappa)  
\end{bmatrix} +  o(\tau^{3/2}), \label{eq:128}\\
  &\widetilde M^{(2,1)}_{\widetilde Q}(\tau;\kappa)=  -\tau\begin{bmatrix}
Q_\perp(\kappa)  \, \big |\, Q_{\mathrm{st}}(\kappa)  
\end{bmatrix}^T M_{\mathrm p}^{(0)}(\kappa)
Q_0(\kappa)  +  o(\tau^{3/2}), \label{eq:199}\\
&\widetilde M^{(2,2)}_Q(\tau;\kappa)= 
\begin{bmatrix}
\mp i\sqrt{-\frac{\kappa}{\rho_1}}Q^T_\perp(\kappa) M^{(2)}(0)Q_\perp(\kappa)\sqrt{\tau} + O(\tau) &  O(\sqrt \tau)\\
O(\sqrt{\tau}) & \kappa I + D(\kappa) + O(\sqrt \tau)  
\end{bmatrix}, \notag
\end{align}
as $\tau \rightarrow 0$.
Here, the diagonal matrix $\mathcal D(\kappa)$ is defined by
\begin{align} \label{eq:212}
D(\kappa)
  &:= \operatorname{diag}\bigl(\kappa^{\mathrm{dim}(\mathrm{Ker}(\kappa I-M^{(1)}(0)))+1},\dots,\kappa^{(6)}\bigr), \quad\kappa^{(j)} \in \mathcal L^{(0)}\;\mathrm{and}\;\kappa^{(j)} \ne \kappa,
\end{align}
and the matrices $M^{(0)}_{\mathrm p}(\kappa)$, $M^{(1)}_{\mathrm{p}}(\kappa)$ and $M^{(2)}_{\mathrm p}(\kappa)$ are defined by \eqref{eq:213}, \eqref{eq:203} and \eqref{eq:204}, respectively.
Since the matrix
\begin{align} \label{eq:237}
a^T_\perp Q^T_\perp(\kappa) M^{(2)}(0) Q_\perp(\kappa)a_{\perp} < 0, \quad \mathrm{for}\; a_\perp \in \CC^{\mathrm{dim}(\mathrm{Ran}(Q_\perp(\kappa)))} \backslash \{0\},
\end{align}
we readily obtain that $\widetilde M^{(2,2)}_Q(\tau;\kappa)$ is invertible and its inverse satisfies
\begin{align}
\left(\widetilde M^{(2,2)}_Q(\tau;\kappa) \right)^{-1} &=  \pm i\sqrt{-\frac{\rho_1}{\kappa \tau }}
\begin{bmatrix}
\left(Q^T_\perp(\kappa) M^{(2)}(0) Q_\perp(\kappa)\right)^{-1}  & 0\\
0 & 0 
\end{bmatrix}  + O(1),\quad \mathrm{as}\;\; \tau \rightarrow 0. \label{eq:202}
\end{align}
Let 
\begin{align*}
   \widetilde M_{Q,S}(\tau;\kappa):= \widetilde M^{(1,1)}_Q(\tau;\kappa) - \widetilde M^{(1,2)}_Q(\tau;\kappa)\left(\widetilde M^{(2,2)}_Q(\tau;\kappa) \right)^{-1}\widetilde M^{(2,1)}_Q(\tau;\kappa).
\end{align*}
Combining \eqref{eq:126}, \eqref{eq:128}, \eqref{eq:199} and \eqref{eq:202} gives 
\begin{align} \label{eq:129}
\widetilde M_{Q,S}(\tau;\kappa) &= \pm \sqrt{-\frac{\kappa} {\rho_1}} 2\rho_1 \kappa_{\mathrm{res}_1}(\tau) - \tau M^{(1)}_{\mathrm{p}}(\kappa) \notag \\
&\mp i\tau^{3/2}\sqrt{-\frac{\kappa} {\rho_1}}
\left[M_{\mathrm{p}}^{(2)}(\kappa) +   M_{\mathrm p}^{(3)}(\kappa)\right] + o(\tau^{3/2}), \quad \mathrm{as}\; \tau \rightarrow 0.
\end{align}
Here, the matrix $M_{\mathrm p}^{(3)}(\kappa)$ is defined by \eqref{eq:205}. This, together with the fact that the two matrices $\widetilde M_{Q}(\tau;\kappa)$ and $\widetilde M_{Q,S}(\tau;\kappa)$ share same non-injective points,  we obtain that
\begin{align}
\kappa_{\mathrm{res}_1}(\tau) &= \pm \frac{\kappa'' \tau}{2\sqrt{-{\rho_1}\kappa}} + \kappa''_{\mathrm{res}_2}(\tau), \;\;\; k''\in \mathcal L^{(2)}(\kappa;0)\quad \mathrm{with}\; \kappa''_{\mathrm{res}_2}(\tau) = o(\tau)\; {as}\; \tau \rightarrow 0. \label{eq:127}
\end{align}

Finally, we derive the asymptotics of $\kappa''_{\mathrm{res}_2}$ in \eqref{eq:127}. 
We note that \eqref{eq:127} also yields an asymptotic expansion for the non-injectivity set of $\widetilde M_{Q,S}(\tau;\kappa,0)$.
Arguing as in the derivation of \eqref{eq:63}, we can use \eqref{eq:129} to obtain that 
\begin{align} 
\kappa''_{\mathrm{res}_2}(\tau) = \tau^{3/2}\left(\frac{i}{2\rho_1} \kappa'''
+ o(1)\right), \;\;\;\kappa'''\in \mathcal L^{(3)}(\kappa;0,\kappa''), \quad \mathrm{as}\; \tau \rightarrow 0.\label{eq:130}
\end{align}

In conjunction \eqref{eq:131}, \eqref{eq:135}, \eqref{eq:127} and \eqref{eq:130}, we obtain \eqref{eq:112}. Moreover, \eqref{eq:236} follows from \eqref{eq:98} and \eqref{eq:237}. 
\end{proof}

\subsection{Proofs of Theorems \ref{th:3} and \ref{th:5}} \label{sec:5.2}

{\color{HW1} Before proving Theorem \ref{th:3}, we first establish a preliminary resolvent expansion in $\Omega$ in Lemma \ref{le:9}, and we derive asymptotic bounds for the inverse of the matrix $M^e$ in Lemma \ref{le:3}. The proofs of Lemmas \ref{le:9} and \ref{le:3} are given in Appendxi \ref{sec:B2}.

\begin{lemma} \label{le:9} 
Let $\Lambda$ be defined by \textnormal{(N\ref{No:1})}, and let $\tau > 0$ and $\omega \in \R \backslash \{0\}$. We have that for $z_0 \in \Lambda \backslash \{0\}$,
\begin{align}
&R_{H^e(\tau;\Omega)}(\omega) f - R_0(\omega) f  = \left[w_{0}^f(\omega;z_0) + O_{\mathbf H^1(\Omega)}\left((\tau + |\omega - z_0|)\left(\|R_0(\omega) f\|_{\mathbf H^1(\Omega)} + \|f\|_{\mathbf L^2(\Omega)}\right)\right)\right] \notag\\ 
&+ \left[ \begin{bmatrix}
e^{(1)}(z_0),\ldots, e^{(n(z_0))}(z_0)
\end{bmatrix} + O_{\mathbf H^1(\Omega)}(\tau + |\omega - z_0|) \right]\left(M^e(\tau,\omega; z_0) + O(\tau + |\omega - z_0|)^2\right)^{-1} \notag\\ 
&\quad\bigg[\widetilde a^f(\omega,\tau;z_0) + O\left((\tau + |\omega - z_0|)^2 \left(\|R_0(\omega) f\|_{\mathbf H^1(\Omega)}+\|f\|_{\mathbf L^2(\Omega)}\right)\right)\bigg]
\quad {\rm{in}}\; \Omega,\notag \\
&\qquad\qquad\qquad\qquad\qquad\qquad\qquad\qquad\qquad\qquad\qquad\qquad\qquad\qquad\mathrm{as}\;(\tau,\omega) \rightarrow (0,z_0),\label{eq:81}
\end{align}
where
$w_0^f(\omega;z_0)$ satisfies
\begin{align}
\left[L_{\lambda_1,\mu_1} + \rho_1 z_0^2 + P(z_0)\right] &[w_0^f(\omega;z_0) + R_0(\omega)f]\notag\\
&= -\rho_1 f + (\rho_1 (z_0^2-\omega^2) + P(z_0)) R_0(\omega) f \quad\mathrm{in}\; \Omega, \label{eq:73}\\
\partial_{\nu,1} [w_{0}^f(\omega;z_0) +  R_0(\omega)f] & = 0  \quad \mathrm{on}\; \Gamma, \label{eq:75}
\end{align} 
and $\widetilde a^f(\omega,\tau;z_0)\in \mathbb C^6$ is defined by
\begin{align*}
&\left[\widetilde a^f(\omega,\tau;z_0)\right]_l := -\rho_1 [\Pi(z_0)f]_l  -\rho_1(\omega-z_0)\bigg((z_0 + \omega)[\Pi(z_0)R_0(\omega)f]_l  \\
& + 2z_0 \left[\Pi(z_0)w_0^f(\omega;z_0)\right]_l\bigg) -\tau \left\langle \mathcal N(z_0)\gamma w_{0}^f(\omega;z_0) + \partial_{\nu,0} R_0(\omega)f, e^{(l)}(z_0)\right\rangle_{\Gamma}\quad \mathrm{if}\; z_0 \in \Lambda \backslash \{0\}.
\end{align*}
\end{lemma}

\begin{lemma} \label{le:3}
Let $z_0 \in \Lambda $ with $\Lambda$ be as in \textnormal{(N\ref{No:1})}, and let $M(\tau,\omega;z_0)$ be as in \textnormal{(N\ref{No:4})}.
Assume that $z_0 \ne 0$, we have 
\begin{align} \label{eq:133}
\left\|\left(M^e(\tau,\omega;z_0)\right)^{-1}\right\| = \Theta\left(\frac{1}{\tau + |\omega-z_0|}\right), \quad \mathrm{as}\; \omega \rightarrow z_0\; \mathrm{and}\; \tau \rightarrow 0.
\end{align}
\end{lemma}}

\begin{proof}[Proof of Theorem \ref{th:3}]

It can be seen that 
\begin{align} \label{eq:153}
R_{H^e(\tau;\Omega)}(\omega) f - R_0(\omega) f= \mathcal F^+(\omega)\gamma\left(R_{H^e(\tau;\Omega)}(\omega) f - R_0(\omega) f\right) \quad \mathrm{in}\; \R^3 \backslash \Omega,
\end{align}
and that the interior value of the difference of $R_{H^e(\tau;\Omega)}(\omega) f$ and $R_0(\omega)$ admits the expansion \eqref{eq:81} in Lemma \ref{le:9}. By the the well-posedness of equations \eqref{eq:73}–\eqref{eq:75} and the regularity properties
of the resolvent $R_0(\omega)$, we have
\begin{align*}
     \left\|w_0^f(\omega;z_0)\right\|_{\mathbf H^1(\Omega)} \le C{\|f\|_{\mathbf L^2(\R^3)}}, \quad \|R_0(\omega) f\|_{\mathbf {H^1}(\Omega)} \le C {(\tau + |\omega-z_0|)\|f\|_{\mathbf L^2(\R^3)}}. 
\end{align*}
This, together with Lemma \ref{le:3} gives
\begin{align}
 \left\|w_0^f(\omega;z_0)\right\|_{\mathbf H^1(\Omega)} \le C\frac{(\tau + |\omega-z_0|)\|f\|_{\mathbf L^2(\R^3)}}{\|M^e(\tau;\omega;z_0)\|}, \quad \|R_0(\omega) f\|_{\mathbf {H^1}(\Omega)} \le C\frac{(\tau + |\omega-z_0|)\|f\|_{\mathbf L^2(\R^3)}}{\|M^e(\tau;\omega;z_0)\|}. \label{eq:144}
\end{align}
Furthermore, using Lemma \ref{le:3} again, we obtain 
\begin{align} 
\left(M^e(\tau,\omega; z_0) + O(\tau + |\omega - z_0|)^2\right)^{-1}  = \left(M^e(\tau,\omega; z_0)\right)^{-1} + \frac{O(\tau + |\omega - z_0|)}{\|M^e(\tau,\omega,z_0)\|}. \label{eq:196}
\end{align}
This, together with \eqref{eq:81} and we obtain that 
\begin{align}
R_{H^e(\tau;\Omega)}(\omega) f & = -\rho_1 \begin{bmatrix}
e^{(1)}(z_0),\ldots, e^{(n(z_0))}(z_0)
\end{bmatrix}\left(M^e(\tau,\omega; z_0)\right)^{-1}\Pi(z_0)f \notag\\ 
& + \frac{O_{\mathbf{H^1}(\Omega)}((\tau + |\omega-z_0|)\|f\|_{\mathbf L^2(\R^3)})}{\|M^e(\tau;\omega;z_0)\|} \quad \mathrm{in}\; \Omega.\label{eq:84}
\end{align}
From this, with the aid of  \eqref{eq:153}, \eqref{eq:144}, \eqref{eq:84} and the fact that $\mathcal F^+(\omega)$ is analytic around $z_0$, as a mapping
from $\mathbf H^{1/2}(\Gamma)$ to $\mathbf L^2_{-\beta}(\R^3)$, we obtain the assertion of the theorem.
\end{proof}

To prove Theorem~\ref{th:5}, we introduce a scaling transformation that reduces the problem into a fixed reference domain. The scaled and pullback estimates play a key role and are collected in the following lemma; its proof is given in Appendix~\ref{sec:B2}.

\begin{lemma} \label{le:10}
Let $\vep >0$ and $\beta > 1/2$. Assume that $y_0$ is any fixed point in $\R^3$ and $\widetilde \omega$ is a nonzero real number. The following arguments hold true.

\begin{enumerate}[(a)]

\item \label{z1} For $\psi_1 \in \mathbf L^2(\Omega)$, we have
\begin{align} \label{eq:170}
&\int_{\Omega}G^{(0)}(y_0+\vep^{-1}(x-y_0),y; \widetilde \omega)\psi_1(y)d y \notag\\ 
&= \sum_{\sigma \in \{p,s\}} \int_{\Omega}\vep \frac{e^{\frac{i\widetilde \omega}{\vep c_{\sigma,0}}|x-y_0|}}{|x-y_0|} (G^{(0)})_\sigma^{\infty}(\hat x_{y_0}, y - y_0 ; \widetilde \omega)\psi_1(y) d y  + O_{\mathbf L_{-\beta}^2(\R^3)}\left(\vep^{3/2}\|\psi_1\|_{\mathbf L^2(\Omega)}\right).
\end{align}
Furthermore, for $\psi_2 \in\mathbf H^{-1/2}(\Gamma)$, we have
\begin{align} \label{eq:171}
&\int_{\Gamma}\bigg[G^{(0)}(y_0+\vep^{-1}(x-y_0),y; \widetilde \omega)\psi_2(y) dS(y) \notag \\
& = \sum_{\sigma \in \{p,s\}}\int_{\Gamma}\vep\frac{e^{\frac{i \widetilde \omega}{\vep c_{\sigma,0}}|x-y_0|}}{|x-y_0|} (G^{(0)})_\sigma^{\infty}(\hat x_{y_0}, y - y_0 ; \widetilde \omega)\psi_2(y)dS(y) \notag\\
& + O_{\mathbf L_{-\beta}^2(\R^3)} \left(\vep^{3/2}\|\psi_2\|_{\mathbf H^{-1/2}(\Gamma)}\right).  
\end{align}

\item \label{z2} For $\psi_1 \in \mathbf L^2(\Omega)$, we have
\begin{align}
&\int_{\Omega}\partial_{y_j}(G^{(0)})_{kl}(y_0+\vep^{-1}(x-y_0),y;\widetilde \omega) \psi_1(y) dy \notag\\
& = \vep \sum_{\sigma \in \{p,s\}}\int_{\Omega} \frac{e^{\frac{i\widetilde \omega}{\vep c_{\sigma,0}}|x-y_0|}}{|x-y_0|}\partial_{y_j}\left((G^{(0)})_\sigma^{\infty}\right)_{kl}(\hat x_{y_0}, y - y_0 ; \widetilde \omega)\psi_1(y) dy \notag\\
&\quad + O_{\mathbf L_{-\beta}^2(\R^3)}\left(\vep^{3/2}\|\psi_1\|_{\mathbf L^2(\Omega)}\right),\;\; j,k,l \in \{1,2,3\}. \label{eq:186}
\end{align}

\item \label{z3} For $f \in \mathbf L^2_{\beta}(\R^3)$, we have
\begin{align} \label{eq:173}
&\int_{\R^3}G^{(0)}(y_0+\vep(x-y_0),y; \widetilde \omega/\vep) f(y) dy\notag \\
& = \sum_{\sigma \in \{p,s\}} \int_{\R^3}\frac{e^{\frac{i\widetilde \omega}{\vep c_{\sigma,0}}|y-y_0|}}{|y-y_0|}(G^{(0)})_\sigma^{\infty}(\hat y_{y_0}, x - y_0 ;\widetilde \omega) f(y) dy + O _{\mathbf H^1(\Omega)}\left(\vep^{1/2}\|f\|_{\mathbf L_{\beta}^2(\R^3)}\right).
\end{align}
\end{enumerate}
\end{lemma}

Now we are ready to prove Theorem \ref{th:5}.

\begin{proof}[Proof of Theorem \ref{th:5}]
We denote 
\begin{align} \label{eq:80}
u^f_{\tau,\vep}(\omega):= R_{H^e(\tau;\Omega_\vep(y_0))}(\omega) f\;\; \mathrm{and}\;\; v^f(\omega):= R_0(\omega)f.
\end{align}
Consider their scaled functions
\begin{align}
 &\left[\widetilde u^f_{\tau,\vep}(\omega)\right](x):= \left[u^f_{\tau,\vep}(\omega)\right](y_0 + \vep (x-y_0)) \label{eq:83}\\
 \mathrm{and}\;\;& \left[\widetilde v^f(\omega)\right](x):= \left[\widetilde v^f(\omega)\right](y_0 + \vep(x-y_0)), \quad x\in \R^3. \label{eq:88}
\end{align}
It is easy to verify that 
\begin{align*}
\widetilde u^f_{\tau,\vep}(\omega) = R_{H(\tau;\Omega)}(\widetilde \omega) \vep^2 \widetilde f \;\; \mathrm{and}\;\; \widetilde v^f(\omega) = R_{0}(\widetilde \omega) \vep^2 \widetilde f.
\end{align*}
Here, 
\begin{align} \label{eq:158}
\widetilde f(x):= f(y_0 + \vep(x-y_0)), \quad x\in \R^3.
\end{align}
Using statement \eqref{z3} of Lemma \ref{le:10}, we have 
\begin{align*}
\widetilde v^f(\omega) &= \sum_{\sigma \in \{p,s\}} \int_{\R^3}\bigg[\frac{e^{\frac{i\omega}{c_{\sigma,0}}|y-y_0|}}{|y-y_0|} (G^{(0)})^{\infty}_{\sigma}(\hat y_{y_0}, x - y_0 ; \vep\omega)f(y) dy  + O_{\mathbf H^1(\Omega)}\left(\vep^{1/2}\right)\|f\|_{\mathbf L^2(\R^3)}.
\end{align*}
This, together with \eqref{eq:188} gives
\begin{align} \label{eq:163}
\widetilde v^f(\omega) = R^{\infty}(z_0) f  + O_{\mathbf H^1(\Omega)}\left(\max\left(\vep^{1/2}, |\widetilde \omega-z_0|\right)\|f\|_{\mathbf L^2(\R^3)}\right) \quad \mathrm{in}\; \Omega.
\end{align}
Furthermore, we have 
\begin{align} 
&\widetilde f = f(y_0) + O_{{\mathbf L^2(\Omega)}}\left(\vep^{1/2}\|f\|_{\mathbf H^2(B_1(y_0))}\right). \label{eq:162}
\end{align}
With the aid of Lemma \ref{le:9}, we obtain \eqref{eq:81} with $\omega = \widetilde \omega$ and $f = \vep^2 \widetilde f$. Utilizing \eqref{eq:73}, \eqref{eq:75}, \eqref{eq:163} and \eqref{eq:162}, we obtain
\begin{align}
w_0^{\vep^2 \widetilde f}(\widetilde w;z_0) & = -(I - P(z_0))R^\infty(z_0) f \notag\\
&+  O_{\mathbf H^1(\Omega)}\left(\max\left(\vep^{1/2}, |\widetilde \omega-z_0|\right) \left[\|f\|_{\mathbf L^2(\R^3)} + \|f\|_{\mathbf H^2(B_1(y_0))}\right]\right) \quad \mathrm{in}\; \Omega. \label{eq:74}
\end{align}
Using \eqref{eq:81} with $\omega = \widetilde \omega$ and $f = \vep^2 \widetilde f$, and applying \eqref{eq:196}, \eqref{eq:163}, \eqref{eq:162} and \eqref{eq:74}, we obtain
\begin{align} 
\widetilde u^f_{\tau,\vep}(\omega) - \widetilde v^f(\omega)& = - (I - P(z_0)) R^\infty(z_0) f + b_{\tau,\vep}(\widetilde \omega,f;z_0) + \widetilde w^{\mathrm{res}}_{\tau,\vep}(\omega,f;z_0), \notag \\ 
& =: \widetilde w^{\mathrm{dom}}_{\tau,\vep}(\omega,f;z_0) + \widetilde w^{\mathrm{res}}_{\tau,\vep}(\omega,f;z_0),\label{eq:207}
\end{align}
where 
\begin{align*}
&\left\|\widetilde w^{\mathrm{res}}_{\tau,\vep}(\omega,f;z_0)\right\|_{\mathbf H^1(\Omega)} \le  C \bigg[ \max\left(\tau,\vep^{1/2},|\widetilde \omega-z_0|\right)\left(1 +\frac{\max\left(\tau,|\widetilde \omega-z_0|\right)}{\|M^e(\tau, \widetilde \omega;z_0)\|}\right)\notag\\
&+ \frac{\vep^{5/2}}{\|M^e(\tau, \widetilde \omega;z_0)\|}\bigg]\left[\|f\|_{\mathbf L^2(\R^3)} + \|f\|_{\mathbf H^2(B_1(y_0))}\right] \\
&\le \bigg[{\max\left(\tau,\vep^{1/2},|\widetilde \omega-z_0|\right)} + \frac{\vep^{5/2}}{\|M^e(\tau, \widetilde \omega;z_0)\|}\bigg] \left[\|f\|_{\mathbf L^2(\R^3)} + \|f\|_{\mathbf H^2(B_1(y_0))}\right],
\end{align*}
$\mathrm{as}\; \widetilde w \rightarrow z_0, \; \tau\rightarrow 0\; \mathrm{and}\; \vep\rightarrow 0$.

On the other hand, set 
\begin{align*}
w_\tau^{\vep^2 \widetilde f}(\widetilde \omega):= R_{H(\tau;\Omega)}(\widetilde \omega) \vep^2 \widetilde f - R_{0}(\widetilde \omega) \vep^2 \widetilde f.
\end{align*}
Based on Green's integral formulas, for each $x\in \R^3 \backslash \Gamma$, we have 
\begin{align}
&w_\tau^{\vep^2 \widetilde f}(\widetilde \omega)(x)
= -\left(SL_0(\widetilde \omega)\mathcal N(\widetilde \omega) \gamma w_\tau^{\vep^2 \widetilde f}(\widetilde \omega)\right)(x) + \frac{\mu_0}{\mu_1} \left(SL_0(\widetilde \omega)\partial_{\nu,1} w_\tau^{\vep^2 \widetilde f}(\widetilde \omega)\right)(x) \notag \\
&+ \widetilde \omega^2 \left(\frac{\mu_0 \rho_1}{\mu_1} - \rho_0 \right) \left(N_{0}(\widetilde \omega) w_\tau^{\vep^2 \widetilde f}(\widetilde \omega)\right)(x) - \frac{\mu_0}{\mu_1} \left(N_0(\widetilde \omega)\left(\frac{\rho_1 L_{\lambda_0,\mu_0}}{\rho_0} - L_{\lambda_1,\mu_1}\right)\widetilde v^f(w)\right)(x) \notag \\
& + \left(\lambda_0 - \lambda_1\frac{\mu_0}{\mu_1}\right)\int_{\Omega} \mathrm{div}\left(G^{(0)}(x,y;\widetilde \omega)\right)\mathrm{div}\left(w_\tau^{\vep^2 \widetilde f}(\widetilde \omega)(y)\right)dy. \label{eq:85}
\end{align}
Here, 
\begin{align*}
N_{0}(z): \mathbf L^2(\Omega) \to \mathbf L_{-\beta}^2(\R^3)\quad &\left(N_{0}{(z})\phi\right)(x)
:=\int_{\Omega}G^{(0)}(x,y;z)\phi(y)dy, \quad\;x\in \R^3,\; z\in \mathbb C.
\end{align*}
We note that 
\begin{align*}
&\left[N_0(\widetilde \omega)\left(\frac{\rho_1 L_{\lambda_0,\mu_0}}{\rho_0} - L_{\lambda_1,\mu_1}\right)\widetilde v^f(w)\right](x) = \left[N_0(\widetilde \omega)\left(\rho_1 \big(-\widetilde \omega^2 \widetilde v^f(\omega) -\vep^2\widetilde f/\rho_0\big)\right)\right](x) \\
&- \left[SL_0(\widetilde \omega)\partial_{\nu,1} \widetilde v^f (\omega)\right](x) + \lambda_1 \int_{\Omega} \textrm{div} \left(G^{(0)}(x,y;\widetilde \omega)\right) \; (\textrm{div} \widetilde v^f(\omega))(y) dy \\
& + 2\mu_1 \int_{\Omega} \mathbb D \left(G^{(0)}(x,y;\widetilde \omega)\right) : \mathbb D (\widetilde v^f(\omega))(y) dy, \quad x\in \R^3 \backslash \Gamma.
\end{align*}

Furthermore, with the aid of \eqref{eq:206}, using Green's integral theorems, for each $e(z_0) \in \Lambda(z_0)$, we have 
\begin{align}
& \left(\mathcal F_\sigma^{\infty}(z_0) \gamma e(z_0)\right)(\hat x_{y_0})e^{iz_0\frac{\hat x_{y_0} \cdot y_0}{c_{\sigma,0}}}= \notag \\ 
&-\int_{\Gamma}(G^{(0)})_\sigma^{\infty}(\hat x_{y_0}, y - y_0 ; z_0)\left[\mathcal N(z_0)e(z_0)\right](y) + \frac{\mu_0}{\mu_1}(G^{(0)})_\sigma^{\infty}(\hat x_{y_0}, y - y_0 ; z_0)\left[\partial_{\nu,1} e(z_0)\right](y)dS(y) \notag \\
& + z_0^2 \left(\frac{\mu_0 \rho_1}{\mu_1} - \rho_0 \right)\int_{\Omega}(G^{(0)})_\sigma^{\infty}(\hat x_{y_0}, y - y_0 ; z_0)\left[e(z_0)\right](y) dy\notag\\ 
&+ \left(\lambda_0 - \lambda_1\frac{\mu_0}{\mu_1}\right)\int_{\Omega} \mathrm{div}_y\left[(G^{(0)})_\sigma^{\infty}\right](\hat x_{y_0}, y - y_0 ; z_0)\mathrm{div}\left[e(z_0)\right](y) dy, \quad \sigma \in \{p,s\}. \label{eq:86}
\end{align}
Similarly, using the fact that 
\begin{align*}
\left(\mathcal L_{\lambda_0,\mu_0} + z_0^2\rho_0\right) R^{\infty}(z_0) f = 0,
\end{align*}
we readily obtain
\begin{align}
& - \left(\mathcal F_\sigma^{\infty}(z_0) \gamma R^{\infty}(z_0) f\right)(\hat x_{y_0})e^{iz_0\frac{\hat x_{y_0} \cdot y_0}{c_{\sigma,0}}}= \int_{\Gamma}(G^{(0)})_\sigma^{\infty}(\hat x_{y_0}, y - y_0 ; z_0)\left[\mathcal N(z_0) R^{\infty}(z_0) f\right](y) \notag\\
&- \frac{\mu_0}{\mu_1}(G^{(0)})_\sigma^{\infty}(\hat x_{y_0}, y - y_0 ; z_0)\left[\partial_{\nu,1} R^{\infty}(z_0) f\right](y)dS(y) \notag \\
& + \int_{\Omega}(G^{(0)})_\sigma^{\infty}(\hat x_{y_0}, y - y_0 ; z_0) \left[z_0^2 \rho_0 - \frac{z_0^2\mu_0\rho_1}{\rho_0} - \frac{\mu_0}{\mu_1} \left(\frac{\rho_1}{\rho_0} L_{\lambda_0,\mu_0} - L_{\lambda_1,\mu_1}\right)\right] \left[R^{\infty}(z_0) f\right](y) dy\notag\\ 
&- \left(\lambda_0 - \lambda_1\frac{\mu_0}{\mu_1}\right)\int_{\Omega} \mathrm{div}_y\left[(G^{(0)})_\sigma^{\infty}\right](\hat x_{y_0}, y - y_0 ; z_0)\mathrm{div}\left[R^{\infty}(z_0) f\right](y) dy, \quad \sigma \in \{p,s\}. \label{eq:87}
\end{align}

Inserting \eqref{eq:163} and \eqref{eq:207} into \eqref{eq:85}, and using pull-back estimates, as stated in statements \eqref{z1} and \eqref{z2} of Lemma \ref{le:10}, we obtain 
\begin{align*}
&[u^f_{\tau,\vep}(\omega)](x) - [u^{f,\mathrm{rem}}_{\tau,\vep}(\omega)](x) \\
&= \sum_{\sigma \in \{p,s\}}\vep\frac{e^{\frac{i \widetilde \omega}{\vep c_{\sigma,0}}|x-y_0|}}{|x-y_0|}\bigg[-\int_{\Gamma}(G^{(0)})_\sigma^{\infty}(\hat x_{y_0}, y - y_0 ; \widetilde w)\left[\mathcal N(\widetilde \omega)\widetilde w^{\mathrm{dom}}_{\tau,\vep}(\omega,f;z_0)\right](y)\\
&\qquad\qquad+ \frac{\mu_0}{\mu_1}(G^{(0)})_\sigma^{\infty}(\hat x_{y_0}, y - y_0 ; \widetilde \omega)\left[\partial_{\nu,1} \widetilde w^{\mathrm{dom}}_{\tau,\vep}(\omega,f;z_0)\right](y)dS(y) \notag \\
&\qquad\qquad + \widetilde\omega^2 \left(\frac{\mu_0 \rho_1}{\mu_1} - \rho_0 \right)\int_{\Omega}(G^{(0)})_\sigma^{\infty}(\hat x_{y_0}, y - y_0 ; \widetilde w)\left[\widetilde w^{\mathrm{dom}}_{\tau,\vep}(\omega,f;z_0)\right](y) dy\notag\\
&\qquad\qquad - \frac{\mu_0 }{\mu_1}\int_{\Omega}(G^{(0)})_\sigma^{\infty}(\hat x_{y_0}, y - y_0 ; \widetilde w)\left[\left(\frac{\rho_1 L_{\lambda_0,\mu_0}}{\rho_0} - L_{\lambda_1,\mu_1}\right)R^{\infty}(z_0)f\right](y) dy\\
& \qquad\qquad + \left(\lambda_0 - \lambda_1\frac{\mu_0}{\mu_1}\right)\int_{\Omega} \left[\mathrm{div}_y\left[(G^{(0)})_\sigma^{\infty}\right](\hat x_{y_0}, y - y_0; \widetilde w)\right)\mathrm{div}\left[\widetilde w^{\mathrm{dom}}_{\tau,\vep}(\omega,f;z_0)\right](y) dy\bigg],
\end{align*}
where
\begin{align*}
\left\|u^{f,\mathrm{rem}}_{\tau,\vep}(\omega)\right\|_{L^2_{-\beta}(\R^3)} &\le C \vep\bigg[{\max\left(\tau,\vep^{1/2},|\widetilde \omega-z_0|\right)} + \frac{\vep^{5/2}}{\|M^e(\tau, \widetilde \omega;z_0)\|}\bigg]\notag\\
&\qquad\left[\|f\|_{\mathbf L^2(\R^3)} + \|f\|_{\mathbf H^2(B_1(y_0))}\right] , \quad  \mathrm{as}\;\widetilde w \rightarrow z_0,\; \tau\rightarrow 0\; \mathrm{and}\; \vep\rightarrow 0.
\end{align*}
From this, utilizing \eqref{eq:206}, \eqref{eq:86} and \eqref{eq:87} and the analyticity of $\mathcal N$ around $z_0$, we obtain the assertion of this theorem.
\end{proof}

\subsection{Proofs of Theorems \ref{th:4} and \ref{th:6}} \label{sec:5.3}

{\color{HW1} The proofs of Theorems \ref{th:4} and \ref{th:6} follow the same strategy as those of Theorems \ref{th:3} and \ref{th:5}. They rely on three preparatory results: a resolvent expansion in $\Omega$ (Lemma \ref{le:4}), asymptotic bounds for $(M^{e})^{-1}$ (Lemma \ref{le:8}), and pullback/scaling estimates in the subwavelength regime (Lemma \ref{le:11}). The proofs of these lemmas are given in Appendix \ref{sec:B3}.

\begin{lemma} \label{le:4}
Let $\Lambda$ be defined by \textnormal{(N\ref{No:1})}, and let $\tau > 0$ and $\omega \in \R \backslash \{0\}$. We have 
\begin{align}
&R_{H^e(\tau;\Omega)}(\omega) f  - R_0(\omega) f = \left[w_{0}^f(\omega;0) + O_{\mathbf H^1(\Omega)}\left((\tau + \omega^2)\left(\| R_0(\omega) f\|_{\mathbf H^1(\Omega)}+\|f\|_{\mathbf L^2(\Omega)}\right)\right)\right] \notag\\ 
&+ \left[\begin{bmatrix}
e^{(1)}(0),\ldots, e^{(6)}(0)
\end{bmatrix} + O_{\mathbf H^1(\Omega)}(\tau + \omega^2) \right]\left(M^e(\tau,\omega; 0) + O(\tau +\omega^2)^3\right)^{-1} \notag\\ 
&\quad\bigg[\widetilde a^f(\omega,\tau;0) + O\left((\tau +  \omega^2)^2 \left(\| R_0(\omega) f\|_{\mathbf H^1(\Omega)} +\|f\|_{\mathbf L^2(\Omega)}\right)\right)\bigg]
\quad {\rm{in}}\; \Omega, \quad \mathrm{as}\;(\tau,\omega) \rightarrow (0,0), \label{eq:195}
\end{align}
where $w_0^f(\omega;0)$ satisfies
\begin{align}
\left[L_{\lambda_1,\mu_1} + P(0)\right] &[w_0^f(\omega;0) + R_0(\omega)f]\notag\\
&= -\rho_1 f + (-\rho_1 \omega^2 + P(0)) R_0(\omega) f \quad\mathrm{in}\; \Omega, \label{eq:61}\\
\partial_{\nu,1} [w_{0}^f(\omega;0) +  R_0(\omega)f] & = 0  \quad \mathrm{on}\; \Gamma, \label{eq:62} 
\end{align} 
and 
\begin{align}
&\left[\widetilde a^f(\omega,\tau; 0)\right]_l :=  -\rho_1 [\Pi(0)f]_l  - \rho_1 \omega^2 \left([\Pi(0)R_0(\omega)f]_l + \left[\Pi(0)w_0^f(\omega;0)\right]_l  \right) \notag\\
&\qquad - \tau \left\langle \mathcal N(0)\gamma w_{0}^f(\omega;0) + \partial_{\nu,0} R_0 f(\omega), e^{(l)}(0)\right\rangle_{\Gamma} - \tau \omega \left\langle \partial_z\mathcal N(0)\gamma w_{0}^f(\omega;0), e^{(l)}(0)\right\rangle_{\Gamma}. \label{eq:218}
\end{align}
\end{lemma}

\begin{lemma} \label{le:8}
Assume that $ 0<\tau\ll 1$ and $\widetilde \omega \in \R$.
let $M^e$ be as in \textnormal{(N\ref{No:4})}. For any $\kappa \in \mathcal L^{(0)} \backslash \mathcal E$ and $a \in \mathbb C^6$, we have  
\begin{align} \label{eq:134}
&\left(M^e(\tau,\sqrt \tau \widetilde \omega;0)\right)^{-1} a = \sum_{\kappa' \in \mathcal L^{(1)}(\kappa)} \frac{1}{\tau}\frac{P_{E_{\mathrm p}(\kappa;\kappa')} a}{\rho_1\widetilde w^2 +\kappa - i \widetilde w \kappa' \sqrt \tau} + \mathrm{Res}^{(1)}(\tau,\widetilde \omega, a),
\end{align}
where $\mathrm{Res}^{(1)}(\tau,\widetilde \omega, a)$ satisfies 
\begin{align}
& \|\mathrm{Res}^{(1)}(\tau,\widetilde \omega, a)\| \le C \frac{|\rho_1\widetilde w^2 + k| + \tau^{1/2}}{\tau\sqrt{|\rho_1\widetilde w^2 + \kappa|^2 - i \sqrt \tau \widetilde w \min_{\kappa' \in \mathcal L(\kappa) }|\kappa'|}}\|a\| \notag\\ 
&\qquad\qquad\qquad\qquad\qquad\qquad\qquad\qquad\quad \mathrm{as}\; \widetilde w \rightarrow \pm \sqrt{-\frac{\kappa}{\rho_1}}\; \mathrm{and}\; \tau \rightarrow 0. \label{eq:191}
\end{align}
Furthermore, assume that \eqref{eq:233} holds. Given $\kappa \in \mathcal L^{(0)} \cap \mathcal E$, for any $\kappa'' \in \mathcal L^{(2)}(\kappa;0)$ and $a \in \mathbb C^6$, we have
\begin{align} \label{eq:121}
&\left(M^e\left(\tau,\pm\sqrt\tau\sqrt\frac{-\kappa}{\rho_1} + \tau^{3/2} \widetilde w;0\right)\right)^{-1} a\notag\\
&= \pm \frac{1}{2\sqrt{-{\rho_1}\kappa}} \sum_{\kappa''' \in \mathcal L^{(3)}(\kappa;0,\kappa'')} \frac{1}{\tau^2}\frac{P_{E_{\mathrm p}(\kappa;0,\kappa'',\kappa''')} a}{ \widetilde w \mp \frac{\kappa''}{2\sqrt{-{\rho_1}\kappa}} - i  \frac{1}{2\rho_1} \kappa'''\sqrt{\tau}} +  \mathrm{Res}^{(2)}(\tau,\widetilde \omega, a)
\end{align}
where $\mathrm{Res}^{(2)}(\tau,\widetilde \omega, a)$ satisfies 
\begin{align}
& \|\mathrm{Res}^{(2)}(\tau,\widetilde \omega, a)\| \le C\frac{\left(|\widetilde w \mp \frac{\kappa''}{2\sqrt{-{\rho_1}\kappa}}| + \tau^{1/2}\right)\left(\|P_{E_{\mathrm p}(\kappa)}a\| + \tau \|(I - P_{E_{\mathrm p}(\kappa)}a)\|\right)}{\tau^2 \sqrt{|\widetilde w \mp \frac{\kappa''}{2\sqrt{-{\rho_1}\kappa}}| - i \frac{1}{2\rho_1}\sqrt{\tau} \min_{\kappa'''\in \mathcal L^{(3)}(\kappa;0,\kappa'')}|\kappa'''|} }\notag\\
&\qquad\qquad\qquad\qquad\qquad\qquad\qquad\qquad\qquad\qquad\mathrm{as}\; \widetilde w \to \pm \frac{\kappa''}{2\sqrt{-{\rho_1}\kappa}} \; \mathrm{and}\; \tau \rightarrow 0. \label{eq:242}
\end{align}
Here, the sets $\mathcal L^{(0)},\; \mathcal L^{(2)}(\kappa;0),\; \mathcal L^{(3)}(\kappa;0,\kappa'')$ and the spaces $E_{\mathrm p}(\kappa),\; E_{\mathrm p}(\kappa;\kappa'),\; E_{\mathrm p}(\kappa;0;\kappa'',\kappa''')$ are specified in \textnormal{(N\ref{No:3})}, 
and $C$ is a positive constant independent of $\tau$, $\widetilde \omega$ and $a$.
\end{lemma}}

Now we use Lemmas \ref{le:4} and \ref{le:8} to prove Theorem \ref{th:4}.

\begin{proof}[Proof of Theorem \ref{th:4}]

The proof of this theorem follows the same steps as in the proof of Theorem \ref{th:3}. 
In particular, \eqref{eq:153} also holds, where the interior value of the
difference admits the expansion \eqref{eq:195} in Lemma \ref{le:4}. 
By the well-posedness of equations \eqref{eq:61}–\eqref{eq:62} and the regularity properties
of the resolvent $R_0(\omega)$, 
\begin{align}
\left\|w_0^f(\omega;0)\right\|_{\mathbf H^1(\Omega)} \le C{\|f\|_{\mathbf L^2(\R^3)}}, \quad \|R_0(\omega) f\|_{\mathbf {H^1}(\Omega)} \le C\tau\|f\|_{\mathbf L^2(\R^3)}. \notag 
\end{align}
This, together with the fact that $\|M^e(\tau; \omega;0)\| = \Theta(\tau)$ for $\omega = \theta(\sqrt \tau)$, we have
\begin{align}
 \left\|w_0^f(\omega;0)\right\|_{\mathbf H^1(\Omega)} \le C\frac{\tau\|f\|_{\mathbf L^2(\R^3)}}{\|M^e(\tau; \omega;0)\|}, \quad \|R_0(\omega) f\|_{\mathbf {H^1}(\Omega)} \le C\frac{\tau\|f\|_{\mathbf L^2(\R^3)}}{\|M^e(\tau; \omega ;z_0)\|}. \label{eq:145}
\end{align}
The rest of the proof consists of two parts: the first part is to prove statement \eqref{h1}, while the second part involves the proof of statement \eqref{h2}.

\textbf{Part I:}
It follows from Lemma \ref{le:8} that  
\begin{align} 
\left(M^e(\tau,\sqrt \tau \widetilde \omega; 0) + O(\tau +\tau\widetilde \omega^2)^3\right)^{-1} &= \left(M^e(\tau,\sqrt \tau \widetilde \omega; 0)\right)^{-1} + O\left(\frac{\tau^{3/2}}{\|M^e(\tau,\sqrt \tau \widetilde \omega; 0)\|}\right) \notag \\
&\qquad\qquad\qquad\mathrm{as}\; \widetilde w \rightarrow \pm \sqrt{-{\kappa}/{\rho_1}}\; \mathrm{and}\; \tau \rightarrow 0 \label{eq:215}
\end{align} 
This, together with \eqref{eq:195} and \eqref{eq:145} yields that 
\begin{align}
R_{H^e(\tau;\Omega)}(\omega) f & = -\rho_1 \begin{bmatrix}
e^{(1)}(0),\ldots, e^{(6)}(0)
\end{bmatrix}\left(M^e(\tau, \tau \widetilde \omega; 0)\right)^{-1}\Pi(0)f \notag\\ 
& + \frac{O_{\mathbf{H^1}(\Omega)}(\tau \|f\|_{\mathbf L^2(\R^3)})}{\|M^e(\tau;\tau\widetilde\omega;0)\|} \quad \mathrm{in}\; \Omega, \;\;\mathrm{as}\; \widetilde w \rightarrow \pm \sqrt{-{\kappa}/{\rho_1}}\; \mathrm{and}\; \tau \rightarrow 0. \notag 
\end{align}
In conjunction with \eqref{eq:134} yields the assertion \eqref{h1} of the theorem.

\textbf{Part II:} Using Lemma \ref{le:8} again, for $\omega =  \pm\sqrt\tau\sqrt{-\kappa/\rho_1} + \tau^{3/2} \widetilde w $, we find 
\begin{align}
\left(M^e\left(\tau, w; 0\right) + O\left(\tau + \omega^2\right)^3\right)^{-1} & = \left(M^e\left(\tau, w; 0\right)\right)^{-1} + O\left(\frac{\tau^{1/2}}{\left\|(M^e\left(\tau, w; 0\right)\right\|}\right),\notag \\ 
&\qquad\qquad \mathrm{as}\; \widetilde w \to \pm \frac{\kappa''}{2\sqrt{-{\rho_1}\kappa}}\; \mathrm{and}\; \tau \rightarrow 0. \label{eq:216}
\end{align}
Combining this with \eqref{eq:195} and \eqref{eq:145} yields that for $\omega =  \pm\sqrt\tau\sqrt{-\kappa/\rho_1} + \tau^{3/2} \widetilde w $, 
\begin{align}
R_{H^e(\tau;\Omega)}(\omega) f & = -\rho_1 \begin{bmatrix}
e^{(1)}(0),\ldots, e^{(6)}(0)
\end{bmatrix}\left(M^e(\tau, \omega; 0)\right)^{-1}\Pi(0)f \notag\\ 
& + \frac{O_{\mathbf{H^1}(\Omega)}(\tau \|f\|_{\mathbf L^2(\R^3)})}{\|M^e(\tau;\omega;0)\|} \quad \mathrm{in}\; \Omega, \;\;\mathrm{as}\; \widetilde w \to \pm \frac{\kappa''}{2\sqrt{-{\rho_1}\kappa}}\; \; \mathrm{and}\; \tau \rightarrow 0, \notag 
\end{align}
whence the assertion \eqref{h2} of this theorem follows from \eqref{eq:121}.
\end{proof}

\begin{lemma} \label{le:11}
Let $\vep,\tau>0$ and $\beta > 1/2$. Assume that $y_0$ is any point in $\R^3$ and $\widetilde \omega$ is a nonzero real number. The following arguments hold true.

\begin{enumerate}[(a)]

\item \label{g1} For $\psi_1 \in \mathbf L^2(\Omega)$, we have
\begin{align} \label{eq:160}
&\int_{\Omega} G^{(0)}(y_0+\vep^{-1}(x-y_0),y;\sqrt \tau \widetilde w)\psi_1(y) dy \notag \\
&= \vep G^{(0)}(x, y_0;\sqrt \tau \vep^{-1} \widetilde w)\int_{\Omega}\psi_1(y)d y + O_{\mathbf L_{-\beta}^2(\R^3)}\left(\max\left(\vep^{3/2},\vep\sqrt\tau\right)\|\psi_1\|_{\mathbf L^2(\Omega)}\right).
\end{align}
Furthermore, we have
\begin{align}
&\int_{\Omega}\partial_{y_j}(G^{(0)})_{kl}(y_0+\vep^{-1}(x-y_0),y; \sqrt \tau \widetilde \omega)\psi_1 (y)dy \notag \\
&\qquad\qquad= O_{\mathbf L_{-\beta}^2(\R^3)}\left(\max\left(\vep^{3/2},\vep\sqrt\tau\right)\|\psi_1\|_{\mathbf L^2(\Omega)}\right),\quad  j,k,l \in \{ 1,2,3\}. \label{eq:234}
\end{align}

\item \label{g2} 
For $\psi_2 \in\mathbf H^{-1/2}(\Gamma)$, we have
\begin{align} 
&\int_{\Gamma} G^{(0)}(y_0+\vep^{-1}(x-y_0),y;\sqrt \tau \widetilde w) \psi_2(y) dS(y) \notag\\ 
&= \vep \int_{\Gamma}G^{(0)}(x, y_0;\sqrt \tau \vep^{-1} \widetilde w)\psi_2(y)d S(y) + O_{\mathbf L_{-\beta}^2(\R^3)}\left(\max\left(\vep^{3/2},\vep\sqrt\tau\right)\|\psi_2\|_{\mathbf H^{-1/2}(\Gamma)}\right). \label{eq:161}
\end{align}
Let $r>0$ be arbitrary. We further have 
\begin{align} \label{eq:235}
\int_{\Gamma} G^{(0)}(y_0+\vep^{-1}(x-y_0),y;\sqrt \tau \widetilde w)  &\psi_2(y) dS(y)= \bigg[\vep G^{(0)}(x, y_0;\sqrt \tau \vep^{-1} \widetilde w) \int_{\Gamma} \psi_2(y) dS(y) \notag \\
&+ \vep^2 \int_{\Gamma} (y-y_0)\nabla_{y}G(x,y_{0}; \sqrt \tau \vep^{-1} \widetilde w) \psi_2(y) dS(y)\bigg] \notag\\
& +  O_{\mathbf L_{-\beta}^2(\R^3\backslash B_r(y_0))}\left(\vep(\sqrt \tau +\vep)^2\|\psi_2\|_{\mathbf H^{-1/2}(\Gamma)}\right).
\end{align}
 
\item \label{g3}
For $f \in \mathbf L^2_{\beta}(\R^3)$, we have
\begin{align} 
\int_{\R^3} G^{(0)}(y_0+\vep(x-y_0),y; \sqrt \tau \widetilde w \vep^{-1})&f(y)dy  = [R_0(\sqrt \tau \widetilde w \vep^{-1}))f](y_0) \notag\\
& + O_{\mathbf H^1(\Omega)}\left(\max\left(\vep^{1/2},\sqrt\tau\right)\|f\|_{\mathbf L_{\beta}^2(\R^3)}\right).\label{eq:224}
\end{align}
\end{enumerate}
\end{lemma}

Now we are in a position to give the proof of Theorem \ref{th:6}.

\begin{proof}[Proof of Theorem \ref{th:6}]

Let $u^f_{\tau,\vep}$ and $v^f(\omega)$ be defined as in \eqref{eq:80}, and let their scaled versions $\widetilde u^f_{\tau,\vep}$ and $\widetilde v^f(\omega)$ be given by \eqref{eq:83} and \eqref{eq:88}, respectively. Using Lemma \ref{le:4}, we obtain that \eqref{eq:195} with $\omega = \vep \omega$ and $f = \vep^2 \widetilde f$ holds. Here, $\widetilde f$ is given by \eqref{eq:158}. The remainder of the proof proceeds in two parts, corresponding to statements \eqref{m1} and \eqref{m2}.

\textbf{Part I}: 
 Using statement \eqref{g2} of Lemma \ref{le:11}, we have 
\begin{align}\label{eq:214}
\widetilde v^f(\omega) &= \left[R_0\left(\omega \right)f\right](y_0) + O_{\mathbf H^1(\Omega)}\left(\max\left(\vep^{1/2},\tau^{1/2}\right)\|f\|_{\mathbf L^2(\R^3)}\right).
\end{align}
Setting $\omega = \vep \omega$ and $f = \vep^2 \widetilde f $  in equations \eqref{eq:61}--\eqref{eq:62}, and using \eqref{eq:162} and \eqref{eq:214}, we arrive at 
\begin{align} \label{eq:217}
w_0^{\vep^2 \widetilde f}(\vep w;0) & =  O_{\mathbf H^1(\Omega)}\left(\max\left(\vep^{1/2}, \sqrt \tau\right)\left[\|f\|_{\mathbf L^2(\R^3)} + \|f\|_{\mathbf H^2(B_1(y_0))}\right]\right) \quad \mathrm{in}\; \Omega.
\end{align}
Setting $\omega = \sqrt \tau \widetilde \omega$ and $f = \vep^2 \widetilde f$ in \eqref{eq:218}, and applying \eqref{eq:162}, \eqref{eq:214} and \eqref{eq:217}, we have 
\begin{align}
\left[\widetilde a^{\vep^2 f}(\vep \omega,\tau; 0)\right]_l = -\rho_1 \left[\Pi(0)(\vep^2 f(y_0) + \vep^2\omega^2 \left[R_0(\omega)f\right](y_0))\right]_l\notag \\
+ O\left(\max\left(\tau \vep^{1/2},\tau^{3/2},\vep^{5/2}\right)\left[\|f\|_{\mathbf L^2(\R^3)} + \|f\|_{\mathbf H^2(B_1(y_0))}\right]\right). \label{eq:241}
\end{align}
Setting $\omega = \sqrt \tau \widetilde \omega$ and $f = \vep^2 \widetilde f$ in \eqref{eq:195}, and using \eqref{eq:134}, \eqref{eq:215}, \eqref{eq:217} and \eqref{eq:241}, we find
\begin{align} \label{eq:219}
\widetilde u^f_{\tau,\vep}(\omega) - \widetilde v^f(\omega)& = b^{(1)}_{\tau,\vep}(w, f)\ + \widetilde w^{\mathrm{res}}_{\tau,\vep}(\omega,f;0),
\end{align}
where 
\begin{align*}
\left\|\widetilde w^{\mathrm{res}}_{\tau,\vep}(\omega,f;z_0)\right\|_{\mathbf H^1(\Omega)} &\le
C \frac{\max(\vep^2,\tau)(|\rho_1\widetilde w^2 + k| + \tau^{1/2}) + \max\left(\tau \vep^{1/2},\tau^{3/2},\vep^{5/2}\right)}{\tau\sqrt{|\rho_1\widetilde w^2 + \kappa|^2 - i \sqrt \tau \widetilde w \min_{\kappa' \in \mathcal L(\kappa) }|\kappa'|}} \notag\\
&\quad\quad\left[\|f\|_{\mathbf L^2(\R^3)} + \|f\|_{\mathbf H^2(B_1(y_0))}\right]\quad \mathrm{as}\;  \widetilde w \rightarrow \pm \sqrt{-\frac{\kappa}{\rho_1}}, \tau \rightarrow 0 \;\mathrm{and}\; \vep \rightarrow 0.
\end{align*}

Moreover, it is known that $\mathcal N_0(z)$ is analytic around $0$ and 
\begin{align}
&\mathcal N(0) e(0) = -\left(S_0(0)\right)^{-1} e(0), \quad e(0) \in {\Lambda_N(0)}. \label{eq:32}
\end{align}
In the sequel, we set $\omega = \sqrt \tau \widetilde \omega$ and $f = \vep^2 \widetilde f$ in \eqref{eq:85}, and insert \eqref{eq:214} and \eqref{eq:219} into \eqref{eq:85}. Since $b^{(1)}_{\tau,\vep}(w, f) \in \Lambda _N(0)$, its contribution to the second, fourth and fifth integral terms on the right-hand side of \eqref{eq:85} vanishes. Then, using pull back estimates, as stated in \eqref{eq:160}, \eqref{eq:234} and \eqref{eq:161}, and applying \eqref{eq:32}, we obtain the assertion \eqref{m1} of this theorem.

\textbf{Part II}:
In what follows, we set 
\begin{align*}
\omega =  \left(\pm\sqrt\tau\sqrt{-\kappa/\rho_1} + \tau^{3/2} \widetilde w\right)/\vep\; \mathrm{and} \;\tau = O(\vep^2).
\end{align*}
Since $f \in \mathbf H^1(\R^3) \cap {\mathbf H^3(B_1(y_0))}$ and $|\omega| \le C$, it can be seen that  
\begin{align} \label{eq:220}
    &\widetilde f = f(y_0) + \vep \nabla f(y_0) \cdot (x-y_0) + O_{{\mathbf L^2(\Omega)}}\left(\vep^{3/2}\|f\|_{\mathbf H^3(B_1(y_0))}\right),
\end{align}
and that,
\begin{align} \label{eq:221}
\widetilde v^f(\omega) &= \left[R_0\left(\omega \right)f\right](y_0) +  \vep \left[\nabla R_0\left(\omega\right)\right](y_0) \cdot (x-y_0) + O_{\mathbf H^1(\Omega)}\left(\vep^{3/2}\|f\|_{\mathbf H^1(\R^3)}\right).
\end{align}
Arguing as in the derivation of \eqref{eq:217}, with \eqref{eq:221} replacing \eqref{eq:214},
\begin{align} \label{eq:246}
w_0^{\vep^2 f}(\vep w;0)(x) & = -\vep (I-P(0))(\left[\nabla R_0\left(\omega\right)\right](y_0) \cdot (x-y_0)) \notag \\
& + O_{\mathbf H^1(\Omega)}\left(\vep^{3/2}\left[\|f\|_{\mathbf L^2(\R^3)} + \|f\|_{\mathbf H^2(B_1(y_0))}\right]\right) \quad \mathrm{in}\; \Omega.
\end{align}
We note that 
\begin{align*}
\langle \partial_{\nu,0} x, e(0) \rangle_{\Gamma} = \langle  x, \partial_{\nu, 0}e(0) \rangle_{\Gamma} = 0, \quad \mathrm{for}\; e(0) \in \Lambda_N(0).
\end{align*}
Applying \eqref{eq:218} with $\omega = \vep \omega$ and $f = \vep^2 \widetilde f$, and using \eqref{eq:217}, \eqref{eq:220} and \eqref{eq:221}, we arrive at 
\begin{align}
&\widetilde a^{\vep^2 f}(\vep \omega,\tau; 0) = a_{\tau,\vep}^{(1)}(\omega,f) + a_{\tau,\vep}^{(2)}(\omega, f) + O\left(\vep^{7/2}\left[\|f\|_{\mathbf H^1(\R^3)} + \|f\|_{\mathbf H^3(B_1(y_0))}\right]\right). \notag
\end{align}
Here, $a_{\tau,\vep}^{(1)}(\omega,f) $ and $a_{\tau,\vep}^{(2)}(\omega, f)$ are given by \eqref{eq:238} and \eqref{eq:239}, respectively.
Furthermore, it can be seen that 
\begin{align}\label{eq:240}
a_{\tau,\vep}^{(1)}(\omega,f) \in \mathrm{Ran}(I - Q(\kappa)).
\end{align}
Here, the matrix $Q(\kappa)$ is as in $\textnormal{(N\ref{No:3})}$.  Setting $\omega = \vep \omega$ and $f = \vep^2 \widetilde f$ in \eqref{eq:195}, and using \eqref{eq:121}, \eqref{eq:216}, \eqref{eq:246} and \eqref{eq:240}, we have
\begin{align} \label{eq:223}
\widetilde u^f_{\tau,\vep}(\omega) - \widetilde v^f(\omega)& = b^{(2)}_{\tau,\vep}(w, f) + \begin{bmatrix}
e^{(1)}(0),\ldots, e^{(6)}(0)
\end{bmatrix} a^{\mathrm{res}}_{\tau,\vep}(\omega,f) + \widetilde w^{\mathrm{res}}_{\tau,\vep}(\omega,f;0),
\end{align}
where
\begin{align*}
&\left\|a^{\mathrm{res}}_{\tau,\vep}(\omega,f)\right\|
\le\frac{\vep^3(|\widetilde w \mp \frac{\kappa''}{2\sqrt{-{\rho_1}\kappa}}| + \tau^{1/2}) \ + \vep^{7/2}}{\tau^2 \sqrt{|\widetilde w \mp \frac{\kappa''}{2\sqrt{-{\rho_1}\kappa}}| - i \frac{1}{2\rho_1}\sqrt{\tau} \min_{\kappa'''\in \mathcal L^{(3)}(\kappa;0,\kappa'')}|\kappa'''|}} \notag \\ 
&\quad\left[\|f\|_{\mathbf H^1(\R^3)} + \|f\|_{\mathbf H^3(B_1(y_0))}\right]\quad\mathrm{as}\; \widetilde w \to \pm \frac{\kappa''}{2\sqrt{-{\rho_1}\kappa}},\;\tau \rightarrow 0 \;\mathrm{and}\; \vep \rightarrow 0.
\end{align*}
and 
\begin{align*}
&\left\|\widetilde w^{\mathrm{res}}_{\tau,\vep}(\omega,f;0)\right\|_{\mathbf H^1(\Omega)} \le \vep^2\|b^{(2)}_{\tau,\vep}(w, f) \|_{\mathbf H^1(\Omega)} \\
&+\vep^2 \frac{\vep^3(|\widetilde w \mp \frac{\kappa''}{2\sqrt{-{\rho_1}\kappa}}| + \tau^{1/2}) + \vep^{7/2}}{\tau^2 \sqrt{|\widetilde w \mp \frac{\kappa''}{2\sqrt{-{\rho_1}\kappa}}| - i \frac{1}{2\rho_1}\sqrt{\tau} \min_{\kappa'''\in \mathcal L^{(3)}(\kappa;0,\kappa'')}|\kappa'''|} } \left[\|f\|_{\mathbf H^1(\R^3)} + \|f\|_{\mathbf H^3(B_1(y_0))}\right].
\end{align*}

Moreover, by the properties of principal enhancement space $E_{\mathrm p}(\kappa)$, we have 
\begin{align} \label{eq:66}
\int_{\Gamma} \left(S_0(0)\right)^{-1} \left(\begin{bmatrix}
e^{(1)}(0),\ldots, e^{(6)}(0)
\end{bmatrix} a\right)(y) d S(y) = 0, \quad a\in E_{\mathrm p}(\kappa).
\end{align}
Since
\begin{align}
&\partial_z \mathcal N (0) = \left(S_0(0)\right)^{-1} S_0^{(1)}\left(S_0(0)\right)^{-1}, \label{eq:94}
\end{align}
with the operator $S_0^{(1)}:\mathbf H^{1/2}(\Gamma) \rightarrow\mathbf H^{1/2}(\Gamma)$ given by 
\begin{align} \label{eq:95}
\left(S^{(1)}_0 \phi\right)(x) =  i\frac{\sqrt{\rho_0}}{12\pi}\left(\frac{2}{{\mu_0}^{\frac32}} + \frac{1}{(\lambda_0 + 2\mu_0)^{\frac 32}}\right) \int_{\Gamma} \phi(y) dS(y),\quad x \in \Gamma.
\end{align}
This, together with \eqref{eq:32} and \eqref{eq:66} gives 
\begin{align}
&\int_{\Gamma} \mathcal N(0) \left(\begin{bmatrix}
e^{(1)}(0),\ldots, e^{(6)}(0)
\end{bmatrix} a\right)(y) d S(y) \notag \\
&= \int_{\Gamma} \mathcal \partial_z N(0) \left(\begin{bmatrix}
e^{(1)}(0),\ldots, e^{(6)}(0)
\end{bmatrix} a\right)(y) d S(y) = 0, \quad a\in E_{\mathrm p}(\kappa). \label{eq:247}
\end{align}
We set $\omega = \sqrt \tau \widetilde \omega$ and $f = \vep^2 \widetilde f$ in \eqref{eq:85}, and insert \eqref{eq:221} and \eqref{eq:223} into \eqref{eq:85}. We note that 
the contribution of $b^{(2)}_{\tau,\vep}(w, f)$ and $[ 
e^{(1)}(0),\ldots, e^{(6)}(0)]a^{\mathrm{res}}_{\tau,\vep}(\omega,f)$ to the second, fourth and fifth integral terms on the right-hand side of \eqref{eq:85} is zero.
Then using pull back estimates, as stated in \eqref{eq:160}, \eqref{eq:234} and \eqref{eq:235}, and applying \eqref{eq:32} and \eqref{eq:247}, we readily obtain the assertion \eqref{m2} of this theorem.
                                                
\end{proof}

\begin{appendices}
\renewcommand{\theequation}{\Alph{section}.\arabic{equation}}

\section{Proofs of Propositions \ref{pro:1}--\ref{pro:4}}\label{sec:A}


\subsection{Proof of Proposition \ref{pro:1} in Section \ref{sec:4.1}} \label{sec:A1}

\begin{proof}[Proof of Proposition \ref{pro:1}]

We first claim that: the resolvent $R_{H^e(\tau;\Omega)}(z): \mathbf L^2(\R^3) \rightarrow D(H^e(\tau;\Omega))$ is bounded in the upper-half complex plane $\CC_+$. Moreover, we have 
\begin{align}\label{eq:148}
\|R_{H^e(\tau;\Omega)}(z)\|_{\mathbf L^2(\R^3) \to \mathbf L^2(\R^3)} \le \frac{C}{\mathrm{Im}(z)\min(|z|\rho_0, |z|\rho_1/\tau)}, \quad z \in \CC_+.
\end{align}
Here, $C$ is a positive constant independent of $z$ and $\tau$. Now we prove this claim. For $u, v\in \mathbf H^1(\R^3)$, let 
\begin{align*}
J_\tau^{\mathrm{whole}}(u,v,z)&:=  -\overline{z}\int_{\R^3} \lambda_\tau(x)\textrm{div} u(x) \; \textrm{div} \overline {v(x)} dx\\
&- 2 \overline{z} \int_{\R^3} \mu_\tau(x)\mathbb D u(x) :\mathbb D \overline{v(x)}dx + |z|^2 {z} \int_{\R^3} \rho_\tau(x) u(x) \cdot \overline{v(x)} dx.
\end{align*}
Setting $v = \overline u$, we find 
\begin{align*}
&\frac{\mathrm{Im}\left(J_\tau^{\mathrm{whole}}(u,\overline u,z)\right)}{\mathrm{Im}({z})} =\\
&\int_{\R^3} \lambda_\tau(x)|\textrm{div} u(x)|^2dx + 2 \int_{\R^3} \mu_\tau(x)\mathbb D u(x) :\mathbb D \overline{u(x)}dx + |z|^2\int_{\R^3} \rho_\tau(x) u(x) \cdot \overline{u(x)} dx.
\end{align*}
Using the Korn's inequality
\begin{align*}
    \int_{\R^3} |\textrm{div} u(x)|^2dx + 2 \int_{\R^3}\mathbb D u(x) :\mathbb D \overline{u(x)}dx +  \int_{\R^3} u(x) \cdot \overline{u(x)} dx \ge C\|u\|^2_{\mathbf H^1(\R^3)}.
\end{align*}
Therefore, we have 
\begin{align*}
\mathrm{Im}(J_\tau^{\mathrm{whole}}(u,u,z)) \ge C\mathrm{Im}(z)\min(\lambda_0,\mu_0,|z|^2\rho_0,\lambda_1/\tau,\mu_1/\tau,|z|^2\rho_1/\tau) \|u\|^2_{\mathbf H^1(\R^3)}.
\end{align*}
Then, with the help of Lax-Milgram theorem, for each $f\in \mathbf L^2(\R^3)$, there exists a unique $u_z^f \in \mathbf H^1(\R^3)$ satisfying
\begin{align} \label{eq:132}
    J_\tau^{\mathrm{whole}}(u^f_z,v,z) = - \overline z\int_{\R^3} f(x) \cdot \overline{v(x)}dx, \quad \forall v \in \mathbf H^1(\R^3).
\end{align}
Obviously, $R_{H^e(\tau;\Omega)}(z) = u^f_z \in  D(H^e(\tau;\Omega))$. From this, we readily obtain the boundedness of $R_{H^e(\tau;\Omega)}(z): \mathbf L^2(\R^3) \rightarrow D(H^e(\tau;\Omega)$ in $\CC_+$. Furthermore, setting $v = u^f_z$ in \eqref{eq:132} and applying Cauchy-Schwartz inequality to the right hand side, we have 
\begin{align*}
\|u^f_z\|_{\mathbf L^2(\R^3)} \le \frac{C}{\mathrm{Im}(z)\min(|z|\rho_0, |z|\rho_1/\tau)} \|f\|_{\mathbf L^2(\R^3)}.
\end{align*}
This proves inequality \eqref{eq:148}. Proceeding as in the derivation of \eqref{eq:148}, we obtain
\begin{align}\label{eq:149}
\|R_{0}(z)\|_{\mathbf L^2(\R^3) \to \mathbf L^2(\R^3)} \le \frac{C}{\mathrm{Im}(z) |z|\rho_0}, \quad z \in \CC_+.
\end{align}
Here, $C$ is a positive constant independent of $z$.

Let $\chi_1, \chi_2, \chi_3 \in C_c^{\infty}(\R^3)$ satisfy 
\begin{align*}
&\chi_j = 1\; \textrm{in the neighborhood of}\; \Omega\; \mathrm{for}\; j \in \{1,2,3\},\\
&\textrm{and} \;\chi_{l+1} = 1 \; \mathrm{on}\; \mathrm{supp}\;\chi_l\; \mathrm{for}\; l \in \{1,2\}.
\end{align*}
By the choice of $\chi_1$ and $\chi_2$, we have
\begin{align*}
(-H^e(\tau;\Omega) - z^2)(1-\chi_2)R_0(z) &= (1-\chi_1)\left(-\frac{1}{\rho_0}L_{\lambda_0,\mu_0} - z^2\right)(1-\chi_2)R_0(z)\\
& = 1 - \chi_2 + (1-\chi_1)\frac{1}{\rho_0}[-L_{\lambda_0,\mu_0},\chi_2]R_0(z).
\end{align*}
Furthermore, it follows from \eqref{eq:148} that $R_{H^e(\tau;\Omega)}(ri)$ with $r>0$ is a linear bounded mapping from $\mathbf L^2(\R^3)$ to $\mathbf L^2(\R^3)$. Then, we find
\begin{align*}
\left(-H^e(\tau;\Omega) - z^2\right)& \chi_2 R_{H^e(\tau;\Omega)}(ri) = (-H^e(\tau;\Omega) + r^2 - r^2 - z^2)\chi_2 R_{H^e(\tau;\Omega)}(ri) \\
\qquad\qquad \qquad\qquad\qquad & = \chi_2 + (-r^2-z^2)\chi_2 R_{H^e(\tau;\Omega)}(ri) + \left[-H^e(\tau;\Omega), \chi_2\right] R_{H^e(\tau;\Omega)}(ri).
\end{align*}
Therefore, we arrive at 
\begin{align*}
\left(-H^e(\tau;\Omega) - z^2\right)\left[(1-\chi_2)R_0(z) + \chi_2 R_{H^e(\tau;\Omega)}(ri) \right] = I + \mathcal T_r(z),
\end{align*}
where the operator $\mathcal T_r(z)$ is defined by 
\begin{align*}
\mathcal T_r(z)&:=  (1-\chi_1)\frac{1}{\rho_0}[-L_{\lambda_0,\mu_0},\chi_2]R_0(z) - (r^2 + z^2)\chi_2 R_{H^e(\tau;\Omega)}(ri) \\
&+ \left[-H^e(\tau;\Omega), \chi_2\right] R_{H^e(\tau;\Omega)}(ri).
\end{align*}
By the choice of $\chi_3$ and $\chi_2$, we have
\begin{align*}
&I + \mathcal T_r(z) = (I + \mathcal T_r(z)(1-\chi_3))( I +  \mathcal T_r(z) \chi_3),\\
& (I + \mathcal T_r(z)(1-\chi_3))(I - \mathcal T_r(z)(1-\chi_3)) = I.
\end{align*}

On the other hand, with the aid of \eqref{eq:148} and \eqref{eq:149}, we can find some large $r^*$ such that 
\begin{align*}
    \|\mathcal T_{r^*}(r^* i)\chi_3\|_{\mathbf L^2(\R^3) \to \mathbf L^2(\R^3)} < \frac 1 2, 
\end{align*}
implying that $I + \mathcal T_{r^*}(r^* i)\chi_3 $ is invertible as a linear bounded operator on $\mathbf L^2(\R^3)$. 
Furthermore, we note that $T_{r^*}(z): \mathbf L^2(\R^3) \to \mathbf L^2(\R^3)$ is compact for $z \in \CC_+$. Therefore, by analytic Fredholm theory \cite{DM}[Theorem C.8], we obtain
\begin{align*}
R_{H^e(\tau;\Omega)}(z) = \left[(1-\chi_2)R_0(z) + \chi_2 R_{H^e(\tau;\Omega)}(i) \right](I +  \mathcal T_{r^*}(z) \chi_3)^{-1} (I - \mathcal T_{r^*}(z)(1-\chi_3)).
\end{align*}
This finishes the proof this lemma.
\end{proof}

\subsection{Proofs of Propositions \ref{le:12}--\ref{pro:4} in Section \ref{sec:4.2}}

\begin{proof}[Proof of Proposition \ref{le:12}]

Let $\chi_1, \chi_2 \in C_c^{\infty}(\R^3)$ satisfy 
\begin{align*}
&\chi_j = 1\; \textrm{in the neighborhood of}\; \Omega\; \mathrm{for}\; j \in \{1,2\},\\
&\textrm{and} \;\chi_{2} = 1 \; \mathrm{on}\; \mathrm{supp}\;\chi_1.
\end{align*}
It can be seen that 
\begin{align*} 
&R_0(z)\,[-H^e(\tau;\Omega),\chi_2]\,R_{H^e(\tau;\Omega)}(z)\chi_1  = -(1-\chi_2)\,R_{H^e(\tau;\Omega)}(z)\chi_1,\\
& R_{H^e(\tau;\Omega)}(z)\,[-\rho_0^{-1}L_{\lambda_0,\mu_0},\chi_2]R_0(z)\,\chi_1 = (1-\chi_2)\,R_{0}(z)\chi_1, \quad z\in \mathbb C.
\end{align*}
Therefore, by using similar arguments that were used in the proof of Theorem 4.9 in \cite{DM}, we obtain the assertion of this lemma.
\end{proof}

\begin{proof}[Proof of Proposition \ref{pro:2}]

Recall that $v_z$ is $z-$ outgoing solution if there exists $g \in \mathbf L^2_{\mathrm{comp}}(\R^3)$ such that 
$v_z = R_0(z) g\; \mathrm{in}\; \R^3 \backslash B_r.$
Here, $r>0$ is chosen large enough such that $\overline \Omega \subset B_r$. We first claim that $v_z$ satisfies
\begin{align}
\int_{\partial B_r} G^{(0)}(q,x;z) \partial_{\nu,0} v_z(x) - \partial_{\nu,0} G^{(0)}(q,x;z) v_z(x)dS(x) = 0, \quad q \in \R^3 \backslash \overline{B_r}. \label{eq:225}
\end{align}
In fact, for any function $h \in \mathbf L^2_{\mathrm{comp}}(\R^3)$, we have 
\begin{align}
&\int_{\partial B_R}G^{(0)}(q,x;z){\partial_{\nu,0}}\left[\int_{\textrm{supp}(h)}G^{(0)}(x,y;z)h(y)dy\right] \notag \\
&- \left[\int_{\textrm{supp}(h)}G^{(0)}(x,y;z)h(y)dy\right]{\partial_{\nu,0}}G^{(0)}(q,x;z)dS(x)= 0, \quad z \in \mathbb C\; \mathrm{and}\; q \in \R^3 \backslash \overline{B_r}. \label{eq:226}
\end{align}
Here, $\textrm{supp}(h)$ denotes the compact support of $h$. Obviously, \eqref{eq:226} is valid for any $z \in \overline{\CC_+} \backslash \{0\}$.
By analyticity of the functions in \eqref{eq:226} with respect to $z$, it can be deduced that \eqref{eq:226} holds for all $z \in \CC$. This directly yields that \eqref{eq:225}.

With the aid of \eqref{eq:225} and Green formulas, for any $v_{z_\tau}$ solving equations \eqref{eq:228}--\eqref{eq:229}, we have 
\begin{align} \label{eq:231}
    \left(-\frac{I}{2} + K_0(z_\tau)\right)\gamma v_{z_\tau} = S_0(z_\tau)\partial_{\nu,0} v_{z_\tau}.
\end{align}
We note that if $z$ is not a zero of $S_0(z)$,
\begin{align*}
 \mathcal N(z) = \left(S_0(z)\right)^{-1}\left(-\frac{I}{2} + K_0(z)\right).
\end{align*}
Therefore, we have 
\begin{align} \label{eq:232}
\partial_{\nu,0} v_{z_\tau} = \mathcal N(z) \gamma v_{z_\tau} \quad \mathrm{on}\; \Gamma.
\end{align}
From this, we find that $v_{z_\tau}$ is a nonzero solution of \eqref{eq:96}--\eqref{eq:97}.

Conversely, if $v_{z_\tau}$ is a nonzero solution of \eqref{eq:96}--\eqref{eq:97}. Define 
\begin{align} \label{eq:227}
    v_{z_\tau} = DL_0(z_\tau) \gamma v_{z_\tau}(\tau) - SL_0(z_\tau)\mathcal N(z_\tau) v(z_\tau)\quad \mathrm{in}\; \R^3 \backslash \Omega.
\end{align}
Furthermore, it is easy to verify that for $\psi \in \mathbf H^{-1/2}(\Gamma)$ 
\begin{align} \label{eq:230}
(1-\chi)SL_{0}(z_\tau) \psi = R_0(z_\tau) [\Delta, \chi] SL_{0} (z_\tau) \psi, 
\end{align}
where $\chi \in C_c^{\infty}(\R^3)$ satisfies $\chi = 1 $ in $B_r$. Similarly, \eqref{eq:230} also holds with $SL_0$ replaced by $DL_0$. Therefore, $v_{z_\tau}$, as defined in \eqref{eq:227}, is $z-$ outgoing solution. This, together with \eqref{eq:227} yields that $v_{z_\tau}$ satisfies \eqref{eq:231}.
Since $z_\tau$ is not a zero of $S_0$, we obtain that $v_{z_\tau}$ satisfies \eqref{eq:232}. Therefore, $v_{z_\tau}$, as defined by \eqref{eq:227}, solves equations \eqref{eq:228}--\eqref{eq:229}. This finishes the proof of this lemma.

\end{proof}

\begin{proof} [Proof of Proposition \ref{pro:3}]
Throughout the proof, we assume without loss of generality that $\tau = 1$,
since the other cases can be treated in the same way. By Proposition \ref{pro:2}, it suffices to prove that $z$ is a point such that there exists a non-zero solution $u \in \mathbf H^1(\Omega)$ of 
\begin{align}
&L_{\lambda_{1}, \mu_1} u + z^2\rho_1 u = 0 \qquad \mathrm{in}\; \Omega, \label{eq:6}\\
& \partial_{\nu,1}  u = \mathcal N(z) \gamma u \qquad\quad\;\;\;\mathrm{on}\; \Gamma. \label{eq:7}
\end{align}
if and only if $z$ is a non-injective point of the boundary integral operator $1/2 I + K_1(z)- S_1(z) \mathcal N(z)$. We now prove this claim.

\textbf{Necessity}. Suppose that $z$ is a point where equations \eqref{eq:6}--\eqref{eq:7} admits a non-zero solution $u\in\mathbf H^{1}(\Omega)$. It follows from Green formulas that 
\begin{align*}
u(x)  = \left(SL_1(z) \mathcal N(z) \gamma u\right)(x) - \left(DL_1(z) \gamma u \right)(x), \quad  x \in \Omega.
\end{align*}
Taking $x \rightarrow \Gamma$ and using the jump relations of $SL_1(z)$ and $DL_1(z)$, we obtain 
\begin{align*}
\left(\frac{1} 2 I + K_1(z) - S_1(z) \mathcal N(z)\right) \gamma u = 0.
\end{align*}
We note that $\gamma u \ne 0$ otherwise $u = 0$ in $\Omega$. Therefore, $z$ is a non-injective point of $1/2 I + K_1(z)- S_1(z) \mathcal N(z)$.

\textbf{Sufficiency}. Suppose that $z$ is a non-injective point of $1/2 I + K_1(z)- S_1(z) \mathcal N(z)$, that is, there exists $\phi \in\mathbf H^{1/2}(\Gamma)$ such that 
\begin{align} \label{eq:8} 
\left(\frac 12 I + K_1(z)- S_1(z) \mathcal N(z)\right) \phi = 0 \quad \mathrm{on}\; \Gamma.
\end{align}
Based on this, we aim to construct a non-zero solution of equations \eqref{eq:6}--\eqref{eq:7}. To do so, setting
\begin{align} \label{eq:12}
u(x) =  \left(SL_1(z) \mathcal N(z)\phi\right)(x) - \left(DL_1(z) \phi\right)(x), \quad  x \in \Omega.
\end{align}
Clearly, $u$ solves Lam\'e equation \eqref{eq:6}. Taking $x\rightarrow \Gamma$, we have
\begin{align*}
\gamma u  - \frac{1}2 \phi = S_1(z)\mathcal N(z)\phi - K_1(z)\phi \quad \mathrm{on}\; \Gamma.
\end{align*}
This, together with \eqref{eq:8} gives
\begin{align} \label{eq:9}
\gamma u  = \phi \quad \mathrm{on}\; \Gamma.
\end{align}
On the other hand, since $u$ solves Lam\'e equation \eqref{eq:6}, we find
\begin{align} \label{eq:10}
u(x) = \left(SL_1(z) \partial_{\nu,1} u\right)(x) - \left(DL_1(z)\gamma u\right)(x), \quad  x \in \Omega,
\end{align}
Combining \eqref{eq:12}, \eqref{eq:9} and \eqref{eq:10} gives
\begin{align}
\left(SL_1(z) \mathcal N(z)\phi\right)(x) = \left(SL_1(z)\partial_{\nu,1}u\right)(x), \quad x\in \Omega. \notag
\end{align}
From this, we can use \eqref{eq:9} to obtain
\begin{align}\label{eq:11}
S_1(z)\left(\mathcal N(z)\gamma u - \partial_{\nu,1} u \right) = 0 \quad &\mathrm{on}\; \Gamma.
\end{align}

To proceed, we consider two separate cases according to whether $z$ is a point such that $z^2$ is a Dirichlet eigenvalue of the Lam\'e operator $L_{\lambda_1,\mu_1}$. Here, we say that $z^2$ is a Dirichlet eigenvalue of the Lam\'e operator $L_{\lambda_1,\mu_1}$ if and only if there exists a nontrivial $v\in \mathbf H^1(\Omega)$ such that 
\begin{align*}
&L_{\lambda_1,\mu_1}v + \rho_1 z^2 v = 0\quad \mathrm{in}\; \Omega\\
& v = 0 \quad \mathrm{on}\; \Gamma.
\end{align*}

\textbf{Case 1}: $z$ is a point where $z^2$ is a not Dirichlet eigenvalue of the Lam\'e operator $L_{\lambda_1,\mu_1}$. In this case, it is known that $S_1(z)$ is invertible from $\mathbf H^{-1/2}(\Gamma)$ to $\mathbf H^{1/2}(\Gamma)$. From this, we can use \eqref{eq:11} to obtain that $u$ satisfies \eqref{eq:7}. Therefore, $u$ specified in  \eqref{eq:12} is a non-zero solution of equations \eqref{eq:6}--\eqref{eq:7}.

\textbf{Case 2}: $z$ is a point where $z^2$ is a Dirichlet eigenvalue of the Lam\'e operator $L_{\lambda_1,\mu_1}$. We note that $z^2 \in \R$ in this case. For every Dirichlet eigenfunction $w^z$ corresponding to $z^2$, we can deduce from Green formulas that 
\begin{align}
w^{z}(x) = \left(SL_1(z)\partial_{\nu,1}w^{z}\right)(x), \quad x\in \Omega. \label{eq:18}
\end{align}
Let $\left\{w_1^{z},\ldots, w_N^{z}\right\}$ denote the orthogonal basis of the Dirichlet eigenfunction space of dimension $N \in \mathbb N$ on $\Omega$ corresponding to the Dirichlet eigenvalue $z^2/\rho_1$. 

Now we assert that the kernel of $S_{1}(z)$ can be characterized by
\begin{align} \label{eq:15}
{\rm{Ker}}\left(S_{1}(z)\right) = \left\{h \in\mathbf H^{-\frac 1 2}(\Gamma): h =\sum_{l=1}^{N}h_l\partial_{\nu, 1} w^z_l\;\textrm{with}\; h_l \in \mathbb C\right\}.
\end{align}
Once \eqref{eq:15} is established, with the aid of \eqref{eq:11}, we obtain that there exists a Dirichlet function $w^z$ such that  
\begin{align} \label{eq:17}
\partial_{\nu, 1}  u =  \mathcal N(z)\gamma u + \partial_{\nu, 1} w^z \quad &\mathrm{on}\; \Gamma.
\end{align}
Since
\begin{align*}
\partial_{\nu, 1} w^z =  \mathcal N(z) \gamma w^z + \partial_{\nu, 1} w^z \quad &\mathrm{on}\; \Gamma,
\end{align*}
\eqref{eq:17} yields 
\begin{align*}
\partial_{\nu,1} ( u - w^z) = \mathcal N(z) \gamma(u - w^z) \quad &\mathrm{on}\;\Gamma.
\end{align*}
It is easy to verify that $u = w^z$ yields $\phi = 0$, which contradicts the non-vanishing condition of $\phi$. Therefore, $u-w^z$ is is a non-zero solution of equations \eqref{eq:6}--\eqref{eq:7}.

In order to prove \eqref{eq:15}, it suffices to prove that for any $\phi \in \textrm{Ker}\left(S_{1}(z)\right)$, 
\begin{align} \label{eq:16} 
\phi(x) = \partial_{\nu,1} g_\phi(x), \quad \textrm{for}\; x\in \Gamma,
\end{align}
where 
\begin{align} \label{eq:26}
g_\phi(x): = \left(SL_{1}(z) \phi\right)(x),\quad  \mathrm{for}\; x \in \Omega.
\end{align}
Since $\phi \in \textrm{Ker}\left(S_{1}(z)\right)$, we find
\begin{align}
S_{1}(z) \partial_{\nu,1} g_\phi = 0 \quad &\rm{on}\; \Gamma. \label{eq:19}
\end{align}
We note that $g_\phi \in \textrm{Span}\left\{w_1^{z},\ldots, w_N^{z}\right\}$. 
Furthermore, from \eqref{eq:18}, we have
\begin{align} 
SL_{1}(z) \partial_{\nu,1} g_\phi = g_\phi \quad &\rm{in}\; \Omega. \label{eq:20}
\end{align}
We define
\begin{align*}
W_\phi(x):=\left(SL_{1}(z) \partial_{\nu,1} g_\phi\right)(x) - \left(SL_{1}(z)\phi\right)(x) \quad \mathrm{for}\; x\in \R^3 \backslash \Gamma.
\end{align*}
It immediately follows from \eqref{eq:26}, \eqref{eq:19} and \eqref{eq:20} that $W_\phi = 0 $ in $\overline \Omega$. Since $W_\phi$ solves Lam\'e equation with the Kupradze radiation condition, we also have 
\begin{align*}
W_\phi =0 \quad \mathrm{in} \quad \R^3 \backslash \overline{\Omega}.
\end{align*}
Therefore, by jump relations of $S_1(z)$, we obtain 
\begin{align*}
\partial^+_{\nu,1} W_\phi - \partial^-_{\nu,1} W_\phi = \phi - \partial_{\nu,1} g_\phi = 0 \quad \rm{on}\; \Gamma,
\end{align*}
whence \eqref{eq:16} follows. 
\end{proof}

\begin{proof}[Proof of Proposition \ref{pro:4}]
We first claim that 

\medskip
\noindent\textbf{Claim 1.}
$\;$For each $h \in\mathbf H^1(\Omega)$, there exists a unique $g^h \in\mathbf H^1(\Omega)$ satisfying
\begin{align}
J(g^h, \psi,z_0) + \left(P(z_0)h,\psi_1\right)_\Omega = \left(h,\psi_1\right)\mathbf H^1(\Omega). \notag
\end{align}
with 
\begin{align*}
  \|g^h\|_{\mathbf H^1(\Omega)} \le C \|h\|_{\mathbf H^1(\Omega)}.
\end{align*}
Here, the sesquilinear form  $J(g^h, \psi;z_0) $ is given by \eqref{eq:31}. 
\hfill$\square_{\text{Claim 1}}$

Claim~1 can be proved by a method similar to that used in the proof of
statement~(a) of Lemma 4.1 in \cite{LS-05}. Using Claim~1 and arguing as in the proof of
statement~(b) Lemma 4.2 in \cite{LS-05}, we obtain statement \eqref{a1} of this lemma.

In what follows, we prove statement \eqref{a2} of this lemma, based on statement \eqref{a1}.

It follows from \eqref{eq:34} and \eqref{eq:33} that for any $\psi_1,\psi_2 \in\mathbf H^1(\Omega)$, 
\begin{align*}
J_{\tau}(\psi_1, \psi_2,z) = J^{\textrm{dom}}_{z_0,\tau}(\psi_1, \psi_2,z) -  \left(P(z_0)\psi_1, \psi_2\right)_\Omega.
\end{align*}
This, together with \eqref{eq:36} gives
\begin{align} 
J^{\textrm{dom}}_{z_0,\tau}(g^h(\tau,z),\psi) - J^{\textrm{dom}}_{z_0,\tau}(g^h_{\textrm{dom}}(\tau,z;z_0),\psi) = \sum^{n(z_0)}_{l=1} J^{\textrm{dom}}_{z_0,\tau}(\eta_l(\tau, z;z_0),\psi) \left(g^h(\tau,z), e^{(l)}(z_0)\right)_\Omega. \notag
\end{align}
By the unique solvability of the sesqulilnear form $J^{\textrm{dom}}_{z_0,\tau}$, we readily obtain
\begin{align} \label{eq:40}
g^h(\tau,z) = g_{\textrm{dom}}^h(\tau,z;z_0) + \sum^{n(z_0)}_{l=1}\eta_l(\tau,z;z_0) \left(g^h(\tau,z), e^{(l)}(z_0)\right)_\Omega \quad \textrm{in}\; \Omega,
\end{align}
Therefore, we obtain that \eqref{eq:38} is solvable if and only if $I - M(\tau,z) $ is invertible and that $g^h({\tau,z})$ satisfies \eqref{eq:39} when $I - M(\tau,z)$ is invertible. 

When $I - M(\tau,z)$ is not invertible, utilizing \eqref{eq:40} again, we obtain that \eqref{eq:41} holds for every $(b_1,\ldots, b_{n(z_0)}) \in \mathrm{Ker}(I - M(\tau,z))$. This finishes the proof of this lemma.
\end{proof}

\section{Proofs of Auxiliary Lemmas}\label{sec:B}

\subsection{Proofs of Lemmas \ref{le:2}--\ref{le:6} in Section \ref{sec:5.1}} \label{sec:B1}

\begin{proof}[Proof of Lemma \ref{le:2}]
It is known that $1/2 I + K_1(z)$ is a Fredholm operator with index $0$ in $\mathcal L(\mathbf H^{1/2}(\Gamma),\mathbf H^{1/2}(\Gamma))$ (see \cite[Proposition 1.3]{AAL}). Due to the discreteness properties of Neumann eigenvalues, we can find $\eta$ such that 
$I/2 + K_1(z)$ is not injective for $z \in B_{\eta}(z_0) \backslash \{z_0\}$, that is, $I /2 + K_1(z)$ is normal at $z_0$. Therefore, by using Gohberg and Sigal theory (see \cite[Chapter 1.2]{AK-09}), we have, in the neighborhood of $z_0$ in $\mathbf C$,
\begin{align} \label{eq:243}
I/2 + K_1(z) = E_{z_0}(z) \left(\mathcal P_0(z_0) + \sum_{l=1}^{n_{z_0}} (z-z_0)^{k_l(z_0)} \mathcal P_l(z_0)\right) F_{z_0}(z), 
\end{align}
where 
\begin{align*}
n_{z_0}:= {\rm{dim}}\left(\textrm{Ker}(I/2 + K_1  (z))\right),
\end{align*}
$\mathcal P_1(z_0),\ldots, \mathcal P_{n_{z_0}}(z_0) \in \mathcal L(\mathbf H^{1/2}(\Gamma),\mathbf H^{1/2}(\Gamma))$ are mutually disjoint one-dimensional projections, satisfying 
\begin{align*}
\sum_{l=0}^{n_{z_0}} \mathcal P_l(z_0) = I,
\end{align*}
$k_1(z_0)\ldots, k_{n_{z_0}}(z_0)\in \mathbb N$, and $E_{z_0}(z)$ and $F_{z_0}(z)$ are both holomorphic and invertible near $z_0$ in $\mathcal L(\mathbf H^{1/2}(\Gamma),\mathbf H^{1/2}(\Gamma))$. Furthermore, it follows from Proposition \ref{pro:3} that 
\begin{align*}
n_{z_0} =  n(z_0).
\end{align*}

In order to prove this lemma, it suffices to prove that in the neighborhood of $z_0$ in $\mathbf C$,
\begin{align} \label{eq:244}
I/2 + K^*_1(z) = 
\begin{cases}
\widetilde E_{z_0}(z) \left(\widetilde {\mathcal P}_0(z_0) + \sum_{l=1}^{n({z_0})} (z-z_0) \widetilde {\mathcal P}_l(z_0)\right) \widetilde F_{z_0}(z),  & \mathrm{if}\; z_0 \ne 0,\\
\widetilde E_{0}(z) \left(\widetilde {\mathcal P}_0(0) + \sum_{l=1}^{6} z^2 \widetilde {\mathcal P}_l(0)\right) \widetilde F_{0}(z),  & \mathrm{if}\; z_0 = 0,
\end{cases}
\end{align}
where $\widetilde {\mathcal P}_1(z_0),\ldots, \widetilde {\mathcal P}_{n{(z_0)}}(z_0) \in \mathcal L(\mathbf H^{-1/2}(\Gamma),\mathbf H^{-1/2}(\Gamma))$ are mutually disjoint one-dimensional projections, satisfying 
\begin{align*}
\sum_{l=0}^{n({z_0})} \widetilde {\mathcal P}_l(z_0) = I,
\end{align*}
$\widetilde E_{z_0}(z)$ and $\widetilde F_{z_0}(z)$ are both holomorphic and invertible near $z_0$ in $\mathcal L(\mathbf H^{-1/2}(\Gamma),\mathbf H^{-1/2}(\Gamma))$. Once \eqref{eq:244} holds, with the aid of the fact that $\left(K_1^*(-r)\right)^{\dagger} = K_1(r)$ under the duality pairing $\langle \cdot \rangle_{\Gamma}$ for $r \in \R$,
we have that in the neighborhood of $z_0$ in $\R$,
\begin{align} \label{eq:255}
 I/2 + K_1(r) = \begin{cases}
\left(\widetilde F_{z_0}(r)\right)^{\dagger}\left(\left(\widetilde {\mathcal P}_0(z_0)\right)^{\dagger}  + \sum_{l=1}^{n({z_0})} (r-z_0) \left(\widetilde {\mathcal P}_l(z_0)\right)^{\dagger}\right) \left(\widetilde E_{z_0}(r)\right)^{\dagger},  & \mathrm{if}\; z_0 \ne 0,\\
\left(\widetilde F_{0}(r)\right)^{\dagger}\left(\left(\widetilde {\mathcal P}_0(0)\right)^{\dagger} + \sum_{l=1}^{6} r^2 \left(\widetilde {\mathcal P}_l(0)\right)^{\dagger} \right) \left(\widetilde E_{0}(r)\right)^{\dagger} ,  & \mathrm{if}\; z_0 = 0,
\end{cases}
\end{align}
Applying $F^{-1}_{z_0}(z)\mathcal P_{l}(z_0)$ (for $l\in \{1,\ldots, n(z_0)\}$) to both sides of \eqref{eq:243}, dividing by $z-z_0$, and let $z$ tends to $z_0$ along the real axis, we can deduce from \eqref{eq:244} that    
\begin{align}
k_l(z_0) = \begin{cases}
1, & \mathrm{if}\; z_0 \ne 0,\\
2, & \mathrm{if}\; z_0 = 0,
\end{cases}, \quad \mathrm{for}\; l\in \{1,\ldots, n(z_0)\}. \notag
\end{align}

It should be noted that proceeding as in the derivation of \eqref{eq:243}, the Gohberg and Sigal theory yields an analogous expansion to \eqref{eq:244}, with the only difference that the  monomial factor becomes $(z-z_0)^{\widetilde k_l(z_0)}$ for some $\widetilde k_l(z_0) \in \mathbb N$. In the sequel, we prove \eqref{eq:244} in two parts, corresponding to cases $z_0 \ne 0$ and $z_0 = 0$. 

\textbf{Part I}: In order to derive  for the case of $z_0 \ne 0$, it is sufficient to prove that if
\begin{align}
\left[\frac 1 2 I + K^*_1(z)\right] \phi_{z_0}^{(0)}(z) = (z-z_0)^m
\phi_{z_0}^{(1)}(z) \quad m \in \mathbb N, \notag
\end{align}
then we have 
\begin{align} \label{eq:25}
  m = 1.
\end{align}
Here, $\phi_{z_0}^{(0)}(z)$ and $\phi_{z_0}^{(1)}(z)$ being holomorphic near $z_0$, and $\phi_{z_0}^{(j)}(z_0)\ne {0}$ for $j\in\{0,1\}$,
Once this statement is verified, we can easily obtain $\widetilde k_l(z_0) = 1$ by applying $I/2 + K^*_1(z)$ with $\widetilde F^{-1}_{z_0}(z) \widetilde {\mathcal P}_{l}(z_0)$ for each $l\in \{1,\ldots, n(z_0)\}$. 

Now we prove $m=1$ from the formula \eqref{eq:13}. We set
\begin{align}
\psi_{z_0}^{(2)}(z) = SL_{1}(z) \phi_{z_0}^{(0)}(z) \quad \textrm{in}\; \Omega. \label{eq:21} 
\end{align}
Clearly,
\begin{align} 
& L_{\lambda_1,\mu_1} \psi_{z_0}^{(2)}(z_0) + z_0^2\rho_1\psi_{z_0}^{(2)}(z_0) = 0 \quad \rm{in}\; \Omega, \label{eq:22}\\ 
&\partial_{\nu,1} \psi_{z_0}^{(2)}(z_0) = 0 \qquad\qquad\qquad\qquad\;\;\mathrm{on}\; \Gamma.\label{eq:23}
\end{align}
Furthermore, it follows from \eqref{eq:21} that 
\begin{align*}
&L_{\lambda_1,\mu_1} \psi_{z_0}^{(2)}(z) + z^2\rho_1\psi_{z_0}^{(2)}(z) = 0 \quad \rm{in}\; \Omega,\\
& \partial_{\nu,1} \psi_{z_0}^{(2)}(z) = (z-z_0)^m \phi_{z_0}^{(1)}(z) \quad\;\; \rm{on}\; \Gamma.
\end{align*}
Thus, with the aid of \eqref{eq:22} and \eqref{eq:23}, integrating by parts yields
\begin{align*}
&\int_{\Gamma} (z-z_0)^m (\phi_{z_0}^{(1)}(z))(y) \cdot \left(\overline{\psi_{z_0}^{(2)}(z_0)}\right)(y) dS(y)\\ 
&= \int_{\Gamma}{\partial_{\nu,1} \psi_{z_0}^{(2)}(z)(y)} \cdot \left(\overline{\psi_{z_0}^{(2)}(z_0)}\right)(y) -  \left(\partial_{\nu,1}\overline{\psi_{z_0}^{(2)}(z_0)}\right)(y) \cdot {\psi_{z_0}^{(2)}}(z)(y) dS(y) \\
& = \int_{\Omega} \frac{{z^2_0} - z^2}{c_1^2} \left(\psi_{z_0}^{(2)}(z) \cdot \overline{\psi_{z_0}^{(2)}(z_0)}\right)(y)dy.
\end{align*}
If $m \ge 2$, let $z$ tend to $z_0$, we directly get
\begin{align*}
\int_{\Omega} \left(\left|\psi_{z_0}^{(2)}({z_0})\right|^2\right)(y) dy = 0
\end{align*}
which implies $ \psi_{z_0}^{(2)}(z_0) = 0$ in $\Omega$. This, together with \eqref{eq:21}, \eqref{eq:22}, \eqref{eq:23}, jump relations of the single-layer potential and the well-posedness of the exterior scattering problem gives $\phi_{z_0}^{(0)}(z_0) = 0$, which is a contradiction to the nonzero assumption of $\phi_{z_0}^{(0)}(z_0)$. Therefore, $m$ must be equal to $1$.

\textbf{Part 2}: Proceeding as in the proof of \eqref{eq:25}, we can obtain that if 
\begin{align}
\left[\frac 1 2 I + K^*_1(z)\right] \phi^{(0)}(z) = z^m
\phi^{(1)}(z) \quad m \in \mathbb N, \notag
\end{align}
with $\phi^{(0)}(z)$ and $\phi^{(1)}(z)$ being holomorphic near $0$, and $\phi^{(j)}(z_0)\ne {0}$ for $j\in\{0,1\}$, then we have 
\begin{align}
  m = 1\; \textrm{or}\;2. \label{eq:29}
\end{align}
Furthermore, it is clear that 
\begin{align*}
\mathrm{Ker}\left(I/2 + K^*_1(0)\right) = \left(\widetilde F_0(0)\right)^{-1} \mathrm{Ran}\left(\sum^6_{l=1}\widetilde {\mathcal P}_l(0)\right). 
\end{align*}
This, together with the asymptotic expansion of $K^*_1(z)$ around $0$, that is, $K^*_1(z) = K^*_1(0) + O(z^2)$ gives 
\begin{align*}
\left(I/2 + K^*_1(z)\right) \left(\widetilde F_0(0)\right)^{-1} a_l = z^2 \mathrm{Rem}(z), \quad a_l \in \mathrm{Ran}(\widetilde {\mathcal P}_l)\; \mathrm{for}\; l \in \{1,\ldots, 6\}.
\end{align*}
In conjunction with \eqref{eq:29} yields $\widetilde k_1(0)=\cdots= \widetilde k_6(0) = 2$. 
\end{proof}

\begin{proof}[Proof of Lemma \ref{le:5}]
Based on Proposition \ref{pro:4}, in order to characterize the asymptotic behaviour of $M(\tau, z ; 0)$, it is sufficient to analyze the asymptotic properties of $\eta_l(\tau, z ; z_0)$, where $\eta_l(\tau,z;z_0)$ solves \eqref{eq:36}. In fact, with the aid of statement \eqref{a2} of Proposition \ref{pro:4}, it can be seen that $\eta_l(\tau, z)$ satisfies
\begin{align}
&\left[L_{\lambda_1,\mu_1,\rho_1} + z_0^2 \rho_1  + (z-z_0)(2z_0 + (z-z_0)) \rho_1 +  P(z_0) \right] \eta_l(\tau, z; z_0) = e^{(l)}(z_0) \quad  \textrm{in}\; \Omega, \label{eq:42}\\
&\partial_{\nu,1} \eta_l(\tau, z ; z_0) = \tau \mathcal N(z)\gamma\eta_l(\tau, z;z_0) \quad \textrm{on}\; \Gamma, \label{eq:43}
\end{align}
and admits the expansion 
\begin{align}
&\eta_l(\tau, z ; z_0) = \sum^{\infty}_{q_1 = 0}\sum^{\infty}_{q_2 = 0} \tau^{q_1}(z-z_0)^{q_2} \eta_{l,q_1,q_2}(z_0) \quad \textrm{in some neighborhood of}\; (0,z_0). \label{eq:44}
\end{align}
Furthermore, due to the analyticity of $ \mathcal N$ with respect to $z$, we have 
\begin{align}
\mathcal N(z) = \sum^{+\infty}_{q = 0} \frac{{(z-z_0)}^{q}}{q!} \partial^q_z  \mathcal N(z)|_{z_0} \quad \textrm{in some neighborhood of}\; z_0. \notag 
\end{align}
The remainder of the proof proceeds in two parts, corresponding to statements \eqref{d1} and \eqref{d2}.

\textbf{Part 1}:
Inserting \eqref{eq:44} into equations \eqref{eq:42}--\eqref{eq:43}, and equating the coefficients for $\tau^{q_1}(z-z_0)^{q_2}$ with $(q_1,q_2) = (0,0), (0,1), (1,0)$, we obtain that
\begin{align}
&\left[L_{\lambda_1,\mu_1} + z^2_0 \rho_1 +P(z_0)\right]\eta_{l,0,0}(z_0) = e^{(l)}(z_0) \qquad\qquad\qquad\qquad\; \mathrm{in}\; \Omega, \label{eq:45} \\
&\partial_{\nu,1} \eta_{l,0,0}(z_0) = 0 \qquad\qquad\qquad\quad\qquad\qquad\qquad\qquad\qquad\qquad\;\;\; \mathrm{on}\; \Gamma, \label{eq:46}\\
& \left[L_{\lambda_1,\mu_1} + z_0^2\rho_1 +P(z_0)\right]\eta_{l,0,1}(z_0) + 2{z_0}{\rho_1} \eta_{l,0,0}(z_0) = 0 \qquad\quad\;\;\mathrm{in}\; \Omega, \label{eq:47}\\
& \partial_{\nu,1} \eta_{l,0,1}(z_0) = 0 \qquad\qquad\qquad\qquad\qquad\qquad\quad\qquad\qquad\qquad\;\;\;\mathrm{on}\; \Gamma, \label{eq:48}\\
& \mathrm{and}\; \left[L_{\lambda_1,\mu_1} + z^2_0 \rho_1 +P(z_0)\right]\eta_{l,1,0} (z_0) = 0 \qquad\qquad\qquad\qquad\quad\mathrm{in}\; \Omega, \label{eq:49} \\
& \partial_{\nu,1} \eta_{l,1,0} (z_0) =  \mathcal N(z_0)\gamma \eta_{l,0,0}(z_0) \qquad \qquad\qquad\qquad\qquad\qquad\quad\;\; \mathrm{on}\; \Gamma,\label{eq:50}
\end{align}
In conjunction with \eqref{eq:45}, \eqref{eq:46}, \eqref{eq:47} and \eqref{eq:48}, we find
\begin{align}\label{eq:51}
\eta_{l,0,0}(z_0) = e^{(l)}(z_0), \quad \eta_{l,0,1}(z_0) = - {2z_0}{\rho_1} e^{(l)}(z_0).
\end{align}
Moreover, building on \eqref{eq:49}, \eqref{eq:50}, \eqref{eq:51} and the fact that $e^{(j)}(z_0) \subset \Lambda_N(|z_0|)$, we arrive at
\begin{align}
\left[\Pi(z_0)\eta_{l,1,0}(z_0)\right]_j = - M^{(1)}_{lj}(z_0). \quad  \notag
\end{align}
This, together with \eqref{eq:35}, \eqref{eq:44} and \eqref{eq:51} gives the asymptotic expansion \eqref{eq:27}.

\textbf{Part 2}: Inserting \eqref{eq:44} into \eqref{eq:42}--\eqref{eq:43}, and equating the coefficients for $\tau^{q_1} z^{q_2}$, with $(q_1,q_2) = (0,0), (0,1)$, $(0,2), (1,0), (1,1), (0,3), (1,2), (0,4), (1,3), (0,5), (2,0), (2,1)$,  we obtain that equations \eqref{eq:45}--\eqref{eq:50} also hold in the case $z_0 = 0$ and that 
\begin{align}
&\left[L_{\lambda_1,\mu_1}  + P(0)\right]\eta_{l,0,j_1}(0) + \rho_1\eta_{l,0,j_1-2}(0) = 0 \quad \mathrm{in}\; \Omega, \label{eq:113} \\
&\partial_{\nu,1} \eta_{l,0,j_1}(0) = 0  \qquad\qquad\qquad\qquad\qquad\qquad\quad \mathrm{on}\; \Gamma, \; j_1 \in \{2,3,4,5\}, \label{eq:114}\\
&\left[L_{\lambda_1,\mu_1} + P(0)\right]\eta_{l,1,j_2}(0) + \rho_1 \eta_{l,1,j_2-2}(0)  = 0 \quad  \mathrm{in}\; \Omega, \label{eq:115} \\
&\partial_{\nu,1} \eta_{l,1,j_2}(0) = \sum_{q=0}^{j_2}\frac {1}{q!}\partial^q_z \mathcal N(0)\gamma \eta_{l,0,j_2-q}(0)  \qquad\;\; \mathrm{on}\; \Gamma, \; j_2\in \{1,2,3\}, \label{eq:116} \\
&\mathrm{and}\; \left[L_{\lambda_1,\mu_1} + P(0)\right]\eta_{l,2,j_3}(0) = 0 \qquad\qquad\qquad \mathrm{in}\; \Omega, \label{eq:117}\\
&\partial_{\nu,1} \eta_{l,2,j_3}(0) = \sum^{j_3}_{q=0} \partial^q_z\mathcal N(0)\gamma \eta_{l,1,j_3 - q}(0) \qquad\quad\;\;\; \mathrm{on}\; \Gamma, \; j_3\in \{0,1\},\label{eq:118}
\end{align}
Here, we set $\eta_{l,1,-1}(0) = 0$. By solving equations \eqref{eq:45}--\eqref{eq:48} for the case of $z_0 = 0$ and equations \eqref{eq:113}--\eqref{eq:114}, we obtain that
\begin{align}
&\eta_{l,0,1}(0) = \eta_{l,0,3}(0) = \eta_{l,0,5}(0) =0, \label{eq:91}\\
&\mathrm{and}\; \eta_{l,0,0}(0) = e^{(l)}(0), \quad \eta_{l,0,2}(0) = -\rho_1 e^{(l)}(0), \quad \eta_{l,0,4}(0) = \rho_1^2 e^{(l)}(0).  \label{eq:92}
\end{align}
Using \eqref{eq:91} and \eqref{eq:92}, we deduce from equations \eqref{eq:49}--\eqref{eq:50} for the case of $z_0 = 0$ and equations \eqref{eq:115}--\eqref{eq:116} that 
\begin{align}
&\left[\Pi(0) \eta_{l,1,q_1}(0)\right]_j = - \left\langle \partial^{q_1}_z \mathcal N(0) e^{(l)}(0), e^{(j)}(0)\right \rangle_{\Gamma},\quad q_1  \in \{0,1\},  \notag 
\end{align}
and
\begin{align}
\left[\Pi(0) \eta_{l,1,q_2}(0)\right]_j & = - \frac{1}{q_2!}\left\langle \partial^{q_2}_z \mathcal N(0) e^{(l)}(0),  e^{(j)}(0)\right\rangle_{\Gamma} - \rho_1\left[\Pi(0)\eta_{l,1,q_2-2}(0)\right]_j\notag \\
& \;\; + \rho_1 \left \langle \partial^{q_2-2}_z \mathcal N(0) e^{(l)}(0), e^{(j)}(0)\right\rangle_{\Gamma},\quad q_2 \in \{2,3\}. \notag
\end{align}
Moreover, from \eqref{eq:91}, together with equations \eqref{eq:117}--\eqref{eq:118} for $j_3 = 0$, we infer that 
\begin{align}
&\left[\Pi(0)\eta_{l,2,0}(0)\right]_{j} = -\left\langle \eta_{l,1,0}(0), \mathcal N(0) e^{(j)}(0) \right\rangle_{\Gamma} =  -\left\langle\eta_{l,1,0}(0), \partial_{\nu,1}\eta_{j,1,0}(0) \right\rangle_{\Gamma}. \notag
\end{align}
Applying the integral identities \eqref{eq:32} and \eqref{eq:94}, and combining \eqref{eq:91} and equations \eqref{eq:117}--\eqref{eq:118} for $j_3 = 1$, we have 
\begin{align}
& \left[\Pi(0)\eta_{l,2,1}(0)\right]_{j} = -\left\langle \eta_{l,1,1}(0), \mathcal N(0) e^{(j)}(0) \right\rangle_{\Gamma} - \left\langle  \eta_{l,1,0}(0), \partial_z \mathcal N(0)e^{(j)}(0) \right\rangle_{\Gamma} \notag\\
&\qquad\qquad\qquad\; =  -\left\langle \eta_{l,1,1}(0), \partial_{\nu,1}\eta_{j,1,0}(0) \right\rangle_{\Gamma} -\left\langle \eta_{l,1,0}(0), \partial_{\nu,1}\eta_{j,1,1}(0)\right\rangle_{\Gamma}\notag\\
&\qquad\qquad\qquad\; =  -\left\langle \partial_{\nu,1}\eta_{l,1,1}(0), \eta_{j,1,0}(0) \right\rangle_{\Gamma} -\left\langle \eta_{l,1,0}(0), \partial_{\nu,1}\eta_{j,1,1}(0)\right\rangle_{\Gamma}. \notag
\end{align}
Building upon the above calculations, we readily obtain the asymptotic expansion \eqref{eq:28}.
\end{proof}

\begin{proof}[Proof of Lemma \ref{le:7}]
We denote by $V_{\kappa^{(l)}(z_0)}$ the eigenspace corresponding to $\kappa^{(l)}(z_0)$. For any nonzero vector $(a_1,\ldots, a_{n(z_0)})^T \in V_{\kappa^{(l)}(z_0)}$, then 
\begin{align}
M^{(1)}({z_0})(a_1,\ldots, a_{n(z_0)})^T = \kappa^{(l)}(z_0) (a_1,\ldots, a_{n(z_0)})^T. \notag
\end{align}
This directly implies 
\begin{align} \label{eq:53}
(\overline{a_1},\ldots, \overline{a_{n(z_0)}}) M^{(1)}(z_0) (a_1,\ldots, a_{n(z_0)})^{T} = \kappa^{(l)}(z_0)\sum^{n(z_0)}_{l=1}|a_l|^2.
\end{align}
We set $w = \sum^{n(z_0)}_{l=1} a_l e^{(l)}(z_0)$ and consider the subsequent scattering problem
\begin{align*}
&L_{\lambda_0,\mu_0} w^{sc} + z_0^2\rho_0 w^{sc} = 0 \quad \textrm{in}\; \R^3 \backslash \Omega,\\
& w^{sc} = \gamma w \quad \textrm{on}\; \Gamma,\\
& w^{sc}\; \textrm{is}\; z_0-\textrm{outgoing}.
\end{align*}
If ${\rm{Im}}(\kappa^{(l)}({z_0})) \le 0$, we deduce from \eqref{eq:53} that
\begin{align} \label{eq:82}
\textrm{Im} \left(\int_{\Gamma} \partial_{\nu,0} w^{sc}(y) \cdot \overline{w^{sc}}(y) dS(y)\right) \le 0.
\end{align}
It is well known that $w^{sc}$ can be decomposed as $w^{sc} = w_{p}^{sc} + w_s^{sc}$, where $w_{p}^{sc}$ and $w_s^{sc}$ are solutions of Helmholtz equations in the exterior domain $\R^3 \backslash \overline{\Omega}$ with wave number $\omega/c_{p,0}$ and $\omega/c_{s,0}$, respectively, and are subject to the respective Sommerfeld radiation condition. Also (see \cite[Chapter 2]{KGBB}),
\begin{align*}
&\left(\partial_{\nu,0} - i \omega(\lambda_0 + 2\mu_0)/c_{p,0}\right) w_p^{sc}= O(r^{-2}),\\
&\left(\partial_{\nu,0} - i \omega \mu/c_{s,0}\right) w_s^{sc} = O(r^{-2}),\\
& w_p^{sc} \cdot \overline{w_s^{sc}} = O(r^{-3}), \quad  w_p^{sc} \cdot {w_s^{sc}} = O(r^{-3}), \quad \mathrm{as}\; r = |x| \rightarrow \infty.
\end{align*}
Therefore, \eqref{eq:82} yields 
\begin{align*}
\lim_{r\rightarrow \infty }\int_{\partial B_r} \left|w_p^{sc}(x)\right|^2 dS(x) = 0, \quad \lim_{r\rightarrow \infty }\int_{\partial B_r} \left|w_s^{sc}(x)\right|^2dS(x) = 0.
\end{align*}
whence follows from Rellich lemma that  
\begin{align*}
w^{sc}(y)=0 \quad \mathrm{in}\; \R^3\backslash \Omega,
\end{align*}
This directly yields $w = 0$ on $\Gamma$. Therefore, by Holmgren's theorem, we obtain 
\begin{align*}
w=0\quad \mathrm{in}\; \Omega,
\end{align*}
which implies $a_1 =\cdots= a_{n(z_0)} = 0$. This leads to a contradiction. Hence, we obtain \eqref{eq:54}.
\end{proof}

\begin{proof}[Proof of Lemma \ref{le:6}]
\eqref{b2} 
We first prove that $M^{(1)}(0)$ is a negative matrix. 

With the aid of the identity \eqref{eq:32},
we obtain  
\begin{align}
M^{(1)}_{lj}(0) =  - \left\langle\left(S_0(0)\right)^{-1} e^{(l)}(0), e^{(j)}(0)\right\rangle_{\Gamma},
\quad j,l \in \left\{1,\ldots,6\right\}. \label{eq:100}
\end{align}
Therefore, for $ a= (a_1,\ldots,a_6)^T \in \R^6$, we have 
\begin{align} \label{eq:55}
a^T M^{(1)}(0) a =  - \left\langle\left(S_0(0)\right)^{-1} \gamma w_0, \gamma w_0\right\rangle_{\Gamma}\le -C\|w_0\|_{H^{1/2}(\Gamma)}.
\end{align}
Here, 
\begin{align*}
w_0 = \sum^{6}_{l=1} a_l e^{(l)}(0).
\end{align*}
Furthermore, it can be seen that $\gamma w_0 = 0 $ yields $a_1 = \cdots = a_6 = 0$. This, together with \eqref{eq:55} yields that $M^{(1)}(0)$ is negative. 

Second, we prove that $M^{(2)}(0)$ is non-positive. With the aid of \eqref{eq:94}, 
it can be deduced from \eqref{eq:56} that 
\begin{align}
M^{(2)}_{jl}&(0)= - \frac{\sqrt{\rho_0}}{12\pi}\left(\frac{2}{\mu^{3/2}_0} + \frac{1}{(\lambda_0 + 2 \mu_0)^{3/2}}\right) \notag\\ 
&\sum^3_{p=1} \left\langle \left(S_0(0)\right)^{-1} e^{(l)}(0), e^{(p)}(0)\right\rangle_{\Gamma} \left\langle \left(S_0(0)\right)^{-1} e^{(j)}(0), e^{(p)}(0)\right\rangle_{\Gamma}, 
\;\; j,l \in \left\{1,\ldots,6\right\}. \notag
\end{align}
Therefore, a straightforward calculation gives that for $a = (a_1,\ldots,a_6)^T \in \R^6$,
\begin{align}
&\frac{12\pi}{\sqrt{\rho_0}}\left(\frac{2}{\mu^{3/2}_0} + \frac{1}{(\lambda_0 + 2 \mu_0)^{3/2}}\right)^{-1}a^T M^{(2)}(0)a \notag\\
& = - \sum_{j=1}^{6}a_j\sum^6_{l=1}a_l \sum^3_{p=1} \left\langle \left(S_0(0)\right)^{-1} e^{(l)}(0), e^{(p)}(0)\right\rangle_{\Gamma} \left\langle \left(S_0(0)\right)^{-1} e^{(j)}(0), e^{(p)}(0)\right\rangle_{\Gamma} \notag\\
& = - \sum^3_{p=1} \left[\sum^6_{l=1}a_l\left\langle \left(S_0(0)\right)^{-1} e^{(l)}(0), e^{(p)}(0) \right\rangle_{\Gamma}\right]^2 \le 0. \label{eq:101}
\end{align}
This implies that $M^{(2)}(0)$ is non-positive. 

Third, given $ a = (a_1,\ldots, a_6)^T \in \mathrm{Ker}\left(M^{(2)}(0)\right) \backslash \{0\}$, it easily follows from \eqref{eq:101} that 
\begin{align}
\sum^6_{l=1}a_l\left\langle \left(S_0(0)\right)^{-1} e^{(l)}(0), e^{(p)}(0) \right\rangle_{\Gamma} = 0, \quad p \in \{1,2,3\}. \label{eq:102}
\end{align}
If there exists $\kappa \ne 0$ such that 
\begin{align*}
M^{(1)}(0) a = \kappa a,
\end{align*}
using \eqref{eq:100} and \eqref{eq:102}, we obtain 
$a_1 = a_2 = a_3= 0$. This finishes the proof of this statement.

\eqref{b3} First, it immediately follows from the definition of $M^{(3)}(0)$ that $M^{(3)}(0)$ is a symmetric matrix.

Second, since 
\begin{align*}
\mathcal N(z) = - \left(S^{(0)}(z)\right)^{-1} + (I/2 + K^*_0(z))\left(S^{0}(z)\right)^{-1}\; \textrm{in the neighborhood of}\; 0,
\end{align*}
it can be seen that 
\begin{align} \label{eq:138}
\partial^2_z \mathcal N(0) = -\left(S_0(0)\right)^{-1} S^{(2)} \left(S_0(0)\right)^{-1} + \mathcal N^{\mathrm{in}},
\end{align}
where the operator $S^{(2)}:\mathbf H^{1/2}(\Gamma) \rightarrow\mathbf H^{1/2}(\Gamma)$ is given by 
\begin{align*}
\left(S^{(2)}\phi\right)(x):= \int_{\Gamma} s_0(x,y) \phi(y) dS(y), \quad x\in \Gamma.
\end{align*}
Here, the $(k,l)$ entry of $s_1(x,y)$ (for $k,l \in \{1,\ldots, 6\}$)
\begin{align*}
s_0(x,y)&:= -\frac{3\delta_{k,l}}{32\pi \mu_0 c^2_{s,0}} |x-y| - \frac{\delta_{k,l}}{32\pi (\lambda_0 + 2 \mu_0) c^2_{p,0}} |x-y| \\
&+\frac{(x_k-y_k)(x_l-y_l)}{32\pi \mu_0 c^2_{s,0}|x-y|} - \frac{(x_k-y_k)(x_l-y_l)}{32\pi (\lambda_0 +2\mu_0)c^2_{p,0}|x-y|}. 
\end{align*}
and $\mathcal N^{\mathrm{in}} f : = \partial_{\nu_1} v^f$, where $v^f$ solves
\begin{align}
L_{\lambda_1,\mu_1} v^f + SL_0 \left(S_0(0)\right)^{-1}f & = 0 \qquad\; \mathrm{in}\; \Omega, \notag\\
v^f &= 0 \quad\quad\;\mathrm{on}\; \Gamma, \notag
\end{align}
Obviously, 
\begin{align*}
SL_0\left(S_0(0)\right)^{-1} e_0^{(k)} = e^{(k)}(0)\;\;\; \mathrm{in} \; \Omega\;\;  \mathrm{for}\; k\in \{1,\cdots,6\}, 
\end{align*}
with the aid of Green formulas, we have
\begin{align*}
\left\langle \mathcal N^{\mathrm{in}} e^{(k)}(0), e^{(j)}(0)\right\rangle_{\Gamma} = \left\langle  e^{(k)}(0), \mathcal N^{\mathrm{in}}e^{(j)}(0)\right\rangle_{\Gamma} \quad  \mathrm{for}\; k,j\in \{1,\cdots,6\}.
\end{align*}
This, together with \eqref{eq:138} gives the symmetric properties of $M^{(4)}$.

Third, similar to the derivation of \eqref{eq:138}, we have
\begin{align} \label{eq:139}
\partial^3_z \mathcal N(0) = i \left(S_0(0)\right)^{-1} S^{(3)} \left(S_0(0)\right)^{-1},
\end{align}
where the operator $S^{(3)}:\mathbf H^{1/2}(\Gamma) \rightarrow\mathbf H^{1/2}(\Gamma)$ is given by 
\begin{align*}
\left(S^{(3)}\phi\right)(x):= \int_{\Gamma} s_1(x,y) \phi(y) dS(y), \quad x\in \Gamma.
\end{align*}
Here, the $(k,l)$ entry of $s_1(x,y)$ (for $k,l \in \{1,\ldots, 6\}$) is specified by 
\begin{align}
s_1(x,y) &:= -\left[\frac{1}{30\pi \mu_0 c^3_{s,0}} +\frac1{120\pi (\lambda_0 + 2 \mu_0)c^3_{p,0}} \right]\delta_{k,l}|x-y|^2\notag\\
& + \left[\frac 1{60\pi \mu_0 c^3_{s,0}}- \frac{1}{60\pi(\lambda_0 +2\mu_0)c^3_{p,0}}\right]{(x_k-y_k)(x_l-y_l)}\notag\\
&=: A_1\delta_{k,l}|x-y|^2 + A_2 (x_k-y_k)(x_l-y_l). \label{eq:107}
\end{align}
In conjunction with \eqref{eq:52}, \eqref{eq:139}, \eqref{eq:107} and the fact that $M^{(2)}(0)$ is a symmetric matrix, we obtain that $M^{(5)}(0)$ is symmetric.

\eqref{b4} We first prove \eqref{eq:98}. Due to the assumption of $a \in \mathrm{Ker}\left(M^{(2)}(0)\right)$ and the expression of $M^{(5)}(0)$ in \eqref{eq:52}, it suffices to prove 
\begin{align*}
a^T M^{(5)}(0)a = \frac{i}6 \left\langle \partial^3_z \mathcal N(0)\left(\sum^6_{k=1} a_k e^{(k)}({0})\right), \left(\sum^6_{l=1} a_l e^{(l)}(0)\right)\right\rangle_{\Gamma} < 0.
\end{align*}
With the aid of \eqref{eq:139} and \eqref{eq:107}, we can rewrite $a^TM^{(5)}a$ as 
\begin{align}
a^TM^{(5)}(0)a 
& = -\frac{1}6 \left\langle S^{(3)}g,g \right\rangle_{\Gamma}. \label{eq:103}
\end{align}
Here,
\begin{align} \label{eq:152}
g = \left(S_0(0)\right)^{-1}\sum^6_{l=1}a_l e^{(l)}(0).
\end{align}
Since $a \in \mathrm{Ker}\left(M^{(2)}(0)\right)$, \eqref{eq:102} yields 
\begin{align*}
\int_{\Gamma} g(x) dS(x) = 0.
\end{align*}
This leads to 
\begin{align}
\sum^3_{l=1}\int_{\Gamma} \int_{\Gamma} |x-y|^2 g_l(x)g_l(y) dS(x) dS(y) &= -2 \sum^3_{l=1}\int_{\Gamma} \int_{\Gamma}x \cdot y g_l(x)g_l(y) dS(x) dS(y) \notag \\
& =-2 \sum^3_{l=1}\sum^3_{p=1}\left(\int_{\Gamma} x_p g_l(x) dS(x)\right)^2 \label{eq:104}
\end{align}
and 
\begin{align}
&\sum^{3}_{l=1}\sum^{3}_{k=1}\int_{\Gamma} \int_{\Gamma} (x_k - y_k)(x_l -y_l)g_k(x) g_l(y) dS(x) dS(y)\notag \\
&= - \left(\int_{\Gamma}\sum^{3}_{l=1} x_l g_l(x) dS(x)\right)^2  - \sum^{3}_{l=1}\sum^{3}_{k=1}\left(\int_{\Gamma} g_l(x) x_k dS(x)\right) \left(\int_{\Gamma} g_k(x) x_ldS(x)\right). \label{eq:105}
\end{align}
Using \eqref{eq:107}, \eqref{eq:103}, \eqref{eq:104} and \eqref{eq:105}, we obtain 
\begin{align} \label{eq:106}
a^T M^{(5)}(0)a = -\frac{1}6 \left[-2A_1\|F\|^2_{\mathbb F} - A_2 \left(\mathrm{tr}\;(F)\right)^2 -A_2 \mathrm{tr}\left(F^2\right) \right].
\end{align}
Here, the $(k,l)$ entry of the matrix $F$ is given by 
\begin{align*}
F_{kl} := \int_{\Gamma }x_k g_l(x) dS(x), \quad k,l \in \{1,\ldots,3\},
\end{align*}
and $\|\cdot\|_{\mathbb F}$ denotes the Frobenius norm of a matrix. We proceed to split $F$ into 
\begin{align*}
F = \frac{F + F^T}2 + \frac{F - F^T}2 =:F_{\mathrm{sym}} + F_{\mathrm{asym}}.
\end{align*}
Then, it can be seen that 
\begin{align*}
\mathrm{tr}\;(F) = \mathrm{tr}\left(F_{\mathrm{sym}}\right), \quad \mathrm{tr}\left(F^2\right) = \mathrm{tr}\left(F^2_{\mathrm{sym}}\right) + \mathrm{tr}\left(F^2_{\mathrm{asym}}\right) = \left\| F_{\mathrm{sym}}\right\|^2_{\mathbb F} + \left\| F_{\mathrm{asym}}\right\|^2_{\mathbb F} =  \left\| F\right\|^2_{\mathbb F} .
\end{align*}
Combining this with \eqref{eq:106} gives  
\begin{align}
    -6a^T M^{(5)}(0) a &= \left(-2A_1 -A_2\right)\left\|F_{\mathrm{sym}}\right\|^2_{\mathbb F} + \left(-2A_1 -A_2\right)\left\|F_{\mathrm{asym}}\right\|^2_{\mathbb F} - A_2\mathrm{tr}\left(F_{\mathrm{sym}})\right)^2 \notag\\
    & \ge (-2A_1-4 A_2)\left\|F_{\mathrm{sym}}\right\|^2_{\mathbb F} + \left(-2A_1 - A_2\right)\left\|F_{\mathrm{asym}}\right\|^2_{\mathbb F}. \label{eq:108}
\end{align}
Here, the last inequality follows from the basic trace identity $(\mathrm{tr}\;(B))^2 \le 3 \|B\|^2_{\mathbb F}$ for every $3\times 3$ symmetric matrix. Utilizing the expressions for $A_1$ and $A_2$ in \eqref{eq:107}, we easily obtain 
\begin{align*}
-2A_1 - A_2 > -2A_1 - 4A_2 = \frac{1}{20\pi (\lambda_0 + 2 \mu_0)c^3_{p,0}} > 0.
\end{align*}
In conjunction with \eqref{eq:108} yields that
\begin{align*}
-6 a^T M^{(5)}(0) a \ge 0,  \quad a \in \mathrm{Ker}\left(M^{(2)}(0)\right) \\
\mathrm{and}\; a\in \mathrm{Ker}(M^{(5)}(0)) \Rightarrow F = 0. 
\end{align*}
In conjunction with \eqref{eq:152}, the condition $F = 0$ and $\eqref{eq:102}$ yields 
\begin{align*}
\left\langle \left(S_0(0)\right)^{-1} e^{(j)}(0), \ \left(\sum^6_{l=1}a_l e^{(l)}(0)\right)\right\rangle_{\Gamma} = 0, \quad j=4,5,6.
\end{align*}
Using \eqref{eq:102} again, we arrive at 
\begin{align*}
\left\langle \left(S_0(0)\right)^{-1} \left[\sum^6_{l=1} a_l e^{(l)}(0)\right], \ \left(\sum^6_{l=1} a_l e^{(l)}(0)\right)\right\rangle_{\Gamma} = 0 ,
\end{align*}
which leads to $a = 0$. Based on the above discussions, we obtain \eqref{eq:98}.

Now we prove \eqref{eq:99}. Building upon \eqref{eq:93}, a straightforward calculation gives 
\begin{align*}
\sum^6_{l=1} a_l\left\langle \partial_{\nu,1} e^{(l)}_{\mathcal N'}(0), e_{\mathcal N}^{(j)}(0)\right\rangle_{\Gamma} =
\left\langle \partial_z \mathcal N(0) \left[\sum_{l=1}^{6} a_l e^{(l)}(0)\right], e_{\mathcal N,0}^{(j)}\right\rangle_{\Gamma},\quad j\in \{1,\ldots,6\}.
\end{align*}
In conjunction with \eqref{eq:94}, \eqref{eq:95} and \eqref{eq:102}, we obtain
\begin{align} \label{eq:150}
\sum^6_{l=1} a_l\left\langle \partial_{\nu,1} e^{(l)}_{\mathcal N'}(0), e_{\mathcal N}^{(j)}(0)\right\rangle_{\Gamma} = 0 \quad \mathrm{for}\; a\in \mathrm{Ker}\left(M^{(2)}(0)\right),\quad j\in \{1,\ldots,6\}.
\end{align}
Similarly, we have 
\begin{align*}
\sum^6_{j=1} a_j\left\langle e_{\mathcal N}^{(l)}(0), \partial_{\nu,1} e^{(j)}_{\mathcal N'}(0)\right\rangle_{\Gamma} = 0 \quad \mathrm{for}\; a\in \mathrm{Ker}\left(M^{(2)}(0)\right),\quad l\in \{1,\ldots,6\}.
\end{align*}
Combining this with \eqref{eq:24} and \eqref{eq:150} gives
\begin{align*}
\left(a^{(1)}\right)^T M^{(6)}(0) & a^{(2)} = - \sum^{6}_{j=1} a^{(1)}_j\sum^6_{l=1} a^{(2)}_{l}\left\langle \partial_{\nu,1} e^{(l)}_{\mathcal N'}(0), e_{\mathcal N}^{(j)}(0)\right\rangle_{\Gamma} \\
& - \sum^{6}_{l=1} a^{(2)}_l\sum^6_{j=1} a^{(1)}_{j}\left\langle  e^{(l)}_{\mathcal N}(0), \partial_{\nu,1} e_{\mathcal N'}^{(j)}(0)\right\rangle_{\Gamma} = 0 \quad \mathrm{for}\; a^{(1)}, a^{(2)}\in \mathrm{Ker}\left(M^{(2)}(0)\right).
\end{align*}
This proves \eqref{eq:99}.
\end{proof}


\subsection{Proofs of Lemmas \ref{le:9}--\ref{le:10} in Section \ref{sec:5.2}} \label{sec:B2}

\begin{proof}[Proof of Lemma \ref{le:9}]

Let $v^f(\omega):= R_0(\omega)f$. Define 
\begin{align} \label{eq:248}
w_\tau^f(\omega):=  R_{H^e(\tau;\Omega)}(\omega) f - v^f(\omega), \quad \omega \in \mathbb R.
\end{align}
Obviously,
\begin{align}
\frac{1}{\rho_1} L_{\lambda_1, \mu_1} w_\tau^f(\omega)  + \omega^2 w_\tau^f(\omega) = S^f(\omega)  & \quad \textrm{in}\; \Omega, \label{eq:64}\\
\partial_{\nu,1} w_\tau^f(\omega) := \tau\left[ \mathcal N(\omega)\gamma w_\tau^f(\omega) + \partial_{\nu,0}v^f(\omega)\right] - \partial_{\nu,1} v^f(\omega) & \quad \textrm{on}\; \Gamma. \label{eq:67}
\end{align}
Here, $S^f(\omega)$ is defined by
\begin{align*}
S^f(\omega):= \frac{1}{\rho_0} L_{\lambda_0, \mu_0} v^f(\omega) -  \frac{1}{\rho_1} L_{\lambda_1, \mu_1}v^f(\omega). 
\end{align*}
It follows from equations \eqref{eq:64}--\eqref{eq:67} that 
\begin{align}
J_{\tau}(w_\tau^f(\omega),\psi,z) = \rho_1 \left( S^f(\omega), \psi\right)_{\Omega} + \left\langle \partial_{\nu,1}v^f(\omega) - \tau \partial_{\nu,0}v^f(\omega), \psi \right\rangle_{\Gamma}\;\; \forall \psi \in \mathbf H^1(\Omega). \notag
\end{align}
By Riesz's theorem, there exists $h^f(\omega) \in \mathbf H^1(\Omega)$ satisfying 
\begin{align}
\langle h^f(\omega), \psi)_{\mathbf H^1(\Omega)} = \rho_1 \left( S^f(\omega), \psi\right)_{\Omega} + \left\langle \partial_{\nu,1}v^f(\omega) - \tau \partial_{\nu,0}v^f(\omega), \psi \right\rangle_{\Gamma} \quad \forall \psi \in \mathbf H^1(\Omega) \label{eq:249}
\end{align}
and 
\begin{align} \label{eq:154}
    \|h^f(\omega)\|_{\mathbf H^1(\Omega)} \le C \left(\| v^f(\omega)\|_{\mathbf H^1(\Omega)} +\|f\|_{\mathbf L^2(\Omega)}\right).
\end{align}
Therefore, for each $z_0 \in \Lambda \backslash \{0\}$, when $\tau$ is sufficiently small and $\omega$ is near $z_0$, with the aid of Proposition \ref{pro:4}, we readily obtain that 
\begin{align}
w_{\tau}^{f}(\omega&)= w_{\tau,{\rm{dom}}}^{f}(\omega;z_0)  + \left(\eta_1(\tau,\omega;z_0), \ldots, \eta_{n(z_0)}(\tau,\omega;z_0)\right) \notag \\
&\left(I- M (\tau,\omega;z_0)\right)^{-1}\left(\Pi(z_0)w_{\tau, \mathrm{dom}}^{f}(\omega;z_0)\right) \quad \mathrm{in}\; \Omega. \label{eq:68}
\end{align} 
Here, for any $\psi \in\mathbf H^1(\Omega)$, $w_{\tau,\textrm{dom}}^{f}(\omega;z_0)$ solves 
\begin{align} 
J^{\rm{dom}}_{z_0,\tau}(w_{\tau,{\rm{dom}}}^{f}(\omega;z_0), \psi) = \left(h^f(\omega), \psi\right)_{\mathbf H^1(\Omega)} \quad \forall \psi \in \mathbf H^1(\Omega). \label{eq:151}
\end{align}
Exploiting \eqref{eq:44} and \eqref{eq:51}, we have  
\begin{align} \label{eq:71}
\eta_l(\tau,\omega;z_0) = 
e^{(l)}({z_0}) + 
O_{\mathbf H^1(\Omega)}((\tau + |\omega - z_0|)), \;\; l\in\{1,\ldots, n(z_0)\},
\quad \mathrm{as}\; (\tau,\omega) \rightarrow (0,z_0).
\end{align}
Furthermore, using statement \eqref{a1} of Proposition \ref{pro:4}, we have the expansion
\begin{align}
w_{\tau, \textrm{dom}}^{f}(\omega;z_0) = \sum^{\infty}_{q_2=0}\sum^{\infty}_{q_1=0} \tau^{q_1}(\omega-z_0)^{q_2} w_{q_1,q_2}^f(\omega; z_0) \label{eq:69}
\end{align}
\textrm{in some neighborhood of} $(0,z_0)$, where
\begin{align} 
&\left\|w_{q_1,q_2}^f(\omega; z_0)\right\|_{\mathbf H^1(\Omega)} \le C_{q_1,q_2}\|h^f(\omega)\|_{\mathbf L^2(\Omega)}, \quad q_1,q_2 \in \mathbb{N} \cup \{0\} \label{eq:70}\\
&{\rm{and}} \quad \sum^{\infty}_{q_2=0}\sum^{\infty}_{q_1=0}\tau^{q_1} |z-z_0|^{q_2} C_{q_1,q_2} < \infty. \label{eq:155}
\end{align}


Inserting \eqref{eq:69} into \eqref{eq:151}, and equating the coefficients for the terms $t^{q_1}( \omega -z_0)^{q_2}$ with $(q_1, q_2) = (0,0), (1,0), (0,1)$, we have that   
\begin{align}
\left[ L_{\lambda_1, \mu_1}   + \rho_1 z_0^2 + P(z_0)\right] w_{0,0}^f(\omega; z_0) = \rho_1 S^f(\omega)  &\quad\textrm{in}\; \Omega, \label{eq:200}\\
\partial_{\nu,1} w_{0,0}^f(\omega;z_0) = -\partial_{\nu,1}v^f(\omega) + \tau \partial_{\nu,0} v^f(\omega)   &\quad \textrm{on}\; \Gamma, \\
\left[L_{\lambda_1, \mu_1}  + \rho_1 z_0^2  +P(z_0)\right] w_{1,0}^f(\omega;z_0) = 0  &\quad \textrm{in}\; \Omega, \\
\partial_{\nu,1} w_{1,0}^f(\omega;z_0) = \mathcal N(z_0)\gamma w_{0,0}^f(\omega;z_0)  &\quad \textrm{on}\; \Gamma, \label{eq:201} \\
\mathrm{and}\; \left[L_{\lambda_1, \mu_1}  + \rho_1 z_0^2 +  P(z_0)\right] w^f_{0,1}(\omega;z_0) +\rho_1 2 z_0 w_{0,0}^f(\omega;z_0) = 0 & \quad \textrm{in}\; \Omega, \notag \\
\partial_{\nu,1} w_{0,1}^f(\omega;z_0) = 0 &\quad \textrm{on}\; \Gamma. \notag 
\end{align}
Furthermore, by the definition of $S^f({\omega})$, equation \eqref{eq:200} can be rewritten as 
\begin{align*}
\left[L_{\lambda_1,\mu_1} + \rho_1 z_0^2 + P(z_0)\right]\left[w_{0,0}^f(\omega;z_0) + v^f(\omega)\right] = -\rho_1 f + (\rho_1 (z_0^2-\omega^2) + P(z_0)) v^f(\omega) & \quad\mathrm{in}\; \Omega.
\end{align*}
Building upon the above equations, using \eqref{eq:154}, \eqref{eq:69}, \eqref{eq:70} and \eqref{eq:155}, we have 
\begin{align} \label{eq:72}
w_{\tau,\mathrm{dom}}^f(\omega;z_0) &= w^f_0(\omega;z_0) + \tau w_1^f(\omega;z_0) + (\omega-z_0) w_2^f(\omega;z_0) \notag \\
&+ O_{\mathbf H^1(\Omega)}\left(\left(|\omega-z_0| + \tau\right)^2 \left( \| v^f(\omega)\|_{\mathbf H^1(\Omega)}+\|f\|_{\mathbf L^2(\Omega)}\right)\right), \quad \mathrm{as}\; (\tau,\omega) \rightarrow (0,z_0),
\end{align}
where $w^f_0(\omega;z_0)$ satisfies \eqref{eq:73}--\eqref{eq:75}, and $w_2^f(\omega;z_0) $ and $w_3^f(\omega;z_0)$ solve  
\begin{align}
\left[ L_{\lambda_1,\mu_1}   + \rho_1 z_0^2  +P(z_0)\right] w_1^f(\omega;z_0) &= 0    \quad  \mathrm{in}\; \Omega, \label{eq:76}\\
\partial_{\nu,1} w_{1}^f(\omega;z_0) &= \mathcal N(z_0)\gamma w_{0}^f(\omega;z_0) + \partial_{\nu,0}v^f(\omega)  \quad \mathrm{on}\; \Gamma, \label{eq:77} \\
\left[L_{\lambda_1,\mu_1} + \rho_1 z_0^2 + P(z_0)\right]w_2^f(\omega;z_0) & = 
- \rho_1 2 z_0 w_{0}^f(\omega;z_0), \quad \mathrm{in}\; \Omega,
\label{eq:78} \\
\partial_{\nu,1} w_2^f(\omega;z_0) & = 0 \quad \mathrm{on}\; \Gamma.\label{eq:79}
\end{align}
We proceed to calculate $\Pi(z_0) w_j^f(\omega;z_0)$ (for $j\in \{0,1,2\}$).
Since 
\begin{align}
&\left\langle\partial_{\nu,1} v^f(\omega), e^{(l)}(z_0)\right\rangle_{\Gamma} = \left\langle\partial_{\nu,1} v^f(\omega), e^{(l)}(z_0)\right\rangle_{\Gamma} - \left\langle v^f(\omega), \partial_{\nu,1}e^{(l)}(z_0)\right\rangle_{\Gamma} \notag \\
&= \left(L_{\lambda_1,\mu_1}v^f(\omega), e^{(l)}(z_0)\right)_{\Omega} + \rho_1 z^2_0 \left[\Pi(z_0)v^f(\omega)\right]_l, \quad l \in \{1,\ldots, n(z_0)\},\label{eq:159} 
\end{align}
using \eqref{eq:73} and \eqref{eq:75}, we have 
\begin{align*}
\left[\Pi(z_0)  w_0^f(\omega;z_0)\right]_l &= \left(\rho_1 S^f(\omega), e^{(l)}(z_0)\right)_{\Omega} + \left\langle\partial_{\nu,1} v^f(\omega), e^{(l)}(z_0)\right\rangle_{\Gamma}.
\end{align*} 
In conjunction with \eqref{eq:159} gives 
\begin{align}
&\left[\Pi(z_0) w_0^f(\omega;z_0)\right]_l = -\rho_1 [\Pi(z_0)f]_l  + \rho_1\left(z^2_0 - \omega^2 \right)[\Pi(z_0)v^f(\omega)]_l.  \label{eq:156}
\end{align}
For the calculation of $\Pi(z_0)w_1^f(\omega;z_0)$, using \eqref{eq:76} and \eqref{eq:77}, we obtain 
\begin{align}
&\left[\Pi(z_0)w_1^f(\omega;z_0)\right]_l = - \left\langle \mathcal N(z_0)\gamma w_{0}^f(\omega;z_0) + \partial_{\nu,0} v^f(\omega), e^{(l)}(z_0)\right\rangle_{\Gamma}, \quad l \in \{1,\ldots, n(z_0)\}.  \label{eq:157}
\end{align}
For the calculation of $\Pi(z_0) w_2^f(\omega;z_0)$, it can be deduced from \eqref{eq:78} and \eqref{eq:79} that
\begin{align*}
\left[\Pi(z_0)w_2^f(\omega;z_0)\right]_l =
- 2z_0 \rho_1  \left[\Pi(z_0)w_0^f(\omega;z_0)\right]_l,\quad l \in \{1,\ldots, n(z_0)\}.
\end{align*}

Building upon the above calculations of $\Pi(z_0) w_j^f(\omega;z_0)$ (for $j\in \{0,1,2\}$), using \eqref{eq:27}, \eqref{eq:68}, \eqref{eq:71} and \eqref{eq:72}, we obtain \eqref{eq:81}. 

\end{proof}

\begin{proof}[Proof of Lemma \ref{le:3}]
It follows from that Lemma \ref{le:7} that $M^{(1)}(z_0)$ is invertible. Therefore, we have    
\begin{align}
  \left\|\left(M^e(\tau;\omega;z_0)\right)^{-1}\right\| = \begin{cases}
  \theta\left(\frac 1 {|\omega - z_0|}\right),  & \mathrm{when}\; \tau \le \frac{|z_0|\rho_1}{\|M^1(z_0)\|}|\omega-z_0|,\\
   \theta\left( \frac{1}\tau\right), & \mathrm{when}\; |\omega - z_0| \le \frac{\|M^1(z_0)\|}{4|z_0|\rho_1} \tau.
  \end{cases} \label{eq:197}
\end{align}
Furthermore, since all eigenvalues of $M^{(1)}(z_0)$ have positive imaginary parts (see Lemma \ref{le:7}), we have 
\begin{align}
    \left\|\left(M^e(\tau;\omega;z_0)\right)^{-1}\right\| = \Theta\left(\frac{1}{\tau}\right), \quad   \mathrm{when}\;\; \frac{|z_0|\rho_1}{\|M^1(z_0)\|}|\omega - z_0|  < {\tau} < \frac{4|z_0|\rho_1}{\|M^1(z_0)\|}|\omega - z_0|. \label{eq:198}
\end{align}
Combining \eqref{eq:197} and \eqref{eq:198} gives \eqref{eq:133}.
\end{proof}

\begin{proof}[Proof of Lemma \ref{le:10}]
Let $\vep >0$ be sufficiently small, and let $C$ is a positive constant depending only on $\Omega$ throughout the proof.

\eqref{z1}
We first prove \eqref{eq:170}. Using the identity
\begin{align*}
G^{(0)}(y_0 + \vep^{-1} (x-y_0), y;\widetilde \omega) = G^{(0)}(\vep^{-1} (x-y_0), y - y_0; \widetilde \omega),
\end{align*}
it is sufficient to investigate
\begin{align}
\int_{\R^3}\int_{\Omega}G^{(0)}(\vep^{-1}(x-y_0),y-y_0;\widetilde \omega) \psi_1(y)\cdot f(x) dy dx, \quad \mathrm{for}\; f\in \mathbf L^2_{\beta}(\R^3)\; \mathrm{and}\; \psi_1 \in \mathbf L^2(\Omega). \notag
\end{align}
To do so, we split the above integral into the following three parts:
\begin{align*}
&\int_{\R^3}\int_{\Omega}G^{(0)}(\vep^{-1}(x-y_0),y-y_0;\widetilde \omega) \psi_1(y) \cdot f(x) dy dx \\
&= \sum^{2}_{j = 1 }\int_{U_j(\vep)}\int_{\Omega} G^{(0)}(\vep^{-1}(x-y_0),y-y_0;\widetilde \omega) \psi_1(y) \cdot f(x) dy dx  =: \sum^{2}_{j=1} I_{j}({\vep}).
\end{align*}
Here,
\begin{align}
&U_1(\vep):=\left\{y \in \R^3: |x-y_0| \ge  4\vep \max_{y\in\overline \Omega}|y-y_0|\right\}, \label{eq:180} \\
&U_2(\vep):= \R^3 \backslash U_1(\vep). \label{eq:176}
\end{align}

We proceed to estimate $I_1(\vep)$ and $I_2(\vep)$ in turn. 
Clearly,
\begin{align} \label{eq:142}
|\vep^{-1}(x-y_0) - (y-y_0)| = \vep^{-1}|x-y_0|\sqrt{1 - 2\vep \frac{(x-y_0)\cdot (y - y_0)}{|x-y_0|^2} + {\frac{\vep^2|y-y_0|^2}{|x-y_0|^2}}}.
\end{align}
Then, 
\begin{align} \label{eq:177}
\left|2\vep \frac{(x-y_0)\cdot (y - y_0)}{|x-y_0|^2}\right|\le \frac 12, \quad {\frac{\vep^2|y-y_0|^2}{|x-y_0|^2}} \le \frac{1}{16}, \quad \mathrm{for}\;x\in U_1(\vep)\; \mathrm{and} \;y \in \Omega.
\end{align}
With the aid of the inequality $1/\sqrt{1+s_1+s_2} \le C(1 + |s_1 + s_2|)$ for $s_1 \in [-1/2,1/2]$ and $s_2 \in [0,1/16]$, we obtain 
\begin{align}
&\frac{1}{|\vep^{-1}(x-y_0) - (y-y_0)|} = \frac{\vep}{|x-y_0|} +   O\left(\frac{\vep^2}{|x-y_0|^2}\right), \label{eq:167}
\end{align}
uniformly for $\vep$, $x\in U_1(\vep)$ and $y\in \Omega$. From this, we also have
\begin{align}
&\frac{\left(\vep^{-1}(x-y_0)_l-(y-y_0)_l\right) \left(\vep^{-1}(x-y_0)_k-(y-y_0)_k\right)}{|\vep^{-1}(x-y_0) - (y-y_0)|^2} \notag\\
&= \frac{(x-y_0)_l (x-y_0)_k}{|x-y_0|^2} + O\left(\frac{\vep}{|x-y_0|}\right), \quad l,k\in\{1,2,3\}, \notag
\end{align}
uniformly for $\vep$, $x\in U_1(\vep)$ and $y\in \Omega$. Furthermore, it can be seen that
\begin{align} 
&\sqrt{1+s_1+s_2} -\left(1+\frac{s_1}2\right) = \frac{s_2- \frac{s^2_1}4}{\sqrt{1+s_1+s_2} + {1+\frac{s_1}2}}, \quad s_1 \in [-1/2,1/2], \;s_2 \in [0,1/16],\label{eq:169}
\end{align}
and
\begin{align} \label{eq:165}
& |1- e^{s_3}| \le C_V|s_3|, \quad s_3 \in V,
\end{align}
where $V$ is a bounded interval in $\R$. Therefore, we find 
\begin{align}
e^{i\widetilde \omega |\vep^{-1}(x-y_0) - (y-y_0)|} = e^{i\frac{\widetilde \omega}\vep|x-y_0|} e^{-i\widetilde \omega \hat x_{y_0} \cdot(y-y_0)} + O\left(\frac{\vep}{|x-y_0|}\right), \label{eq:166}
\end{align}
uniformly for $\vep$, $x\in U_1(\vep)$ and $y\in \Omega$. Furthermore, it is known that for $\widetilde \omega \in \R \backslash \{0\}$,
\begin{align}
G^{(0)}(x,y;\widetilde \omega)
&=  \sum_{\sigma \in \{p,s\}} \frac{e^{\frac{i\widetilde \omega}{c_{\sigma,0}}|x|}}{|x|} (G^{(0)})_\sigma^{\infty}(\hat x, y; \widetilde \omega)+ O\left(\frac{1}{|x|^2}\right) \;\;\mathrm{as}\; |x| \rightarrow \infty, \label{eq:140} 
\end{align}

Based on the above discussions, proceeding in the derivation of \eqref{eq:140}, we obtain 
\begin{align}
&G^{(0)}(\vep^{-1}(x-y_0),y-y_0;\widetilde \omega ) \notag\\ 
&= \vep  \sum_{\sigma \in \{p,s\}} \frac{e^{\frac{i\widetilde \omega}{\vep c_{\sigma,0}}|x-y_0|}}{|x-y_0|} (G^{(0)})_\sigma^{\infty}(\hat x_{y_0}, y - y_0 ; \widetilde \omega)  + O\left(\frac{\vep^2}{|x-y_0|^2}\right), \label{eq:141}
\end{align}
as $\vep\rightarrow 0$ uniformly for $\vep$, $x\in U_1(\vep)$ and $y\in \Omega$.
Furthermore, it is easy to verify that 
\begin{align} \label{eq:181}
\int_{U_1(\vep)}\frac{1}{|x-y_0|^4} dx =  \int^{2\pi}_0\int^{\pi}_{0}\int^{\infty}_{2\vep \max_{y\in\overline \Omega}|y-y_0|} \frac{1}{r^4} r^2 \sin\theta  dr d\theta d\phi
\le \frac{C}\vep.
\end{align}
Here, we represent $x$ by 
\begin{align*}
x = y_0 + r(\sin\theta\cos\phi, \sin\theta\sin\phi, \cos\theta)^T,\;\; r\in[4\vep \max_{y\in\overline \Omega}|y-y_0|,\infty), \phi\in[0,\pi]\; \mathrm{and}\; \theta\in [0,2\pi].
\end{align*}
Therefore, utilizing \eqref{eq:141}, we arrive at 
\begin{align}
&\left|I_1(\vep) -  \sum_{\sigma \in \{p,s\}}\int_{U_1(\vep)} \int_{\Omega} \vep \frac{e^{\frac{i\widetilde \omega}{\vep c_{\sigma,0}}|x-y_0|}}{|x-y_0|} (G^{(0)})_\sigma^{\infty}(\hat x_{y_0}, y - y_0 ; \widetilde \omega) \psi_1(y)\cdot f(x) dy dx\right| \notag \\
&\le C \vep^2 \|\psi_1(y)\|_{L^2(\Omega)}\int_{U_1(\vep)} \frac{1}{|x-y_0|^2}|f(x)|dx \le C \vep^{3/2} \|\psi_1\|_{\mathbf L^2(\Omega)} \|f\|_{\mathbf L^2_\beta(\R^3)}. \label{eq:175}
\end{align}
Here, the last inequality follows from Cauchy-Schwartz inequality and \eqref{eq:181}.

Moreover, by a change of variable $t = \vep^{-1}(x-y_0) + y_0$, we have 
\begin{align}
I_2(\vep) = \vep^3\int_{\widetilde U_2(\vep)} \int_{\Omega} G_0(t-y_0,y-y_0;\widetilde \omega) \psi_1(y)\cdot f(y_0+\vep(t-y_0)) dy dt,\label{eq:143}
\end{align}
where 
\begin{align} \label{eq:168}
\widetilde U_2(\vep):=y_0 + (U_2(\vep)-y_0)/\vep.
\end{align}
We note that 
\begin{align*}
\widetilde U_2(\vep) \subset B_{2\max_{y\in\overline \Omega}|y-y_0|} (y_0).
\end{align*}
This, together with \eqref{eq:143} and the estimate
\begin{align*}
\|f(y_0+\vep(\cdot - y_0))\|_{\mathbf L^2(y_0 + (U_1(\vep)-y_0)/\vep)} \le \vep^{-3/2} \|f\|_{\mathbf L_\beta^2(\R^3)}
\end{align*}
yields 
\begin{align}
|I_2(\vep)| \le C{\vep^{3/2}}\|\psi_1\|_{\mathbf L^2(\Omega)} \|f\|_{\mathbf L^2_\beta(\R^3)}. \label{eq:179}
\end{align}
Similarly, under the change of variable $t = \vep^{-1}(x-y_0) + y_0$, we have 
\begin{align}\label{eq:174}
\left|\int_{U_2(\vep)} \frac{1}{|x-y_0|} g(y) dy\right| \le C\vep^{1/2} \|g\|_{L^2(\Omega)}.
\end{align}
Utilizing the continuity of the far-field pattern with respect to $y$, and applying \eqref{eq:174}, we obtain
\begin{align} \label{eq:178}
&\left|\vep \sum_{\sigma \in \{p,s\}}\int_{{U_2(\vep)}} \int_{\Omega}\frac{e^{\frac{i\widetilde \omega}{\vep c_{\sigma,0}}|x-y_0|}}{|x-y_0|} (G^{(0)})_\sigma^{\infty}(\hat x_{y_0}, y - y_0 ; \widetilde \omega) dy dx\right|\le C \vep^{3/2} \|\psi_1\|_{\mathbf L^2(\Omega)} \|f\|_{\mathbf L^2_\beta(\R^3)}.
\end{align}

In conjunction with \eqref{eq:175}, \eqref{eq:179} and \eqref{eq:178}, we obtain \eqref{eq:170}.

Similarly as in the derivation of \eqref{eq:170}, we readily obtain \eqref{eq:171}.

\eqref{z2} 
Using \eqref{eq:167} and \eqref{eq:166}, we have that for each $j\in\{1,2,3\}$,
\begin{align}
&\partial_{y_j} G^{(0)} (\vep^{-1}(x-y_0),y-y_0;\widetilde \omega ) \notag\\ 
&= \vep  \sum_{\sigma \in \{p,s\}}\frac{e^{\frac{i\widetilde \omega}{\vep c_{\sigma,0}}|x-y_0|}}{|x-y_0|} \partial_{y_j} (G^{(0)})_\sigma^{\infty}(\hat x_{y_0}, y - y_0 ; \widetilde \omega) + O\left(\frac{\vep^2}{|x-y_0|^2}\right).\label{eq:182}
\end{align}
as $\vep\rightarrow 0$ uniformly for $\vep$, $x\in U_1(\vep)$ and $y\in \Omega$, and the partial derivative of the $(k,l)$ entry of $G_0$ with respect to the variable $y_j$ satisfies
\begin{align}
\left|\partial_{y_j}(G^{(0)})_{kl}(\vep^{-1}(x-y_0),y-y_0;\widetilde \omega)\right| \le \frac{C\vep}{|x-y_0|}, \quad \mathrm{for}\; x \in U_2(\vep)\; \mathrm{and}\; y\in \Omega. \label{eq:183}
\end{align}
Here, $U_1(\vep)$ and $U_2(\vep)$ are given by \eqref{eq:180} and \eqref{eq:176}, respectively. Furthermore, given any compact set $\mathcal K$, we have 
\begin{align}
\left\|\int_{\Omega} \partial_{y_j}(G^{(0)})_{kl}(\cdot-y_0,y-y_0;\widetilde \omega) g(y)dy\right\|_{L^2(\mathcal K)} \le C_{\mathcal K} \|g\|_{L^2(\Omega)}, \label{eq:184}
\end{align}
Building upon \eqref{eq:182}, \eqref{eq:183} and \eqref{eq:184}, proceeding as in the derivation of \eqref{eq:170}, we readily obtain \eqref{eq:186}.

\eqref{z3}
Using the identity
\begin{align} \label{eq:185}
   \vep G^{(0)}(y_0+\vep(x-y_0),y; \widetilde \omega/\vep) = G^{(0)}(\vep^{-1}(y-y_0), x-y_0;\widetilde \omega),
\end{align}
we find that 
\begin{align*}
&\vep \int_{\Omega}\int_{\R^3}G^{(0)}(y_0+\vep(x-y_0),y; \vep^{-1}\widetilde \omega ) f(y)\cdot \psi_1(x) dy dx\\
&= \int_{\R^3}\int_{\Omega} G^{(0)}(\vep^{-1}(y-y_0), x-y_0;\widetilde \omega) \psi_1(x) \cdot f(y) dx dy \quad \mathrm{for}\; f\in \mathbf L^2_{\beta}(\R^3)\; \mathrm{and}\; \psi_1 \in \mathbf L^2(\Omega),
\end{align*}
It follows from \eqref{eq:170} and \eqref{eq:185} that
\begin{align}
&\int_{\R^3} G^{(0)}(y_0+\vep(x-y_0),y; \widetilde \omega/\vep) f(y) dy \notag \\
&= \sum_{\sigma \in \{p,s\}}\int_{\R^3}\frac{e^{\frac{i\widetilde \omega}{\vep c_{\sigma,0}}|y-y_0|}}{|y-y_0|}(G^{(0)})_\sigma^{\infty}(\hat y_{y_0}, x - y_0 ; \widetilde \omega)f(y) dy +  O_{\mathbf L^2(\Omega)} \left(\vep^{1/2}\|f\|_{\mathbf L_{\beta}^2(\R^3)}\right). \label{eq:187}
\end{align}
Furthermore, with the aid of \eqref{eq:185}, for $j,k,l\in\{1,2,3\}$, we have 
\begin{align*}
\vep \partial_{x_j} \left(G^{(0)}\right)_{kl}(y_0+\vep(x-y_0),y; \vep^{-1}\widetilde \omega) = \partial_{x_j} \left(G^{(0)}\right)_{kl}(\vep^{-1}(y-y_0), x-y_0;\widetilde \omega),
\end{align*}
whence \eqref{eq:173} follows from \eqref{eq:186} and \eqref{eq:187}.
\end{proof}

\subsection{Proofs of Lemmas \ref{le:4}--\ref{le:11} in Section \ref{sec:5.3}}  \label{sec:B3}

\begin{proof}[Proof of Lemma \ref{le:4}]
The proof of Lemma \ref{le:4} is analogous to that of Lemma \ref{le:9}.
Let $v^f(\omega):= R_0(\omega)f$ and let $w_\tau^f(\omega)$ be given by \eqref{eq:248}. 
Proceeding as in the derivation of \eqref{eq:68}, in the high-contrast limit and for $\omega$ near $0$, we have,
\begin{align}
w_{\tau}^{f}(\omega&) = w_{\tau,{\rm{dom}}}^{f}(\omega;0)  + \left(\eta_1(\tau,\omega;0), \ldots, \eta_{6}(\tau,\omega;0)\right) \notag \\
&\left(I- M (\tau,\omega;0)\right)^{-1}\left(\Pi(z_0)w_{\tau, \mathrm{dom}}^{f}(\omega;0)\right) \quad \mathrm{in}\; \Omega. \label{eq:254}
\end{align} 
Here, for any $\psi \in\mathbf H^1(\Omega)$, $w_{\tau,\textrm{dom}}^{f}(\omega;0)$ solves 
\begin{align} 
J^{\rm{dom}}_{0,\tau}(w_{\tau,{\rm{dom}}}^{f}(\omega;0), \psi) = \left(h^f(\omega), \psi\right)_{\mathbf H^1(\Omega)} \quad \forall \psi \in \mathbf H^1(\Omega) \label{eq:250}
\end{align}
with $h^f$ given by \eqref{eq:249} and satisfying \eqref{eq:154}, and $\eta_l(\tau,\omega;0)$ (for $l\in\{1,\ldots,6\}$) is given by Proposition \ref{pro:4}. It follows from statement \eqref{a1} of Proposition \ref{pro:4} that 
\begin{align}
w_{\tau, \textrm{dom}}^{f}(\omega;0) = \sum^{\infty}_{q_2=0}\sum^{\infty}_{q_1=0} \tau^{q_1}\omega^{q_2} w_{q_1,q_2}^f(\omega;0) \label{eq:251}
\end{align}
\textrm{in some neighborhood of} $(0,z_0)$, where
\begin{align} 
&\left\|w_{q_1,q_2}^f(\omega; 0)\right\|_{\mathbf H^1(\Omega)} \le C_{q_1,q_2}\|h^f(\omega)\|_{\mathbf L^2(\Omega)}, \quad q_1,q_2 \in \mathbb{N} \cup \{0\} \notag\\
&{\rm{and}} \quad \sum^{\infty}_{q_2=0}\sum^{\infty}_{q_1=0}\tau^{q_1} |z|^{q_2} C_{q_1,q_2} < \infty. \notag
\end{align}
Furthermore, using \eqref{eq:44} and \eqref{eq:92}, we have 
\begin{align}
\eta_l(\tau,\omega;0) = 
e^{(l)}(0) +   O_{\mathbf H^1(\Omega)}((\tau + \omega^2)),\;\; l\in\{1,\ldots,6\},  \quad \mathrm{as}\; (\tau,\omega) \rightarrow (0,0). \label{eq:252}
\end{align}
 Inserting \eqref{eq:251} into \eqref{eq:250}, and equating the coefficients of the terms $t^{q_1}\omega^{q_2}$ for $(q_1, q_2) =  (0,0), (1,0), (0,1), (1,1),(0,2)$, we obtain that equations \eqref{eq:200}--\eqref{eq:201} hold for $z_0 = 0$, and that
\begin{align}
\left[L_{\lambda_1, \mu_1} + P(0)\right] w_{0,1}^f(\omega;0) = 0 &\quad \textrm{in}\; \Omega, \notag \\
\partial_{\nu,1} w_{0,1}^f(\omega;0) = 0  &\quad \textrm{on}\; \Gamma, \notag \\
\left[L_{\lambda_1, \mu_1} + P(0)\right] w_{1,1}^f(\omega;0) = 0 &\quad \textrm{in}\; \Omega, \notag \\
\partial_{\nu,1} w_{1,1}^f(\omega;0) = \partial_z\mathcal N(0)\gamma w_{0,0}^f(\omega;0)  &\quad \textrm{on}\; \Gamma, \notag \\
\left[L_{\lambda_1, \mu_1} + P(0)\right] w^f_{0,2}(\omega;0) +\rho_1 w_{0,0}^f(\omega;0) = 0 & \quad \textrm{in}\; \Omega, \notag \\
\partial_{\nu,1} w_{0,2}^f(\omega;0) = 0 &\quad \textrm{on}\; \Gamma. \notag 
\end{align}
Proceeding as in the derivation of \eqref{eq:72}, we have 
\begin{align}
w_{\tau,\mathrm{dom}}^f(\omega;0) &= w^f_0(\omega;0) + \tau w_1^f(\omega;0) + \omega^2 w_2^f(\omega;0) + \tau \omega w_3^f(\omega;0)\notag \\
&+ O_{\mathbf H^1(\Omega)}\left(\left(\tau + \omega^2\right)^2 \left( \| v^f(\omega)\|_{\mathbf H^1(\Omega)} +\|f\|_{\mathbf L^2(\Omega)}\right)\right), \quad \mathrm{as}\; (\tau,\omega) \rightarrow (0,z_0),\label{eq:253}
\end{align}
where $w^f_0(\omega;0)$ satisfies \eqref{eq:61}--\eqref{eq:62}, $w_1^f(\omega;0)$ satisfies \eqref{eq:76}--\eqref{eq:77} for $z_0 = 0$, and $w_2^f(\omega;0)$ and $w_3^f(\omega;0)$ satisfy
\begin{align*}
\left[L_{\lambda_1,\mu_1}  + P(0)\right]w_2^f(\omega;0) &=
-\rho_1 w_{0}^f(\omega;0) 
\quad \mathrm{in}\; \Omega, \\
\partial_{\nu,1} w_2^f(\omega;0) & = 0 \quad \mathrm{on}\; \Gamma,\\
\left[L_{\lambda_1,\mu_1} + P(0)\right]w_3^f(\omega;0) &=
0  \quad \mathrm{in}\; \Omega, \\
\partial_{\nu,1} w_3^f(\omega;0) & = \partial_z \mathcal N(0) \gamma w_{0}^f(\omega;0) \quad \mathrm{on}\; \Gamma, 
\end{align*}
Then, by a straightforward calculation, we find that \eqref{eq:156} and \eqref{eq:157} hold for $z_0 = 0$ and that 
\begin{align*}
&\left[\Pi(0)w_2^f(\omega;0)\right]_l =
-\rho_1  \left[\Pi(0)w_0^f(\omega;0)\right]_l,\\
&\left[\Pi(0)w_2^f(\omega;0)\right]_l =
-\left\langle\partial_z N(0)\gamma w_0^f(\omega;0), e^{l}(0)\right\rangle_\Gamma, \quad l \in \{1,\ldots, n(0)\}.
\end{align*}
Therefore, using \eqref{eq:28}, \eqref{eq:254}, \eqref{eq:252} and \eqref{eq:253}, we readily obtain \eqref{eq:195}. The proof of this lemma is thus completed.

\end{proof}

\begin{proof}[Proof of Lemma \ref{le:8}]
We first prove \eqref{eq:134}. For $\kappa \in \mathcal L^{(0)} \backslash \mathcal E$, define 
\begin{align*}
\widetilde Q(\kappa):= 
\begin{bmatrix}
Q_{\mathrm{p}}(\kappa)  \, \big |\, Q_{\mathrm{st}}(\kappa)
\end{bmatrix},
\end{align*}
where $Q_{\mathrm{st}}(\kappa)$ is defined by \eqref{eq:211}, and the columns of
\[
  Q_{\mathrm p}(\kappa)
  := \bigl[\,Q_{\mathrm p}(\kappa^{(1)})\ \big|\ Q_{\mathrm p}(\kappa^{{(2)}})\ \big|\ \cdots\,\bigr]
\]
form an orthonormal basis of the orthogonal direct sum $E_{\mathrm{p}}(\kappa)$. Here, for each $\kappa' \in \mathcal L^{(1)}(\kappa)$, $Q_{\mathrm p}(\kappa')$ is a matrix whose columns form an
orthonormal basis of $E_{\mathrm p}(\kappa;\kappa')$.

By the definition of $Q_p(\kappa)$, we have 
\begin{align} \label{eq:209}
Q^T_p(\kappa) M^{(2)}(0)Q_p(\kappa) = D_{\mathrm p}(\kappa),
\end{align}
where $D_{\mathrm p}(\kappa)$ is a diagonal matrix whose diagonal entries are exactly the element of $\mathcal L^{(1)}(\kappa)$.
We define 
\begin{align*}
    \widetilde M^e(\tau, \widetilde \omega):=\widetilde Q^T(\kappa) M^e(\tau,\sqrt \tau \widetilde \omega;0) \widetilde Q(\kappa).
\end{align*}
For any $a \in \mathbb C^6$, we can write 
\begin{align} 
a = \widetilde Q(\kappa) \widetilde a, \notag
\end{align}
where $\widetilde a$ denotes its coordinate vector with respect to the basis formed by the columns of $\widetilde Q(\kappa)$. Therefore,
we have 
\begin{align}
    &\left(M^e(\tau,\sqrt \tau \widetilde \omega;0)\right)^{-1} a = \widetilde Q(\kappa) \left(\widetilde M^e(\tau,\widetilde \omega)\right)^{-1} \widetilde a \label{eq:208}
\end{align}
Based on \eqref{eq:28} and \eqref{eq:209}, we have 
\begin{align}
\widetilde M^e(\tau,\widetilde \omega)
& =: \begin{bmatrix}
&\widetilde M^{e,1,1}(\tau, \widetilde \omega) &\widetilde M^{e,1,2}(\tau,\widetilde \omega) \\
&\widetilde M^{e,2,1}(\tau,\widetilde \omega) &\widetilde M^{e,2,2}(\tau, \widetilde \omega) 
\end{bmatrix} \notag \\
& = \tau \begin{bmatrix}
    \rho_1 \widetilde \omega^2 + \kappa - i\sqrt{\tau}\widetilde \omega D_{\mathrm p}(\kappa) + O(\tau)  & O(\sqrt \tau) \\
    O(\sqrt \tau ) & \kappa I - D(\kappa) + O(\sqrt \tau)
\end{bmatrix}, \label{eq:210}
\end{align}
where $D(\kappa)$ is defined by \eqref{eq:212}. Furthermore, it can seen that 
\begin{align*}
&\left\|\left(\widetilde M^{e,1,1}(\tau, \widetilde \omega)\right)^{-1} -\left(\rho_1 \widetilde \omega^2 + \kappa -  i\sqrt{\tau} \widetilde w D_{\mathrm p}(\kappa)\right)^{-1}\right\| \\
&\le \frac{C \sqrt \tau}{\left|\rho_1 \widetilde w^2 + \kappa - i \sqrt \tau \widetilde w \min_{\kappa' \in \mathcal L(\kappa)} |\kappa'|\right|},\\
&\mathrm{and}\; \left\|\rho_1 \widetilde \omega^2 + \kappa -  i\sqrt{\tau} \widetilde w D_{\mathrm p}(\kappa)\right\|^{-1} \le C \frac{1}{{\left|\rho_1 \widetilde w^2 + \kappa - i \sqrt \tau \widetilde w \min_{\kappa' \in \mathcal L(\kappa)} |\kappa'|\right|}} \le \frac{C}{\sqrt \tau}.
\end{align*}
From this, with the aid of Schur complement formula for the inverse, \eqref{eq:208} and \eqref{eq:210}, for any $a \in \mathbb C^6$, we readily obtain that \eqref{eq:134} with the remainder term $\mathrm{Res}^{(1)}$ satisfying \eqref{eq:191}.

Second, we prove \eqref{eq:121}. For $\kappa \in \mathcal L^{(0)} \cap \mathcal E$, define 
\begin{align*}
\widetilde Q(\kappa):= 
\begin{bmatrix}
\widetilde Q_0(\kappa) \, \big |\, Q_{\perp}(\kappa)  \, \big |\,Q_{\mathrm{st}}(\kappa)
\end{bmatrix},
\end{align*}
where $Q_{\mathrm{st}}(\kappa)$ is defined by \eqref{eq:211}, $\widetilde Q_{0}(\kappa)$ is the matrix obtained by concatenating, 
in blocks of columns, the matrices 
\begin{align*}
\bigl[\,\widetilde Q_{0}(\kappa;\alpha_1,\alpha_2)\,\bigr]_{\alpha_2\in \mathcal L^{(3)}(\kappa;0,\alpha_1)}\quad \mathrm{for}\; \alpha_1\in \mathcal L^{(2)}(\kappa;0). 
\end{align*}
Here, $\widetilde Q_{0}(\kappa;0,\alpha_1,\alpha_2)$ is a matrix whose columns form an
orthonormal basis of $E_{\mathrm p}(\kappa;0,\alpha_1,\alpha_2)$.
In particular, the first block of columns in $\widetilde Q_{0}(\kappa)$ corresponds to $\alpha_1=\kappa''$ and is equal to
\begin{align*}
\bigl[\,\widetilde Q_{0}(\kappa;\kappa'',\kappa''')\,\bigr]_{\kappa'''\in \mathcal L^{(3)}(\kappa;0,\kappa'')}.
\end{align*}
We proceed to define 
\begin{align*}
\widetilde M^e(\tau, \widetilde \omega):=\widetilde Q^T(\kappa)M^e\left(\tau,\pm\sqrt\tau\sqrt\frac{-\kappa}{\rho_1} + \tau^{3/2} \widetilde w;0\right) \widetilde Q(\kappa).
\end{align*}
Based on \eqref{eq:120}, a straightforward calculation gives
\begin{align*}
\widetilde M^e(\tau;\omega) = \tau
\begin{bmatrix}
\widetilde M^{(e,1,1)}(\tau;\widetilde\omega) &  \widetilde M^{(e,1,2)}(\tau;\widetilde\omega)\\
\widetilde M^{(e,2,1)}(\tau;\widetilde\omega) &  \widetilde M^{(e,2,2)}(\tau;\widetilde\omega)      
\end{bmatrix}, \notag
\end{align*}
where $\widetilde M^{(e,l,j)}(\tau;\widetilde\omega)$ (for $l,j \in \{1,2\}$) satisfies
\begin{align*}
& \widetilde M^{(e,1,1)}(\tau;\widetilde\omega)=  \pm \sqrt{-\frac{\kappa} {\rho_1}} 2\rho_1 \tau \widetilde w - \tau \widetilde M^{(1)}_{\mathrm p}(\kappa) \mp \sqrt{-\frac{\kappa} {\rho_1}} i \tau^{3/2} \widetilde M_{\mathrm{p}}^{(2)}(\kappa) + o(\tau^{3/2}), \\
  &\widetilde M^{(e,1,2)}(\tau;\widetilde\omega)= -\tau \widetilde Q^T_0(\kappa) M_{\mathrm p}^{(0)}(\kappa)
\begin{bmatrix}
Q_\perp(\kappa) \, \big |\, Q_{\mathrm{st}}(\kappa)  
\end{bmatrix} +  o(\tau^{3/2}), \\
  &\widetilde M^{(e,2,1)}(\tau;\widetilde\omega)=  -\tau\begin{bmatrix}
Q_\perp(\kappa)  \, \big |\, Q_{\mathrm{st}}(\kappa)  
\end{bmatrix}^T M_{\mathrm p}^{(0)}(\kappa)
\widetilde Q_0(\kappa)  +  o(\tau^{3/2}),\\
&\widetilde M^{(e,2,2)}(\tau;\widetilde\omega)= 
\begin{bmatrix}
\mp i\sqrt{-\frac{\kappa}{\rho_1}}Q^T_\perp(\kappa) M^{(2)}(0)Q_\perp(\kappa)\sqrt{\tau} + O(\tau) &  O(\sqrt \tau)\\
O(\sqrt{\tau}) & \kappa I + D(\kappa) + O(\sqrt \tau)  
\end{bmatrix}, \notag
\end{align*}
as $\tau \rightarrow 0$. Here, $D(\kappa)$ is defined by \eqref{eq:212},  $M^{(0)}_{\mathrm p}(\kappa)$ is defined by \eqref{eq:213}, and  $\widetilde M^{(1)}_{\mathrm{p}}(\kappa)$ and $\widetilde M^{(2)}_{\mathrm p}(\kappa)$ are defined as in \eqref{eq:203} and \eqref{eq:204}, with $Q_0(\kappa)$ replaced by $\widetilde Q_0(\kappa)$.
Arguing as in the derivation of \eqref{eq:202} and \eqref{eq:129}, we readily obtain that $\widetilde M^{(e,2,2)}$ is invertible and admits an asymptotic expansion of the same type as in \eqref{eq:202}, and that the Schur complement of $\widetilde M^e$, denoted by $\widetilde M^e_{S}$, satisfies
\begin{align*}
    \widetilde M^e_S(\tau;\kappa) &= \pm \sqrt{-\frac{\kappa} {\rho_1}} 2\rho_1 
    \tau\widetilde w - \tau \widetilde M^{(1)}_{\mathrm{p}}(\kappa) \mp i\tau^{3/2}\sqrt{-\frac{\kappa} {\rho_1}}
\sum^3_{l=2}\widetilde M_{\mathrm{p}}^{(l)}(\kappa) + o(\tau^{3/2}), \quad \mathrm{as}\; \tau \rightarrow 0.
\end{align*}
Here, $\widetilde M^{(3)}_{\mathrm p}(\kappa)$ is defined as in  \eqref{eq:205}, with $Q_0(\kappa)$ replaced by $\widetilde Q_0(\kappa)$. By the choice of $\widetilde Q(\kappa)$, we find that 
\begin{align*}
&\widetilde M^{(1)}_{\mathrm{p}}(\kappa) = D_{\mathrm{p}}^{(1)}(\kappa),
\end{align*}
where $D_{\mathrm{p}}^{(1)}(\kappa)$ is a diagonal matrix whose diagonal entries are the element of $\mathcal L^{(2)}(\kappa;0)$. We note that the leading principal submatrix of $\widetilde M^{(1)}_{\mathrm{p}}(\kappa)$ of size $\mathrm{dim}(E_{\mathrm p}(\kappa;0;\kappa''))$ is $\kappa''$ times the identity matrix.
Furthermore, the leading principal submatrix of $\widetilde M_{\mathrm{p}}^{(2)}(\kappa) + \widetilde M_{\mathrm{p}}^{(3)}(\kappa)$ of size $\mathrm{dim}(E_{\mathrm p}(\kappa;0;\kappa''))$ is also
a diagonal matrix whose diagonal entries are the element of $\mathcal L^{(2)}(\kappa;0;\kappa'')$, and is denoted by $D_{\mathrm{p}}^{(2)}(\kappa;0;\kappa'')$. Therefore, based on the above discussions, proceeding as in the derivation of \eqref{eq:134}, we obtain that \eqref{eq:121} with the remainder term $\mathrm{Res}^{(2)}$ satisfying \eqref{eq:242}.
\end{proof}

\begin{proof}[Proof of Lemma \ref{le:11}]
Let $\vep, \tau>0$ be both sufficiently small, and let $C$ is a positive constant depending only on $\Omega$ throughout the proof throughout the proof.

\eqref{g1} The proof of this assertion consists of two parts: the first part involves the derivation of \eqref{eq:160}, and the second part focuses on the proof of \eqref{eq:234}.

\textbf{Part I}: 
With the aid of 
\begin{align} \label{eq:189}
G^{(0)}(y_0+\vep^{-1}(x-y_0),y; \sqrt \tau \widetilde \omega) = \vep G^{(0)}(x-y_0, \vep(y-y_0);\sqrt \tau  \widetilde \omega/\vep),
\end{align}
It is sufficient to estimate
\begin{align*}
\bigg[\int_{U_1(\vep)} + \int_{U_2(\vep)} \bigg]\int_{\Omega} \vep G^{(0)}(x-y_0, \vep(y-y_0);\sqrt \tau  \widetilde \omega/\vep) \psi_1(y) f(x) dy dx =: I_1(\vep,\tau) + I_2(\vep,\tau)
\end{align*}
for $f\in \mathbf L^2_{\beta}(\R^3)\; \mathrm{and}\; \psi_1 \in \mathbf L^2(\Omega)$. Here, $U_1(\vep)$ and $U_2(\vep)$ are specified by \eqref{eq:180} and \eqref{eq:176}, respectively.

In the sequel, we estimate $I_1(\vep,\tau)$ and $I_2(\vep,\tau)$ in turn. By a straightforward calculation, we have 
\begin{align*}
\partial_k \partial_l\frac{e^{i\frac{\omega}{c_{s,0}}|x-y|}}{|x-y|}
= e^{i\frac{\omega}{c_{s,0}}|x-y|}&\bigg[-\frac{\omega^2}{c^{2}_{s,0}} \frac{(x_l-y_l)(x_k-y_k)}{|x-y|^3} + \frac{i\omega}{c_{s,0}} \bigg(\frac{\delta_{kl}}{|x-y|^2} - 3\frac{(x_l-y_l)(x_k-y_k)}{|x-y|^4}\bigg)\\
&+\frac{3(x_l-y_l)(x_k-y_k)}{|x-y|^5} - \frac{\delta_{kl}}{|x-y|^3}\bigg], \quad k,l\in\{1,2,3\}.
\end{align*}
From this, using the Taylor expansions
\begin{align*}
& e^{iz} = 1 - \int^z_{0} e^{i\xi} d\xi \;\;\; \mathrm{and}\;\;\; e^{iz} = 1 + i z +\int^z_{0} -e^{i\xi}(z-\xi)d\xi, \quad z\in \R,
\end{align*}
we readily obtain that for $l,k\in\{1,2,3\}$,
\begin{align}
&\frac{1}{4\pi\rho_0\omega^2}\partial_k \partial_l\left[\frac{e^{i\frac{\omega}{c_{s,0}}|x-y|}}{|x-y|} - \frac{e^{i\frac{\omega}{c_{p,0}}|x-y|}}{|x-y|}\right] \notag \\
&= \delta_{kl} h_1(x,y;\omega) + \delta_{kl}h_2(x,y;\omega) + (x_l-y_l)(x_k-y_k) \left[\sum^5_{p=3}h_p(x,y;\omega)\right].\label{eq:164}
\end{align}
Here, 
\begin{align*}
&h_1(x,y;\omega) := - \frac{1}{4\pi\rho_0\omega^2}\frac{\frac{i\omega}{c_{s,0}}\int^{\frac{\omega|x-y|}{c_{s,0}}}_{0} e^{i\xi}d\xi - \frac{i\omega}{c_{p,0}} \int^{\frac{\omega|x-y|}{c_{p,0}}}_{0} e^{i\xi}d\xi }{|x-y|^2},\\
&h_2(x,y;\omega) := \frac{1}{4\pi\rho_0\omega^2}\frac{\int^{\frac{\omega|x-y|}{c_{s,0}}}_{0} e^{i\xi}\left(\frac{\omega|x-y|}{c_{s,0}}-\xi\right)d\xi - \int^{\frac{\omega|x-y|}{c_{s,0}}}_{0} e^{i\xi}\left(\frac{\omega|x-y|}{c_{s,0}}-\xi\right)d\xi} {|x-y|^3},\\
&h_3(x,y;\omega) := \frac{1}{4\pi\rho_0}\left[\frac{e^{i\frac{\omega}{c_{s,0}}|x-y|}}{c^{2}_{p,0}} - \frac{e^{i\frac{\omega}{c_{s,0}}|x-y|}}{c^{2}_{p,0}} \right]\frac{1}{|x-y|^3},\\
&h_4(x,y;\omega) := \frac{3}{4\pi\rho_0\omega^2}\frac{ \frac{i\omega}{c_{s,0}}\int^{\frac{\omega|x-y|}{c_{s,0}}}_{0} e^{i\xi}d\xi - \frac{i\omega}{c_{p,0}}\int^{\frac{\omega|x-y|}{c_{s,0}}}_{0} e^{i\xi}d\xi} {|x-y|^4},
\end{align*}
and
\begin{align*}
&h_5(x,y;\omega) :=- \frac{3}{4\pi\rho_0\omega^2}\frac{\int^{\frac{\omega|x-y|}{c_{s,0}}}_{0} e^{i\xi}\left(\frac{\omega|x-y|}{c_{s,0}}-\xi\right)d\xi - \int^{\frac{\omega|x-y|}{c_{s,0}}}_{0} e^{i\xi}\left(\frac{\omega|x-y|}{c_{s,0}}-\xi\right)d\xi} {|x-y|^5}.
\end{align*}
Furthermore, with the help of \eqref{eq:142}, \eqref{eq:177} and \eqref{eq:169}, we find that 
\begin{align} \label{eq:190}
&\frac{\sqrt \tau}{\vep}\left||x-y_0 + \vep(y-y_0)| - |x-y_0|\right| \le C\sqrt \tau, 
\end{align}
uniformly for $\vep$, $\tau$, $x\in U_1(\vep)$ and $y\in \Omega$. 
Moreover, by a straightforward calculation, we obtain
\begin{align} \label{eq:192}
&\left|\int^{z+\sqrt{\tau}}_{0} -e^{i\xi}d\xi + \int^{z}_0 e^{i\xi}  d\xi\right| \le \sqrt\tau,
\end{align}
and that
\begin{align} 
\left|\int^{z+\sqrt{\tau}}_{0} -e^{i\xi}(z+\sqrt{\tau}-\xi)d\xi + \int^{z}_0 e^{i\xi}  (z-\xi) d\xi\right| &= \left|\sqrt\tau \int_0^{z+\sqrt{\tau}} e^{i\xi}d\xi \right|+ \left|\int^{z+\sqrt\tau}_{z} e^{i\xi}(z-\xi)d\xi\right| \notag \\
& \le 2\sqrt\tau (z+\sqrt\tau)+ 2\sqrt \tau z. \label{eq:193}
\end{align}
Using \eqref{eq:167}, \eqref{eq:190} and \eqref{eq:192}, we have that
\begin{align*}
&h_1(x-y_0,\vep(y-y_0);\sqrt\tau  \widetilde \omega/\vep) = h_1(x,y_0;\sqrt\tau  \widetilde \omega/\vep) +O\left(\frac{\vep}{|x-y_0|^2}\right),\\
&\left((x-y_0)_l-\vep(y-y_0)_l\right) \left((x-y_0)_k-\vep(y-y_0)_k\right)h_4(x-y_0,\vep(y-y_0);\sqrt\tau \widetilde \omega/\vep) \\
& = (x-y_0)_l (x-y_0)_k h_4(x,y_0;\sqrt\tau \widetilde \omega/\vep) +O\left(\frac{\vep}{|x-y_0|^2}\right),
\end{align*}
uniformly for $\vep$, $\tau$, $x\in U_1(\vep)$ and $y\in \Omega$. 
Similarly, it follows from \eqref{eq:167}, \eqref{eq:190} and \eqref{eq:193} that 
\begin{align*}
&h_2(x-y_0,\vep(y-y_0);\sqrt\tau  \widetilde \omega/\vep) = h_2(x,y_0;\sqrt\tau  \widetilde \omega/\vep) +O\left(\frac{\vep}{|x-y_0|^2}\right),\\
&\left((x-y_0)_l-\vep(y-y_0)_l\right) \left((x-y_0)_k-\vep(y-y_0)_k\right)h_5(x-y_0,\vep(y-y_0);\sqrt\tau  \widetilde \omega/\vep) \\
&=  (x-y_0)_l (x-y_0)_k h_5(x,y_0;\sqrt\tau  \widetilde \omega/\vep) +O\left(\frac{\vep}{|x-y_0|^2}\right),
\end{align*}
uniformly for $\vep$, $\tau$, $x\in U_1(\vep)$ and $y\in \Omega$. 
Applying \eqref{eq:142}, \eqref{eq:165} and \eqref{eq:190}, we have that
\begin{align*}
& \frac{e^{i\frac{\sqrt{\tau} \widetilde \omega}{\vep c_{\sigma,0} }|x-y_0 + \vep(y-y_0)|}}{|x-y_0 + \vep(y-y_0)|} = \frac{e^{i\frac{\sqrt{\tau} \widetilde \omega}{\vep c_{\sigma,0}}|x-y_0|}}{|x-y_0|} + O\left(\frac{\sqrt\tau }{|x-y_0|} + \frac{\vep}{|x-y_0|^2}\right), \;\; \sigma \in \{p,s\},\\
&\left((x-y_0)_l-\vep(y-y_0)_l\right) \left((x-y_0)_k-\vep(y-y_0)_k\right)h_3(x-y_0,\vep(y-y_0);\sqrt\tau  \widetilde \omega/\vep)\\
&= (x-y_0)_l (x-y_0)_k h_3(x,y_0;\sqrt\tau  \widetilde \omega/\vep) + O\left(\frac{\sqrt\tau}{|x-y_0|} + \frac{\vep}{|x-y_0|^2}\right),
\end{align*}
uniformly for $\vep$, $\tau$, $x\in U_1(\vep)$ and $y\in \Omega$. Building upon the above established estimates associated with $h_j$ (for $j\in \{1,2,3,4,5\}$), we obtain 
\begin{align} \label{eq:172}
G^{(0)}(x-y_0, \vep(y-y_0);\sqrt \tau  \widetilde \omega) =  G^{(0)}(x,y_0;\sqrt \tau  \widetilde \omega/\vep) + O\left(\frac{\sqrt\tau}{|x-y_0|} + \frac{\vep}{|x-y_0|^2}\right)
\end{align}
uniformly for $\vep$, $\tau$, $x\in U_1(\vep)$ and $y\in \Omega$. This, together with \eqref{eq:181} and \eqref{eq:189}, we have 
\begin{align}
&\left|I_1(\vep,\tau) - \int_{U_1(\vep)} \int_{\Omega}  \vep G^{(0)}(x, y_0;\sqrt \tau  \widetilde \omega/\vep) \psi_1 (y) f(x) dy dx\right| \notag \\
&\le C\max (\vep\sqrt \tau, \vep^{3/2})\|\psi_1\|_{\mathbf L^2(\Omega)} \|f\|_{L^2_\beta(\R^3)}. \notag
\end{align}


For the estimate of $I_2(\vep,\tau)$, by a change of variable $t = \vep^{-1}(x-y_0) + y_0$, we have 
\begin{align*}
I_2(\vep,\tau) = \vep^3\int_{\widetilde U_2(\vep)} \int_\Omega G^{(0)}(t-y_0, y-y_0;\sqrt \tau \widetilde w) \psi_1(y) \cdot f(y_0+\vep(t-y_0)) dy dt.
\end{align*}
Here, $\widetilde U_2(\vep)$ is specified by \eqref{eq:168}. Proceeding as in the derivation of \eqref{eq:179}, we have 
\begin{align}
|I_2(\vep,\tau)| \le C{\vep^{3/2}}\|\psi_1\|_{\mathbf L^2(\Omega)} \|f\|_{L^2_\beta(\R^3)}. \notag 
\end{align}

Moreover, in conjunction with \eqref{eq:174} and \eqref{eq:164}, we obtain 
\begin{align*}
    &\left|\int_{{U_2(\vep)}} \int_{\Omega} \vep G^{(0)}(x, y_0;\sqrt \tau \vep^{-1}) \psi_1(y) \cdot f(x)  d y  d x\right| 
\le C \vep^{3/2} \|\psi_1\|_{\mathbf L^2(\Omega)} \|f\|_{L^2_\beta(\R^3)}.
\end{align*}
Combining this with the estimates of $I_1(\vep, \tau)$ and $I_2(\vep,\tau)$ gives \eqref{eq:160}.

\textbf{Part II}: With the aid of \eqref{eq:189}, we have 
\begin{align*}
\partial_{y_j}G^{(0)}(y_0+\vep^{-1}(x-y_0),y; \sqrt \tau  \widetilde \omega) = \vep \partial_{y_j}G^{(0)}(x-y_0, \vep(y-y_0);\sqrt \tau  \widetilde \omega/\vep),
\end{align*}
Arguing as in the derivation of \eqref{eq:172}, for $j,k,l\in\{1,2,3\}$, we have 
\begin{align*}
\vep \partial_{x_j} \left(G^{(0)}\right)_{kl} = O\left(\frac{\vep\sqrt\tau}{|x-y_0|} + \frac{\vep^2}{|x-y_0|^2}\right),
\end{align*}
uniformly for $\vep$, $\tau$, $x\in U_1(\vep)$ and $y\in \Omega$. Furthermore, we note that \eqref{eq:184} also holds with $\widetilde w $ replaced by $\sqrt \tau \widetilde w$. Therefore, by repeating the argument leading to \eqref{eq:161}, we obtain \eqref{eq:234}.

\eqref{g2} Similarly as in the derivation of \eqref{eq:160}, we obtain \eqref{eq:161}.

In the sequel, we prove \eqref{eq:235}. With the aid of \eqref{eq:185}, it suffices to investigate 
\begin{align}
\int_{\R^3\backslash B_r(y_0)}\int_{\Gamma} \vep G^{(0)}(x-y_0, \vep(y-y_0);\sqrt \tau  \widetilde \omega/\vep) \psi_2(y) f(x) d S(y)dx \notag 
\end{align}
for $f\in \mathbf L^2_{\beta}(\R^3\backslash B_r(y_0))\; \mathrm{and}\; \psi_2 \in \mathbf H^{-1/2}(\Gamma)$. Using \eqref{eq:142}, we find that for $x \in \R^3 \backslash B_r(y_0)$,
\begin{align*}
&G^{(0)}(x-y_0, \vep(y-y_0);\sqrt \tau  \widetilde \omega) =  G^{(0)}(x,y_0;\sqrt \tau  \widetilde \omega/\vep) \\
& + \vep \nabla G^{(0)}(x,y_0;\sqrt \tau  \widetilde \omega/\vep) (y-y_0)
+ O\left(\frac{\tau}{|x-y_0|} + \frac{\sqrt\tau\vep}{|x-y_0|^2} + \frac{\vep^2}{|x-y_0|^3}\right).
\end{align*}
This, together with \eqref{eq:192} gives \eqref{eq:235}.

\eqref{g3} We note that \eqref{eq:185} holds with $\widetilde w$ replaced by $\sqrt \tau \widetilde w$. Therefore, building upon the assertion \eqref{g1} and \eqref{g2} of this lemma, and using similar arguments that were used in the proof of statement \eqref{z3} of Lemma \ref{le:10}, we readily obtain \eqref{eq:224}.

\end{proof}

\end{appendices}

\section*{Acknowledgment}

This work is supported by the Austrian Science Fund (FWF) grant P: 36942.

\end{document}